\documentclass[reqno]{amsart}
\usepackage{amsmath}
\usepackage{amsthm}
\usepackage{amsfonts}
\usepackage{amssymb}
\usepackage{url}

\usepackage{verbatim}
\usepackage{subfigure}

\newtheorem{theorem}{Theorem}[section]
\newtheorem{proposition}[theorem]{Proposition}
\newtheorem{lemma}[theorem]{Lemma}
\newtheorem{corollary}[theorem]{Corollary}

\newtheorem{definition}[theorem]{Definition}
\newtheorem{remark}[theorem]{Remark}

\theoremstyle{definition}
\numberwithin{equation}{section}  

\usepackage{graphicx}

\usepackage[all]{xy}
\numberwithin{equation}{section}
\def\p{\partial}
\def\n{\nonumber}
\def\d{\displaystyle}
\def\t{\tilde}
\def\h{\dot H^1_0(\Omega)}

\def\N{\mathbb N}

\def\n{\nonumber}
\def\d{\displaystyle}

\def\bG{ \bar G }

\def\rhoi{ \frac{1}{\rho_0} }

\def\bee{\bar\eta}
\def\bje{\bar J}
\def\bve{\bar v}
\def\bAe{\bar A}
\def\be{\bar\eta}

\def\bxi{\bar\Xi}
\def\ve{v^\epsilon}
\def\pbv{ T\, P( \|\bar v\|^2_{\boldsymbol{Z}_T})}
\def\pbvs{ \sqrt{T}\, P( \|\bar v\|^2_{\boldsymbol{Z}_T})}
\def\pcurl{ P( \| \operatorname{curl} \bar v\|^2_{\boldsymbol{Y}_T})}

\def\XT{{\boldsymbol{X}_T}}
\def\YT{{\boldsymbol{Y}_T}}
\def\ZT{{\boldsymbol{Z}_T}}

\DeclareMathOperator{\spt}{spt}

\title[Free-boundary 3-D compressible Euler equations  in physical vacuum]
{Well-posedness in smooth function spaces for the moving-boundary 
3-D compressible Euler equations in physical vacuum}

\author[D. Coutand]{Daniel Coutand}
\author[S. Shkoller]{Steve Shkoller}
\address{canpde, Maxwell Institute for Mathematical Sciences and department of Mathematics, Heriot-Watt University, Edinburgh, EH14 4AS, UK}
\address{Department of Mathematics, University of California, Davis, CA 95616, USA}

\subjclass{35L65, 35L70, 35L80, 35Q35, 35R35, 76B03}
\keywords{compressible Euler equations,  gas dynamics, free boundary problems, physical vacuum, characteristic hyperbolic systems,
degenerate hyperbolic systems, systems of conservation laws}

\date{March 29, 2010}

\email{d.coutand@hw.ac.uk}
\email{shkoller@math.ucdavis.edu}

\begin{document}

\begin{abstract}
We prove well-posedness  for the 3-D compressible Euler equations with moving {\it physical} vacuum boundary,
with an equation of state given by $p(\rho) = C_\gamma \rho^\gamma $ for $\gamma >1$.
The physical vacuum singularity requires the sound speed $c$ to go to zero as the square-root of the distance to the moving
boundary, and thus creates a degenerate and characteristic hyperbolic {\it free-boundary} system wherein the density vanishes
on the free-boundary, the uniform Kreiss--Lopatinskii condition is violated, and  manifest derivative loss ensues.  
Nevertheless, we are able to  establish the existence of unique solutions to this system on a short time-interval, which are smooth (in Sobolev
spaces) all the way to the moving boundary, and our estimates have no derivative loss with respect to initial data.  Our proof is founded on an
 approximation of the Euler equations
by a  degenerate parabolic regularization  obtained from a specific choice of  a degenerate artificial
viscosity term, chosen to preserve as much of the geometric structure of the Euler equations as possible.    We first construct solutions to this degenerate parabolic
regularization using a new higher-order Hardy-type inequality; we then establish estimates for solutions to this degenerate parabolic
system which are independent of the artificial viscosity parameter.  Solutions to the compressible Euler equations are found in the
limit as the artificial viscosity tends to zero.  Our regular solutions can be viewed as {\it degenerate viscosity solutions}.   Out methodology
can be applied to many other systems of degenerate and characteristic hyperbolic systems of conservation laws.
\end{abstract}

\maketitle
{\small 
\tableofcontents}

\section{Introduction}
\label{sec_introduction}

\subsection{The compressible Euler equations in Eulerian variables}
For $0 \le t \le T$,
the evolution of a three-dimensional  compressible gas  moving inside of a dynamic vacuum  boundary is modeled by the one-phase compressible Euler equations:
\begin{subequations}
  \label{ceuler}
\begin{alignat}{2}
\rho[u_t+ u\cdot Du] + D p(\rho)&=0  &&\text{in} \ \ \Omega(t) \,, \label{ceuler.a}\\
 \rho_t+ {\operatorname{div}}  (\rho u) &=0 
&&\text{in} \ \ \Omega(t) \,, \label{ceuler.b}\\
p &= 0 \ \ &&\text{on} \ \ \Gamma(t) \,, \label{ceuler.c}\\
\mathcal{V} (\Gamma(t))& = u \cdot n(t) &&\ \ \label{ceuler.d}\\
(\rho,u)   &= (\rho_0,u_0 ) \ \  &&\text{on} \ \ \Omega(0) \,, \label{ceuler.e}\\
   \Omega(0) &= \Omega\,.  && \label{ceuler.f}
\end{alignat}
\end{subequations}
The open, bounded subset
 $\Omega(t) \subset \mathbb{R}^3  $ denotes the changing volume occupied by the gas,  $\Gamma(t):= \partial\Omega(t)$ denotes
 the moving vacuum boundary, $ \mathcal{V} (\Gamma(t))$ denotes the normal
 velocity of $\Gamma(t)$, and $n(t)$ denotes the exterior unit normal vector to $\Gamma(t)$.
  The vector-field $u = (u_1,u_2,u_3)$ denotes the Eulerian velocity
field, $p$ denotes the pressure function, and $\rho$ denotes the density of the gas.
The equation of state $p(\rho)$ is given by
\begin{equation}\label{eos}
p(x,t)= C_\gamma\, \rho(x,t)^\gamma\ \  \text{ for } \ \  \gamma> 1 ,
\end{equation} 
where $C_\gamma $ is the adiabatic constant which we set to unity, and
$$
\rho>0 \ \text{ in } \ \Omega(t) \ \ \ \text{ and } \ \ \ \rho=0 \ \text{ on } \Gamma(t) \,.
$$
Equation (\ref{ceuler.a}) is the conservation of momentum; (\ref{ceuler.b}) is the conservation of mass; the boundary
condition (\ref{ceuler.c}) states that the pressure (and hence the density function)  vanish along the moving vacuum boundary
$\Gamma(t)$; (\ref{ceuler.d}) states that the 
vacuum boundary $\Gamma(t)$ is moving with speed equal to the normal component of the fluid velocity, and (\ref{ceuler.e})-(\ref{ceuler.f}) are the 
initial conditions for the density, velocity, and domain.  Using the equation of state (\ref{eos}),
 (\ref{ceuler.a}) is written as
\begin{alignat}{2}
\rho [u_t+ u\cdot Du]+  D\rho^{\gamma} &=0 \ \ \   &&\text{in} \ \ \Omega(t) \,. \tag{ \ref{ceuler.a}'} 
\end{alignat}

\subsection{Physical vacuum}  \label{subsec_physicalvacuum}
With the sound speed given by $c := \sqrt{\p p/ \p \rho}$ and $N$ denoting the {\it outward} unit normal to the initial surface $\Gamma$, satisfaction of the condition
\begin{equation}\label{phys_vac}
\frac{ \partial c_0^2}{ \partial N} < 0 \text{ on } \Gamma
\end{equation} 
defines a {\it physical vacuum} boundary (see \cite{Lin1987}, \cite{Liu1996}, \cite{LiYa1997}, \cite{LiYa2000},
\cite{LiSm1980}, \cite{XuYa2005}), where $c_0 = c|_{t=0}$ denotes the initial sound speed of the gas.

 The physical vacuum condition (\ref{phys_vac}) is equivalent to the
requirement that
\begin{equation}\label{stab1}
\frac{ \p \rho_0^{\gamma-1}} { \partial N} < 0 \text{ on } \Gamma \,,
\end{equation}
a condition necessary for the gas particles on the boundary to accelerate.
Since $ \rho_0 >0$ in $\Omega$, (\ref{stab1}) implies that for some positive constant $C$ and $x\in \Omega$ near the vacuum boundary $\Gamma$,
\begin{equation}\label{degen}
\rho_0^{\gamma-1}(x) \ge  C \text{dist}(x, \Gamma) \,.
\end{equation} 

Because of condition (\ref{degen}), the compressible Euler system (\ref{ceuler}) is a {\it degenerate} and  {\it characteristic} hyperbolic
system  which violates the uniform Kreiss--Lopatinskii condition \cite{Kr1970}, and to which standard methods of symmetric hyperbolic conservation laws are apparently ineffective.   The moving boundary is characteristic because of the evolution law (\ref{ceuler.d}), and the system of conservation laws is degenerate because of the appearance of the density function as a coefficient in the nonlinear wave equation which governs the dynamics of the divergence of the velocity of the gas;  in turn, weighted estimates show that this wave equation indeed loses
derivatives with respect to the uniformly hyperbolic non-degenerate case of a compressible liquid, wherein the density takes
the value of a strictly positive constant on the moving boundary \cite{CoFr48}.     We provide a brief history of results in this area in Section
\ref{history} below.

We note that with 
 a faster rate of degeneracy of the density function, such as, for example,  $ \operatorname{dist}(x, \Gamma(t))^b$ 
for $b=2,3,....$, the analysis becomes significantly easier; for instance, if  $b=2$,
then $\frac{D\rho_0^{\gamma-1}(x,t)}{\sqrt{ \rho_0^{\gamma-1}(x,t)}}$ is bounded for all $x \in \Omega$. This bound makes it possible to
readily  control error terms in energy estimates, and in effect removes the singular behavior associated with the physical vacuum condition
(\ref{degen}).   On the other hand, if $\rho^{\gamma -1}$ tends to zero like  $ \operatorname{dist}(x, \Gamma(t))^b$ 
for $b=2,3,....$, then the gas cannot accelerate into vacuum.

\subsection{Fixing the domain and the Lagrangian variables on $\Omega$}   We transform the system (\ref{ceuler}) into
Lagrangian variables.
We let $\eta(x,t)$ denote the ``position'' of the gas particle $x$ at time $t$.  Thus,
\begin{equation}
\nonumber
\begin{array}{c}
\partial_t \eta = u \circ \eta $ for $ t>0 $ and $
\eta(x,0)=x
\end{array}
\end{equation}
where $\circ $ denotes composition so that
$[u \circ \eta] (x,t):= u(\eta(x,t),t)\,.$  
We set 
\begin{align*}
v &= u \circ \eta   \text{ (Lagrangian velocity)},  \\
f&= \rho \circ \eta   \text{ (Lagrangian density)}, \\
A &= [D \eta]^{-1}  \text{ (inverse of deformation tensor)}, \\
J &= \det D \eta  \text{ (Jacobian determinant)}, \\
a &= J\, A  \text{ (transpose of cofactor matrix)}. 
\end{align*}
Using  Einstein's summation convention defined in Section \ref{subsec_einstein} below, and using
the notation $F,_k$ to denote $\frac{\p F}{ \p x_k}$, the k$th$-partial derivative of $F$ for $k=1,2,3$,
the Lagrangian version of equations (\ref{ceuler.a})-(\ref{ceuler.b}) can be written on
the fixed reference domain $\Omega$ as
\begin{subequations}
\label{ceuler00}
\begin{alignat}{2}
f  v^i_t + A^k_i f^ \gamma ,_k &=0 \ \ && \text{ in } \Omega \times (0,T] \,, \label{ceuler00.a} \\
f _t + f A^j_i v^i,_j  &=0 \ \ && \text{ in } \Omega \times (0,T] \,,\label{ceuler00.b}  \\
f  &=0 \ \ && \text{ in } \Omega \times (0,T] \,,\label{ceuler00.c}  \\
(f,v,\eta)  &=(\rho_0, u_0, e) \ \  \ \ && \text{ in } \Omega \times \{t=0\} \,, \label{ceuler00.d} 
\end{alignat}
\end{subequations}
where $e(x)=x$ denotes the identity map on $\Omega$.

Since $J_t = J A^j_i v^i,_j$ and since $J(0)=1$ (since we have taken $\eta(x,0)=x$), it follows that
\begin{equation}\label{J}
f = \rho_0 J^{-1},
\end{equation}
so that the initial density function $\rho_0$ can be viewed as a parameter in the Euler equations.   Let $\Gamma:= \partial \Omega$ denote
the initial vacuum boundary;  using, that $A^k_i = J ^{-1} \, a^k_i$, we write the compressible Euler equations (\ref{ceuler00}) as
\begin{subequations}
\label{ceuler0}
\begin{alignat}{2}
\rho_0 v^i_t + a^k_i (\rho_0^ \gamma J^{-\gamma}),_k &=0 \ \ && \text{ in } \Omega \times (0,T] \,, \label{ceuler0.a} \\
(\eta, v)  &=( e, u_0) \ \  \ \ && \text{ in } \Omega \times \{t=0\} \,, \label{ceuler0.b} \\
\rho_0^{\gamma-1}& = 0 \ \ &&\text{ on }  \Gamma \,, \label{ceuler0.c}  
\end{alignat}
\end{subequations}
with $ \rho _0^{ \gamma -1}(x) \ge C \operatorname{dist}( x, \Gamma ) $ for $x \in\Omega$ near $\Gamma$.

\subsection{Setting $\gamma=2$} We will begin our analysis for the case that $\gamma=2$, and in Section \ref{sec::generalgamma},
we will explain the modifications required for the case of general $\gamma >1$.

With $\gamma$ set to $2$, we thus seek solutions $\eta(t)$ to the following system:
\begin{subequations}
  \label{ce0}
\begin{alignat}{2}
\rho_0  v_t^i + a^k_i (\rho_0^2 J^{-2}),_k&=0  &&\text{in} \ \ \Omega \times (0,T] 
\,, \label{ce0.a}\\
(\eta,v)&= (e,u_0 ) \ \ \ &&\text{on} \ \ \Omega \times \{t=0\} \,, \label{ce0.b}\\
\rho_0& = 0 \ \ &&\text{ on }  \Gamma \,, \label{ce0.c}
\end{alignat}
\end{subequations}
with $ \rho _0(x) \ge C \operatorname{dist}( x, \Gamma ) $ for $x \in\Omega$ near $\Gamma$.

The equation
(\ref{ce0.a}) is equivalent to 
\begin{equation}
v_t^i + 2A^k_i (\rho_0 J^{-1} ),_k =0 \label{ce_vor} \,,
\end{equation} 
and (\ref{ce_vor}) can be written as
 
\begin{equation}
v_t^i +\rho_0 a^k_i J^{-2},_k +2 \rho_0,_k a^k_i J^{-2}  =0 \label{ce_elliptic} \,.
\end{equation} 
Because of the degeneracy caused by $\rho_0 =0$ on $\Gamma$, all three equivalent forms of the compressible
Euler equations are crucially used in our analysis.  The equation (\ref{ce0.a}) is used for energy estimates, while (\ref{ce_vor}) is
used for estimates of the vorticity, and (\ref{ce_elliptic}) is used for additional elliptic-type estimates used to recover the bounds
for normal derivatives.

\subsection{The reference domain $\Omega$}\label{subsec_domain}
 To avoid the use of local coordinate charts necessary for arbitrary geometries,
and to simplify our exposition, we will assume  that the initial domain $\Omega \subset \mathbb{R}  ^3$ at time $t=0$ is given by
$$
\Omega = \{ (x_1,x_2,x_3) \in \mathbb{R}^3  \ | \ (x_1, x_2) \in  \mathbb{T}  ^2 , \ x_3\in (0,1) \} \,,
$$
where $\mathbb{T}  ^2$ denotes the $2$-torus, which can be thought of as the unit square with periodic
boundary conditions.  This permits the use of {\it one} global Cartesian coordinate system. At $t=0$, the reference {\it vacuum} boundary  is
comprised of the {\it bottom} and
{\it top} of the domain $\Omega$ so that
$$\Gamma =\{ x_3=0\} \cup \{ x_3=1\}\,.$$

Then, according to the evolution law for the moving vacuum boundary $\Gamma(t)$ given by (\ref{ceuler.d}),  we have that
$$
\Gamma(t) = \eta(t)(\Gamma) \,.
$$
(We will sometimes write $\eta(t,\Gamma)$ to denote $\eta(t)(\Gamma)$.)
Hence, solving (\ref{ce0}) for $\eta(t)$ (and $v(t)=\eta_t(t)$) completely determines the motion and regularity of the moving
vacuum boundary $\Gamma(t)$.

 \subsection{The higher-order energy function for the case $\gamma=2$} The physical energy $  \int_\Omega \bigl[{\frac{1}{2}}  \rho_0 |v|^2 +  \rho_0^2 J^{-1} \bigl]dx$ is a conserved quantity, but is far too weak for the
purposes of constructing solutions; instead, we consider the higher-order energy function
\begin{align} 
E(t) & = \sum_{a=0}^4\Bigl[ \| \p_t^{2a}  \eta(t)\|^2_{4-a} + \|\rho_0 \p_t^{2a} D \eta(t)\|^2_{4-a}
+\| \sqrt{\rho_0} \bar\p^{4-a}\p_t^{2a}  v(t)\|^2_0 \Bigr] \nonumber\\
& \qquad\qquad  
 + \| \operatorname{curl} _\eta v(t)\|^2_ 3 
+ \| \rho_0 \bar \p^4 \operatorname{curl} _\eta v(t)\|^2_0 \,, \label{formal}
\end{align} 
where $\bar \p = \left(\frac{\p}{\p x_1}, \frac{\p}{\p x_2}\right)$ and $ \operatorname{curl} _ \eta v = [\operatorname{curl} u] \circ \eta $.  Section \ref{notation} explains the notation.

We also define $M_0 = P( E(0))$, where $P$ denotes a polynomial function of its argument.

While the higher-order energy function $E(t)$  is 
not conserved,  we will construct solutions to (\ref{ce0})  for which $ \sup_{t \in [0,T]} E(t)$ remains bounded whenever $T>0$ is taken sufficiently small;   the bound depends only on $E(0)$.

 \subsection{Main Result}   
\begin{theorem} [Existence and uniqueness for the case $\gamma=2$] \label{theorem_main}
Suppose that $\rho_0 \in H^4(\Omega)$, $\rho_0(x) >0$ for $x \in \Omega$, and  $\rho_0$ satisfies (\ref{degen}).  Furthermore,
suppose that $u_0$ is given such that  $M_0< \infty $.  Then there exists a  solution to (\ref{ce0}) (and hence (\ref{ceuler})) on $[0,T]$ for $T>0$ taken
 sufficiently small, such that
 $$
 \sup_{t \in [0,T]} E(t) \le 2M_0 \,.
 $$
 In particular, the flow map $\eta \in L^ \infty (0,T; H^4(\Omega))$ and the moving vacuum boundary $\Gamma(t)$ is of Sobolev class $H^{3.5}$.
 
 Moreover if the initial data satisfies
 \begin{align}
 &\sum_{a=0}^5\Bigl[ \| \p_t^{2a}  \eta(0)\|^2_{5-a} + \|\rho_0 \p_t^{2a}  D\eta(0)\|^2_{5-a}
+\| \sqrt{\rho_0} \bar\p^{5-a}\p_t^{2a}  v(0)\|^2_0 \Bigr] \nonumber\\
& \qquad\qquad  
 + \| \operatorname{curl} _\eta v(0)\|^2_ 4 
+ \| \rho_0 \bar \p^5 \operatorname{curl} _\eta v(0)\|^2_0   < \infty \,,  \label{uniquedata}
 \end{align} 
 then the solution is unique.
\end{theorem} 
\begin{remark} The case of arbitrary $\gamma >1$ is treated in Theorem \ref{thm_main2} below.
\end{remark}

Theorem \ref{theorem_main}  also covers the 2-D case that $\Omega \subset  \mathbb{R}   ^2$.  We established the analogous result
in 1-D in \cite{CoSh2009}.   We note that by using a collection of local
coordinate charts, we could modify our proof to allow for arbitrary initial domains $\Omega$, as long as the initial boundary is of Sobolev class $H^{3.5}$.   

The multi-D physical vacuum problem is  not only  a characteristic hyperbolic system, but it is also  degenerate because the density function vanishes on 
 the boundary $\Gamma$.
 In 1-D, the two characteristic curves  of the isentropic system intersect with the moving vacuum boundary $\Gamma(t)$ tangentially;
 this triple point of intersection is suggestive of singular behavior.
  While the degeneracy produces ``honest'' derivative loss
 with respect to uniformly hyperbolic systems, we develop a methodology based on {\it nonlinear estimates} which provides us with
 a priori control of smooth solutions which do not suffer from the derivative loss phenomenon (see \cite{CoLiSh2009} for the a priori estimates to this problem).
As we will outline below, our method for constructing smooth solutions does not rely on linearization, Kreiss--Lopatinskii theory, or
the Nash-Moser iteration scheme, but rather on a carefully chosen   nonlinear approximation to the  characteristic and degenerate Euler equations, which preserves a great deal of the nonlinear structure of the original system.

\subsection{History of prior results on the analysis of multi-D free-boundary Euler problems}
\subsubsection{The incompressible setting}
There has been 
a recent explosion of interest in the analysis of the free-boundary {\it incompressible} Euler equations, particularly in irrotational form, that has produced a number of different methodologies for obtaining a priori estimates, and the accompanying existence theories have mostly relied on the Nash-Moser  iteration to deal with derivative loss in linearized equations when arbitrary domains are considered, or complex analysis tools for  the irrotational problem with infinite depth.  We refer the reader to \cite{AmMa2009}, \cite{CoSh2007}, \cite{La2005}, \cite{Li2005}, \cite{Na1974}, \cite{ShZe2008}, \cite{Wu1997}, \cite{Wu1999}, \cite{Yo1982}, and \cite{ZhZh2009} for a partial list of papers on this topic.  

\subsubsection{The compressible setting} The mathematical analysis of moving hypersurfaces in the multi-D compressible Euler equations
is essential for the understanding of shock waves, vortex sheets or contact discontinuities, as well as phase transitions such as the motion
of gas into the vacuum state
considered herein.   

The stability and regularity of the multi-D shock solution was initiated in \cite{Ma1984} and extensively studied by
 \cite{FrMe2000}, \cite{GlMa1991}, \cite{GuMeWiZu2005}, and \cite{Me2001} (see the references in these articles for a more
 extensive bibliography).   The shock wave problem is non-characteristic on the boundary,
 and in fact, produces  the so-called  dissipative boundary conditions, and satisfies the uniform  Kreiss--Lopatinskii condition. Even so
 the methodologies employed produce derivative loss with respect to initial data.
 
   More delicate
 than the non-characteristic case, is the characteristic boundary case, encountered in  the study of vortex sheet or current vortex sheet problems.  This
 class of problems has been studied by \cite{ChWa2007},  \cite{CoSe2008}, \cite{CoSe2009}, \cite{Tr2005}, \cite{Tr2009}
 and others, and has the relative disadvantage of violating the uniform Kreiss--Lopatinskii condition, which produces derivative loss in the linearization,
 similar to that experienced by many authors in the incompressible flow setting (both irrotational flows and flows with vorticity).

\subsection{History of prior results for the compressible Euler equations with vacuum boundary}\label{history}
The physical vacuum free-boundary problem, also described as the physical vacuum singularity, has a rich history, as well as a great
deal of renewed interest (see \cite{IMA2009}).   

Some of the early developments in the theory of vacuum states for  compressible gas dynamics can be found in \cite{LiSm1980} and
\cite{Lin1987}.
We are aware of only a handful of previous theorems pertaining to the existence of solutions to the compressible and {\it undamped} Euler equations\footnote{The parabolic free-boundary viscous Navier-Stokes equations do not experience the same sort of analytical difficulties
as the compressible Euler equations, so we do not focus on the viscous regime in this paper.   We refer the reader to
\cite{Ja2010}, \cite{LuXiYa2000}, \cite{MaOkMa1997}, and \cite{OkMa1993} for the analysis of the corresponding
viscous system.}
 with a moving  vacuum boundary.   In
 \cite{M1986},  compactly supported initial data was considered, and the compressible  Euler equations were treated as a PDE set on  $\mathbb{R}^3  \times (0,T]$.  Unfortunately,  with the methodology of \cite{M1986},
it is not possible to track the location of the vacuum boundary (nor is it necessary); nevertheless, an existence theory was developed in
 this context, by a variable change that permitted the standard theory of symmetric hyperbolic systems to be employed,  but, the constraints on the data were too severe to allow for the evolution of the physical vacuum boundary.
 
Existence and uniqueness for the 3-D compressible Euler equations modeling a {\it liquid} rather than a gas was established in \cite{Li2005b}.
 As discussed in \cite{CoFr48}, 
 for a compressible liquid, the density $\rho\ge \lambda>0$ is assumed to be a strictly positive constant on the moving vacuum boundary $\Gamma(t)$ and $\rho$ is
 thus uniformly bounded from below  by a positive constant.  As such, the compressible liquid provides a uniformly hyperbolic, but characteristic, system.   Lagrangian variables combined with Nash-Moser iteration was used in \cite{Li2005b}  to construct solutions.   More recently,   \cite{Tr2008}
provided an alternative proof for the existence of a compressible liquid, employing a solution strategy based on symmetric hyperbolic systems
combined with Nash-Moser iteration, but as stated in Remark 2.2 of that paper, the $\gamma$-gas law equation-of-state $p=\rho^ \gamma$
cannot be used.

In the presence of damping, and with mild singularity, some existence results of smooth
solutions are available, based on the adaptation of the theory of symmetric hyperbolic systems.
In \cite{LiYa1997}, a local existence theory was developed for the case that $c^ \alpha $ (with $0< \alpha \le 1$) is smooth across $\Gamma$,
using methods that are not applicable to the local existence theory
for the physical vacuum boundary. An existence theory for the small perturbation of a
planar wave was developed in \cite{XuYa2005}.  See also \cite{LiYa2000} and \cite{Yang2006}, for other features of the vacuum state problem.

In the 1-D setting, recently,  \cite{JaMa2009} have established existence and uniqueness using weighted Sobolev norms  for their energy
estimates.  From these weighted norms, the regularity  of the solutions cannot be directly determined.  Letting $d$ denote the distance
function to the boundary $\partial I$, and letting $\| \cdot \|_0 $ denote
the $L^2(\Omega)$-norm, an example of the type
of bound that is proved for their rescaled velocity field $u$ in \cite{JaMa2009} is the following:
\begin{align}
&\| d\, u\|_0^2 + \|d\, u_x\|_0^2 + \|d\, u_{xx} + 2 u_x\|_0^2  + \|d\, u_{xxx} + 2 u_{xx} - 2 d^{-1}\,  u_x\|_0^2 \nonumber \\
&\qquad\qquad \qquad \qquad \qquad\qquad \qquad + \|d\, u_{xxxx} + 4 u_{xxx} - 4 d^{-1}\,  u_{xx}\|_0^2  < \infty  \,.
\label{jama}
\end{align} 
This bound is obtained from their paper by considering the case $ \gamma=3$, $k=1$, and making the assumption that $\phi=\xi$ (using
the variable terminology of their paper) which
is certainly true near the boundary.
The problem with inferring the regularity of $u$ from this bound can already be seen at the level of an $H^1(\Omega)$ estimate.   In particular,
the bound on the norm $\|d\, u_{xx} + 2 u_x\|_0^2$ only implies a bound on $\|d\, u_{xx}\|_0^2$ and $\|u_x\|_0^2$ if the integration by
parts on the cross-term,
$$
4\int_I  d\, u_{xx} \, u_x \, dx =  -2\int_I d_x \, |u_x|^2 \, dx \,,
$$
can be justified, which in turn requires having better regularity for $u_x$ than the a priori bounds provide.   Any methodology  which
seeks regularity in (unweighted) Sobolev spaces for solutions must contend with this type of issue.  

We overcame this difficulty in 1-D \cite{CoSh2009} by 
constructing (sufficiently) smooth solutions to a degenerate parabolic regularization and consequently avoiding this sort of integration-by-parts difficulty.  Our solution strategy in \cite{CoSh2009} was based on a 1-D version of our  higher-order Hardy-type inequality.
In this paper, we extend our ideas to the multi-D setting.

\subsection{Outline of the paper and our methodology}  Section \ref{notation} defines the notation used throughout the paper.  In
Section \ref{sec::Hardy}, we introduce our higher-order Hardy inequality for functions on $\Omega$ that vanish on $\Gamma$; this
inequality is of fundamental importance to our strategy for constructing solutions.  In this section, we also
state a lemma on $ \kappa $-independent estimates for equations $ \kappa f_t + f=g$, which will be of great use to us in the elliptic-type
estimates that we shall employ for bounding normal derivatives.  We end this section with a standard weighted embedding into
standard Sobolev spaces.   Section \ref{sec::lagcurldiv} defines the Lagrangian curl and divergence operators. In Section \ref{sec::properties},
we provide basic differentiation rules for the Jacobian determinant $J$ and cofactor matrix $a$, and state the basic geometric and relevant
analytical properties of the cofactor matrix.  Section \ref{sec::Hodge} provides some well-known elliptic estimates based on the
Hodge decomposition of vector fields, as well as some basic trace estimates for the normal and tangential components of vectors fields
in $L^2(\Omega)$.

In Section \ref{sec::approx}, we introduce the degenerate parabolic approximation (\ref{approx}) to the compressible Euler equations
(\ref{ce0}), which takes the form $ \rho_0 v_t^i + a^k_i ( \rho_0 ^2   J^{-2} ),_k + \kappa \p_t\bigl[a^k_i ( \rho_0^2   J^{-2} ),_k \bigr]=0$,
where $ \kappa >0$ denotes the artificial viscosity parameter, and  with the  special choice of  the degenerate parabolic operator $ \kappa \p_t\bigl[a^k_i ( \rho_0^2   J^{-2} ),_k \bigr]$ which
preserves a majority of the geometric structure of the Euler equations.   In particular,  the structure of the energy estimates for the horizontal space derivatives as well as time derivatives are  essentially preserved, the elliptic-type estimates for vertical (or normal) derivatives are
kept intact, while the estimates for vorticity are not exactly preserved, but can still be obtained with some additional structural observations
employed.

Section \ref{sec_kproblem} is devoted to the construction of solutions to the degenerate parabolic $ \kappa $-problem (\ref{approx}) on
a time interval $[0,T_ \kappa ]$, where $T_ \kappa $ may a priori approach zero as $ \kappa \to 0$.   The 1-D version of the parabolic
$ \kappa$-problem has been studied by us in \cite{CoSh2009} and also in \cite{EpMa2009} in the context of  the Wright-Fisher diffusion
arising in mathematical biology.

The construction of solutions in the 3-D setting is significantly more challenging.  Our approach is to (1) compute the Lagrangian divergence
of the $ \kappa $-problem to find a nonlinear degenerate parabolic equation for   $ \rho_0 \operatorname{div} _\eta v$, (2) compute
the Lagrangian curl of the $ \kappa $-problem to  find the evolution equation for $ \operatorname{curl} _\eta v$, and (3) to consider the
vertical (or normal) component  of the trace of the  $ \kappa $-problem on the boundary $\Gamma$, and find an evolution 
equation for $v^3$.   We then linearize these three
evolution equations, and obtain a solution to the linearized problem via an additional approximation scheme, which requires us to
horizontally smooth the linearized boundary evolution PDE for $v^3$, using convolution operators on $\Gamma$.   We find a fixed-point
to this horizontally smoothed problem using the contraction mapping principle, and then perform energy estimates to find a solution
on a time-interval which is independent of the horizontal convolution parameter.   An additional contraction mapping argument is then
made to find a solution of the nonlinear $ \kappa $-problem.  One of the serious subtleties of our analysis involves the solution and
regularity of the degenerate parabolic equation for the $\rho_0 \operatorname{div}_\eta v$.

In Section \ref{section_mainestimates}, we establish $ \kappa $-independent estimates for the solutions that we have constructed to the $ \kappa $-problem
(\ref{approx}).   This is done by a combination of energy estimates for the horizontal and time-derivatives of $\eta(t)$, which rely
on the determinant structure of the Euler equations in Lagrangian variables, followed by elliptic-type estimates that give bounds
on the vertical derivatives of $\eta(t)$ and its time-derivatives.

Section \ref{sec10} uses these $ \kappa $-independent estimates to construct a solution to the compressible Euler equations as 
a limit of the sequence of parabolic solutions as $ \kappa \to 0$.  Uniqueness is proven as well.

Finally,  in Section \ref{sec::generalgamma}, we describe the modifications which are necessary for the case of general $ \gamma >1$.

The methodology developed for the multi-D compressible Euler equations with physical vacuum singularity is somewhat general,
and can be applied to a host of other degenerate and characteristic hyperbolic systems of conservation laws such as the equations
of magneto-hydrodynamics.   

\subsection{Generalization of the isentropic gas assumption}  The  general form of the compressible Euler equations in three
space dimensions are the $5 \times 5$ system of conservation laws
\begin{subequations}
  \label{claw}
\begin{alignat}{2}
\rho[u_t+ u\cdot  D u] + Dp(\rho)&=0 \,,   \label{claw.a}\\
 \rho_t+ {\operatorname{div}}  (\rho u) &=0 \,, \label{claw.b} \\
  (\rho {\mathfrak E})_t+ {\operatorname{div}}  (\rho u\mathfrak{E} + p u) &=0 \,,  \label{claw.c}
\end{alignat}
\end{subequations}
where (\ref{claw.a}), (\ref{claw.b}) and (\ref{claw.c}) represent the respective conservation of  momentum, mass, and 
total energy.  Here, the quantity $\mathfrak{E}$ is the sum of contributions from the kinetic energy $ {\frac{1}{2}} |u|^2$, and
the internal energy $e$, i.e.,$\mathfrak{E}= {\frac{1}{2}} |u|^2 + e$.  For a single phase of compressible liquid or gas, $e$ becomes a
well-defined function of $\rho$ and $p$ through the theory of thermodynamics, $e = e(\rho, p)$.  Other interesting and useful
physical quantities, the temperature $T(\rho,p)$ and the entropy $S(\rho,p)$ are defined through the following consequence of the
second law of thermodynamics
$$
T\ dS = de  - \frac{p}{ \rho^2} \ d\rho  \,.
$$
For {\it ideal gases}, the quantities $e,T,S$ have the explicit formulae:
\begin{align*} 
e(\rho, p) & = \frac{p}{\rho(\gamma -1)} = \frac{T}{\gamma -1} \\
T(\rho,p) &= \frac{p}{\rho} \\
p &= e^S \rho ^ {\gamma}, \ \ \ \gamma > 1,  \ \text{ constant} \,.
\end{align*} 
In regions of smoothness, one often uses velocity and a convenient choice of two additional variables among the five
quantities $S,T,p,\rho,e$ as independent variables.  For the Lagrangian formulation, the entropy $S$ plays an important role, as it satisfies the
transport equation
\begin{equation}\nonumber
S_t + (u \cdot D )S =0 \,,
\end{equation} 
 and as such, $S \circ \eta = S_0$, where $S_0(x) = S(x,0)$ is the initial entropy function.   Thus, by replacing $f$ with $e^{S \circ \eta} \rho_0^{\gamma} J ^{-\gamma} $, our analysis for the isentropic case naturally generalizes to the $5 \times 5$ system of conservation laws.

\section{Notation and Weighted Spaces}\label{notation}

\subsection{The gradient and the horizontal derivative} The reference domain $\Omega$ is defined in Section 
\ref{subsec_domain}.
 Throughout the paper the symbol $D$  will be used to 
 denote the three-dimensional gradient vector
 $$D=\left(\frac{\p}{\p x_1}, \frac{\p}{\p x_2},\frac{\p}{\p x_3}\right) \,,$$ 
 and let $\bar \p$ denote the horizontal derivative
 $$
 \bar \p = \left(\frac{\p}{\p x_1}, \frac{\p}{\p x_2}\right) \,.
 $$
 
 \subsection{Notation for partial differentiation}
 The $k$th partial derivative of $F$ will be denoted by $F,_k = \frac{ \p F}{ \p x_k}$.

\subsection{The divergence and curl operators}
We use the notation $ \operatorname{div} V$ for the divergence of a vector field $V$ on $ \Omega  $:
$$
\operatorname{div} V = V^1,_1 + V^2,_2 + V^3,_3 \,,
$$
and we use $ \operatorname{curl} V$ to denote the curl
of a vector $V$ on $ \Omega $:
$$
\operatorname{curl} V = \left( V^3,_2 - V^2,_3\,, V^1,_3 - V^3,_1 \,, V^2,_1 - V^1,_2\right) \,.
$$

Throughout the paper, we will make use of the permutation symbol 
\begin{equation}\label{permutation}
\varepsilon_{ijk} = \left\{\begin{array}{rl}
1, & \text{even permutation of } \{ 1, 2, 3\}, \\
-1, & \text{odd permutation of } \{ 1, 2, 3\}, \\
0, & \text{otherwise}\,,
\end{array}\right.
\end{equation} 
This allows us to write the $i$th component of the curl of a vector-field $V$ as
$$
[ \operatorname{curl}  V]_i = \varepsilon_{ijk} V^k,_j  \text{ or equivalently } \operatorname{curl} V = \varepsilon_{\cdot jk} V^k,_j
$$
which agrees with our definition above, but is notationally convenient.

We will also define the Lagrangian divergence and curl operators as follows:
\begin{equation}\label{lagrangian_div}
\operatorname{div} _\eta W = A^j_i W^i,_j \,,
\end{equation} 
and 
\begin{equation}\label{lagrangian_curl}
\operatorname{curl} _\eta V = \varepsilon_{\cdot jk} A^r_j W^k,_r \,.
\end{equation} 
In the sequel we shall also use the notation $ \operatorname{div} _{\bar \eta}$ and $\operatorname{curl} _{\bar \eta}$ to
mean the operations defined by (\ref{lagrangian_div}) and (\ref{lagrangian_curl}), respectively, with $\bar A = [ D\bar \eta] ^{-1} $
replacing $A$,

Finally, we will make use of the 2-D divergence operator $\operatorname{div} _\Gamma$ for vector-fields $F$ on  the 2-D boundary $\Gamma$:
\begin{equation}\label{divGamma}
\operatorname{div} _\Gamma F = F^1,_1 + F^2,_2 \,.
\end{equation} 

\subsection{Sobolev spaces on $\Omega$}

For integers $k\ge 0$ and a smooth, open domain $\Omega$ of $\mathbb{R}  ^3$, 
we define the Sobolev space $H^k(\Omega)$ ($H^k(\Omega; {\mathbb R}^3 )$) to
be the completion of $C^\infty(\Omega)$ ($C^\infty(\Omega; {\mathbb R}^3)$) 
in the norm
$$\|u\|_k := \left( \sum_{|a|\le k}\int_\Omega \left|   D^ a u(x)
\right|^2 dx\right)^{1/2},$$
for a multi-index $a \in {\mathbb Z} ^3_+$, with the standard convention that  $|a|=a_1 +a_2+ a _3$. 
For real numbers $s\ge 0$, the Sobolev spaces $H^s(\Omega)$ and the norms $\| \cdot \|_s$ are defined by interpolation.
We will  write $H^s(\Omega)$ instead of $H^s(\Omega;{\mathbb R} ^3)$ 
for vector-valued functions.  In the case that $s\ge 3$, the above definition also holds for domains
 $\Omega$  of class $H^s$.
 
Our analysis will often make use of the following subspace of $H^1(\Omega)$:
$$
\h = \{ u \in H^1(\Omega) \ : \  u=0 \text{ on } \Gamma \,, ( x_1,x_2) \mapsto u(x_1,x_2) \text{ is periodic } \} \,,
$$
where, as usual, the vanishing of $u$ on $\Gamma$ is understood in the sense of trace.

We will on occasion also refer to  the Banach space $W^{1, \infty }(\Omega)$ consisting of $L^ \infty(\Omega)$ functions whose
 weak derivatives are also in $L^ \infty (\Omega)$.


\subsection{Einstein's summation convention} \label{subsec_einstein} Repeated Latin indices $i,j,k,$, etc., are summed
from $1$ to $3$, and repeated Greek indices $ \alpha , \beta , \gamma $, etc., are summed from $1$ to 
$2$.   For example, $F,_{ii} := \sum_{i=1,3} \frac{\p^2}{\p x_i \p x_i}$, and $F^i,_ \alpha I^{\alpha
\beta}  G^i,_\beta :=
\sum_{i=1}^3 \sum_{\alpha=1}^2 \sum_{\beta=1}^2  \frac{ \partial F^i}{ \partial x_ \alpha } 
I^{\alpha \beta } \frac{ \partial G^i}{ \partial x_\beta} $.

\subsection{Sobolev spaces on $\Gamma$}   For  functions $u\in H^k(\Gamma)$, $k \ge 0$,  we set
$$|u|_k := \left( \sum_{|\alpha |\le k}\int_\Omega \left|  \bar \p^ \alpha  u(x)
\right|^2 dx\right)^{1/2},$$
for a multi-index $\alpha  \in {\mathbb Z} ^2_+$.
   For real $s \ge 0$, the Hilbert space $H^s(\Gamma)$ and the boundary norm $| \cdot |_s$ is defined by interpolation.  The negative-order
Sobolev spaces $H^{-s}(\Gamma)$ are defined via duality: for  real $s \ge 0$,
$$
H^{-s}(\Gamma) := [ H^s(\Gamma)]' \,.
$$

\subsection{Notation for derivatives and norms}
Throughout the paper, we will use the following notation:
\begin{align*}
D & = \text{ three-dimensional gradient vector } = \left( \frac{ \partial }{ \partial x_1}  , \frac{ \partial }{ \partial x_2},
\frac{ \partial }{ \partial x_3} \right) \,, \\
\bar \p & = \text{ two-dimensional gradient vector or {\it horizontal derivative} } = \left( \frac{ \partial }{ \partial x_1}  , \frac{ \partial }{ \partial x_2} \right) \,, \\
\operatorname{div} & =  \text{ three-dimensional divergence operator} \,, \\
\operatorname{div}_\eta & =  \text{ three-dimensional Lagrangian divergence operator} \,, \\
\operatorname{curl} & =  \text{ three-dimensional curl operator} \,, \\
\operatorname{curl}_\eta & =  \text{ three-dimensional Lagrangian curl operator} \,, \\
\operatorname{div}_\Gamma & =  \text{ two-dimensional  divergence operator} \,, \\
\| \cdot \|_s & = \text{  $H^s(\Omega)$ interior norm}   \,, \\
| \cdot |_s & = \text{  $H^s(\Gamma)$ boundary norm}   \,.
\end{align*}

\subsection{The outward unit normal to $\Gamma$}  We
set $N=(0,0,1)$ on $\{x_3=1\}$ and $N=(0,0,-1)$ on $\{ x_3= 0\}$.  We use the standard basis on $ \mathbb{R}^3  $:
$e_1=(1,0,0)$, $e_2=(0,1,0)$ and $e_3=(0,0,1)$.

\section{A higher-order Hardy-type inequality and some useful Lemmas}\label{sec::Hardy}

We will make fundamental use of the following generalization of the well-known Hardy inequality to higher-order derivatives:

\begin{lemma}[Higher-order Hardy-type inequality]\label{Hardy}   Let $s\ge 1$ be a given integer, and suppose that
\begin{equation}\nonumber
u\in H^s(  \Omega )\cap \h\,. 
\end{equation} 
If $d(x) >0$ for $x \in \Omega$,  $d\in H^{r}( \Omega )$, $r= \max(s-1,3)$, and $d$ is the distance function to $\partial \Omega $ near $\partial \Omega $,  then $\d\frac{u}{d}\in H^{s-1}( \Omega )$ and 
\begin{equation}
\label{Hardys}
\d\left\|\frac{u}{d}\right\|_{s-1}\le C \|u\|_s.
\end{equation}
\end{lemma}
\begin{proof}
Given the assumptions on $d(x)$, it is clear that (\ref{Hardys}) holds on all interior regions, and so on all open subsets $ \omega \subset \Omega $.  We thus prove
that this inequality holds near the boundary $\Gamma$.  All of our computations below will hold  near the boundary  $\Gamma$. 

We use an induction argument.
The case $s=1$ is of course the classical Hardy inequality. Let us now assume that  the inequality (\ref{Hardys}) holds for a given $s\ge 1$, and suppose that
$$u\in H^{s+1}(\Omega )\cap \h\,.$$ 
With $ \alpha$ denoting the multi-index $( \alpha _1, \alpha _2, \alpha _3)$ and using the notation $| \alpha | = \alpha _1+ \alpha _2+ \alpha _3$, we set
$$D^ \alpha = \frac{\partial ^{| \alpha |}}{\p_{x_1}^{ \alpha _1} \p_{x_2}^{ \alpha _2} \p_{x_3}^{ \alpha _3} }
$$
and for $m \in \mathbb{N}  $,
$$
D^m = \sum_{ | \alpha| =m} D^ \alpha  \,.
$$

A simple computation shows that for $m \in \mathbb{N}  $,
\begin{equation}
\label{Hardy1}
D^m(\frac{u}{d})=\frac{f}{d^{m+1}},
\end{equation}
with $$f=\sum_{k=0}^m C_m^k D^{m-k}u\ k! (-1)^k d^{m-k}$$
for a constant $C_m^k$ depending on $k$ and $m$.
From the regularity of $u$, we see that $f\in \h$. Next, with $D=D^1 $, we obtain the identity
\begin{align}
Df&=\sum_{k=0}^s C_s^k D^{s+1-k}u\ k! (-1)^k d^{s-k}+\sum_{k=0}^{s-1} C_s^k D^{s-k}u\ k! (-1)^k d^{s-k-1}(s-k)\n\\
&=D^{s+1}u\ s! (-1)^s d^s +\sum_{k=1}^s C_s^k D^{s+1-k}u\ k! (-1)^k d^{s-k}\n\\
& \qquad\qquad\qquad +\sum_{k=0}^{s-1} C_s^{k+1} D^{s-k}u\ (k+1)! (-1)^k d^{s-k-1}\n\\
&=D^{s+1}u\ s! (-1)^s d^s \,. \label{acs1}
\end{align}
Since $f\in \h$, we deduce from (\ref{acs1}) that for any $x \in \Omega$ with $x_3\in (0,\frac{1}{2}]$, 
\begin{equation*}
f(x_1,x_2,x_3)=(-1)^s s!\ \int_0^{x_3} D^{s+1}u(x_1,x_2,y_3)\ y_3^s dy_3,
\end{equation*}
which by substitution in (\ref{Hardy1}) yields the identity
\begin{equation*}
\d D^s(\frac{u}{d})(x_1,x_2, y_3)=\frac{(-1)^s s!\ \int_0^{x_3} D^{s+1}u(x_1,x_2,y_3)\ y_3^s dy_3}{x_3^{s+1}} \,.
\end{equation*}
Hence, a simple upper-bound provides the inequality
\begin{equation*}
\d \bigl|D^s(\frac{u}{d})(x)\bigr|\le s!\ \frac{\psi_1(x_3)\int_0^{x_3} |D^{s+1}u(x_1,x_2,y_3)|\  dy_3}{d(x)},
\end{equation*}
where $\psi_1$ is the piecewise affine function equal to $1$ on $[0,\frac{1}{2}]$ and to $0$ on $[\frac{3}{4},1]$. Next, for any $x \in \Omega$ with  $x_3\in [\frac{1}{2},1)$, we obtain in a similar fashion that
\begin{equation*}
\d \bigl|D^s(\frac{u}{d})(x)\bigr|\le s!\ \frac{\psi_2(x_3)\int_{x_3}^1 |D^{s+1}u(x_1,x_2,y_3)|\  dy_3}{d(x)},
\end{equation*}
where $\psi_2$ is the piecewise affine function equal to $0$ on $[0,\frac{1}{4}]$ and to $1$ on $[\frac{1}{2},1]$.
It follows  that for any $x\in \Omega $,
\begin{equation}
\label{Hardy2}
\d \bigl|D^s(\frac{u}{d})(x)\bigr|\le s!\ \frac{\psi_1(x_3)\int_0^{x_3} |D^{s+1}u(x_1,x_2,y_3)|\  dy_3+\psi_2(x_3)\int_{x_3}^1 |D^{s+1}u(x_1,x_2,y_3)|\  dy_3}{d(x)}.
\end{equation}
Setting
 $$
 g=\psi_1(x_3)\int_0^{x_3} |D^{s+1}u(x_1,x_2,y_3)|\  dy+\psi_2(x_3)\int_{x_3}^1 |D^{s+1}u(x_1,x_2,y_3)|\  dy_3  \,, 
 $$
we notice that $g\in \h$, with $$\|g\|_1\le C\|D^{s+1}u\|_0.$$ Therefore, by the classical Hardy inequality, we infer from (\ref{Hardy2}) that
\begin{equation}
\label{Hardy3}
\bigl\|D^s(\frac{u}{d})\bigr\|_0\le C \|g\|_1\le C \|D^{s+1}u\|_0.
\end{equation}
Since we assumed in our induction process that our generalized Hardy inequality is true at  order $s$, we then have that
 $$\bigl\|\frac{u}{d}\bigr\|_{s-1}\le C \|u\|_s,$$ 
which, together with (\ref{Hardy3}), implies that
$$\bigl\|\frac{u}{d}\bigr\|_{s}\le  C \|u\|_{s+1},$$
and thus establishes the property at order $s+1$, and concludes the proof.
\end{proof}

\subsection{$ \kappa $-independent elliptic estimates}
 In order to obtain estimates for solutions of our approximate $ \kappa $-problem (\ref{approx}) defined below in Section \ref{sec::approx}, which are independent of the regularization parameter $\kappa$ ,  we will  need the following Lemma, whose proof can be found in Lemma 1, Section 6 of \cite{CoSh2006}:
\begin{lemma}\label{kelliptic}
 Let $\kappa>0$ and $g\in L^\infty(0,T;H^s(\Omega )))$ be given, and let $f\in H^1(0,T;H^s(\Omega ))$ be such that $$f+\kappa f_t=g\ \ \ \text{in}\ (0,T)\times I.$$ Then, $$\|f\|_{L^\infty(0,T;H^s(\Omega ))}\le C\, \max\{\|f(0)\|_s,\|g\|_{L^\infty(0,T;H^s(\Omega ))}\}.$$
\end{lemma}
In practice, $f$ will usually denote $L(V)$, where $L$ is some nonlinear (possibly degenerate) elliptic-type  operator and $V$
is some combination of space and time derivatives of $\eta(t)$.


\subsection{The embedding of a weighted Sobolev space}
The derivative loss inherent to this degenerate problem is a consequence of the weighted embedding
we now describe.

Using $d$ to denote the distance function to
the boundary $\Gamma$,  and letting $p=1$ or $2$,
the weighted Sobolev space  $H^1_{d^p}(\Omega)$, with norm given by 
$\int_\Omega d(x)^p (|F(x)|^2+| DF (x)|^2 )\, dx$ for any $F \in H^1_{d^p}(\Omega)$, 
satisfies the following embedding:
$$H^1_{d^p}(\Omega) \hookrightarrow   H^{1 - \frac{p}{2}}(\Omega)\,;$$
therefore, there is a constant $C>0$ depending only on $\Omega$ and $p$, such that
\begin{equation}\label{w-embed}
\|F\|_{1-p/2} ^2 \le C \int_\Omega d(x)^p \bigl( |F(x)|^2 + \left| DF(x) \right|^2\bigr) \, dx\,.
\end{equation} 
See, for example,  Section 8.8 in Kufner \cite{K1985}.

\section{The  Lagrangian vorticity and divergence} \label{sec::lagcurldiv}

We use  the permutation symbol  (\ref{permutation}) to 
write  the basic identity regarding the $i$th component of the curl of a vector
field $u$:
$$
(\operatorname{curl}  u)_i = \varepsilon_{ijk} u^k,_j \,.
$$
The chain rule shows that
$$
(\operatorname{curl}  u(\eta))_i  = \varepsilon _{ijk} A^s_j v^k,_s \,,
$$
Using our definition (\ref{lagrangian_curl}) of the Lagrangian curl operator $ \operatorname{curl} _\eta$, we
write
\begin{equation}\label{lagcurl}
[\operatorname{curl} _\eta v]_i: = \varepsilon _{ijk} A^s_j v^k,_s \,.
\end{equation} 
Taking the Lagrangian curl of (\ref{ce_vor}) yields the Lagrangian
vorticity equation
\begin{equation} \label{vorticity}
\varepsilon _{kji} A^s_j  v_t^i,_s
=0\,, \ \ \text{ or } \ \ \operatorname{curl} _ \eta v_t =0 \,.
\end{equation}

Similarly, the chain-rule shows that $ \operatorname{div} u (\eta) = A^j_i v^i,_j$, and according to (\ref{lagrangian_div}),
\begin{equation} \label{lagdiv}
\operatorname{div} _\eta v = A^j_i v^i,_j \,.
\end{equation}

\section{Properties of the determinant $J$, cofactor matrix $a$, unit normal $n$, and a polynomial-type inequality}\label{sec::properties}
\subsection{Differentiating the Jacobian determinant}  The following identities will be
useful to us:
\begin{align}
{\bar\partial}J&=  a^s_r{\bar\partial} \frac{\partial \eta^r}{\partial x^s} {\text{ (horizontal differentiation )}}\,, \label{J1}\\
\partial_t  J&= a^s_r \frac{\partial v^r}{\partial x^s}  \ \ \text{ (time differentiation using $v=\eta_t$)} \,. \label{J2}
\end{align}

\subsection{Differentiating the cofactor matrix}  Using (\ref{J1}) and (\ref{J2}) and the fact that $a= J\, A$,
we find that
\begin{align}
{\bar\partial} a^k_i &=  {\bar\partial} \frac{\partial \eta^r}{\partial x^s} J^{-1}
[a^s_r a^k_i - a^s_i a^k_r]  
{\text{ (horizontal differentiation)} }\,, \label{a1}\\
\partial_t  a^k_i &=  \frac{\partial v^r}{\partial x^s}  J^{-1} 
[a^s_r a^k_i - a^s_i a^k_r]  
\ \ \text{ (time differentiation using $v=\eta_t$)} \,. \label{a2}
\end{align}

\subsection{The Piola identity}  It is a fact that the columns of every cofactor matrix
are divergence-free and satisfy
\begin{equation}\label{a3}
a^k_i, _k =0\,.
\end{equation}
The identity (\ref{a3}) will play a vital role in our energy estimates. (Note that we use
the notation cofactor for what is commonly termed the {\it adjugate matrix}, or the transpose
of the cofactor.)

\subsection{A geometric identity involving the curl operator}
\begin{lemma} \label{lem_curlcurl}
\begin{align*} 
\p_t a^k_i,_j a^j_i & = [\operatorname{curl} \operatorname{curl} v]^k + v^r,_{sj} \Bigl( J^{-1} [a^s_ra^k_i - a^s_i a^k_r]a^j_i - [\delta^s_r \delta ^k_i - \delta ^s_i \delta ^k_r] \delta ^j_i \Bigr)  \\
& \qquad\qquad\qquad\qquad +  v^r,_{s} \Bigl( J^{-1} [a^s_ra^k_i - a^s_i a^k_r] \Bigr),_j a^j_i\,.
\end{align*} 
\end{lemma} 
The structure of the right-hand side will be very important to us:  the curl
structure of the first term will be crucially used in  order to  construct solutions;
the second term can be made small by virtue of the fact that  $J^{-1} [a^s_ra^k_i - a^s_i a^k_r]a^j_i - [\delta^s_r \delta ^k_i - \delta ^s_i \delta ^k_r] \delta ^j_i$ can be made small for short time; the third term is lower-order with respect to the derivative count on $v$ and can
be made small using the fundamental theorem of calculus.

\begin{proof}[Proof of Lemma \ref{lem_curlcurl}] Using the identity (\ref{a2}), we see that
\begin{align*} 
\p_t a^k_i,_j a^j_i & = v^r,_{sj}  J^{-1} [a^s_ra^k_i - a^s_i a^k_r]a^j_i   
 +  v^r,_{s} \Bigl( J^{-1} [a^s_ra^k_i - a^s_i a^k_r] \Bigr),_j a^j_i\,.
\end{align*} 
Adding and subtracting $ [\delta^s_r \delta ^k_i - \delta ^s_i \delta ^k_r] \delta ^j_i$, and using the identity 
$ \operatorname{curl} \operatorname{curl} = D \operatorname{div} - \Delta $ yields the result.
\end{proof} 

\subsection{Geometric identities for the surface $\eta(t)(\Gamma)$}
The vectors 
$\eta,_ \alpha $ for $ \alpha =1,2$ span the tangent plane to the surface $\Gamma(t)= \eta(t)(\Gamma) $ in $\mathbb{R}  ^3$, and
$$
{\tau}_ 1 : =\frac{\eta,_ 1}{|\eta,_1 |} \,, \ \  {\tau}_ 2 : =\frac{\eta,_ 2}{|\eta,_2 |} \,, \ \ \text{ and }  n:= \frac{\eta,_1 \times \eta,_2}{ | \eta,_1 \times \eta,_2|}
$$
are the unit tangent and normal vectors, respectively, to $\Gamma$.    

Let ${g}_{ \alpha \beta }= \eta,_ \alpha  \cdot \eta,_ \beta $ denote the induced metric on the surface $\Gamma$; then
$\det g = | \eta,_1 \times \eta,_2|^2$  so that
$$
\sqrt{g}\, n:= \eta,_1 \times \eta,_2\,,
$$
where we will use the notation $\sqrt{g}$ to mean $\sqrt{ \det g}$.

By definition of the cofactor matrix,  the row vector
\begin{equation}\label{a3i}
{a}^3_i = \left[
\begin{array}{c}
\eta^2,_1 \eta^3,_2 - \eta^3,_1 \eta^2,_2 \\
\eta^3,_1 \eta^1,_2 - \eta^1,_1 \eta^3,_2 \\
\eta^1,_1 \eta^2,_2 - \eta^1,_2 \eta^2,_1 
\end{array}\right] \,, \text{ and } \sqrt{g} = | {a}^3_i| \,.
\end{equation}
It follows that
\begin{equation}\label{nisa3i}
n = {a}^3_i/ \sqrt{g}\,.
\end{equation} 


\subsection{A polynomial-type inequality} \label{subsec_poly} For a constant $M_0\ge 0$,  suppose that $f(t)\ge 0$, 
$t \mapsto f(t)$ is continuous,  and
\begin{equation}\label{f}
f(t) \le M_0 + C\,t\, P(f(t))\,,
\end{equation}
where $P$ denotes a polynomial function,  and  $C$ is a generic constant.
Then for $t$ taken sufficiently small, we have the bound
$$
f(t) \le 2M_0\,.
$$
This type of inequality, which we introduced in  \cite{CoSh2006}, can be viewed
as  a generalization of standard nonlinear Gronwall inequalities.  We will make use of this inequality often
in our subsequent analysis, and also in the form wherein $\sqrt{t}$ replaces $t$.


\section{Trace estimates and the Hodge decomposition elliptic estimates}\label{sec::Hodge}
The normal trace theorem  provides  the existence of the
normal trace $w \cdot N$ of a velocity field $w\in L^2(\Omega)$ with
${\operatorname{div}}  w \in L^2(\Omega) $ (see, for example, \cite{Temam1984}). For our purposes, the following
form is most useful:  if $\bar \p w \in L^2(\Omega) $ with  ${\operatorname{div}}  w\in
L^2(\Omega) $, then $\bar \p w\cdot N$ exists in
$H^{-0.5}(\Gamma)$ and
\begin{align}
\|\bar \p w\cdot N\|^2_{H^{-0.5}(\Gamma)} \le C \Big[\|\bar \p w\|^2_{L^2(\Omega)}
+ \|{\operatorname{div}}   w\|^2_{L^2(\Omega) }\Big] \label{normaltrace}
\end{align}
for some constant $C$ independent of $w$. In addition to the normal
trace theorem, we have the following
\begin{lemma}\label{tangential_trace}
Let $\bar \p w\in L^2(\Omega)$ so that ${\operatorname{curl}}  w\in L^2(\Omega) $, and let
$T_1$, $T_2$ denote the unit tangent vectors on $\Gamma$,
so that any vector field $u$ on $\Gamma$ can be uniquely written as $u^\alpha
T_\alpha$. Then
\begin{align}
\|\bar \p w\cdot T_\alpha\|^2_{H^{-0.5}(\Gamma)} \le C
\Big[\|\bar \p w\|^2_{L^2(\Omega)} + \| {\operatorname{curl}} 
w\|^2_{ L^2(\Omega) }\Big]\,,\qquad\alpha=1,2 \label{tangentialtrace}
\end{align}
for some constant $C$ independent of $w$.
\end{lemma}
See \cite{ChCoSh2007} for the proof.
Combining (\ref{normaltrace}) and (\ref{tangentialtrace}), 
\begin{align}
\|\bar \p w\|_{H^{-0.5}(\Gamma)} \le C\Big[\|\bar \p w\|_{L^2(\Omega)} + \|{\operatorname{div}} 
w\|_{L^2(\Omega) } + \|{\operatorname{curl}}  w\|_{L^2(\Omega) }\Big]
\label{tracetemp}
\end{align}
for some constant $C$ independent of $w$.

The construction of our higher-order energy function is based on the following Hodge-type elliptic estimate:
\begin{proposition}\label{prop1}
For an $H^r$ domain $\Omega$, $r \ge 3$,
if $F \in L^2(\Omega;{\mathbb R} ^3)$ with $\operatorname{curl}F \in H^{s-1}(\Omega;{\mathbb R} ^3)$,
${\operatorname{div}}F\in H^{s-1}(\Omega)$, and $F \cdot N|_{\Gamma} \in
H^{s -{\frac{1}{2}}}(\Gamma)$ for $1 \le s \le r$, then there exists a
constant $\bar C>0$ depending only on $\Omega$ such that
\begin{equation}
\begin{array}{l}
\|F\|_s \le \bar C\left( \|F\|_0 + \|\operatorname{curl} F\|_{s-1}
+ \|\operatorname{div} F\|_{s-1} + |\bar \p F \cdot N|_{s-{\frac{3}{2}}}\right)\,, \\
\|F\|_s \le \bar C\left( \|F\|_0 + \|\operatorname{curl} F\|_{s-1}
+ \|\operatorname{div} F\|_{s-1} + \sum_{ \alpha =1}^2 |\bar \p F \cdot T_ \alpha |_{s-{\frac{3}{2}}}\right)\,,
\end{array}
\label{hodge}
\end{equation}
where $N$ denotes the outward unit-normal to $\Gamma$, and $T_ \alpha$ are tangent vectors for $ \alpha =1,2$.
\end{proposition}
These estimates are well-known and follows from the identity $-\Delta F= {\operatorname{curl}}\, 
{\operatorname{curl}}F - D {\operatorname{div}}F$; a convenient reference is Taylor
\cite{Taylor1996}.

\section{An asymptotically consistent degenerate parabolic $ \kappa $-approximation of the compressible Euler equations in vacuum}
\label{sec::approx}
In order to construct solutions to (\ref{ce0}), we will add a specific artificial viscosity term to the Euler equations that preserves much
of the geometric structure of the Euler equations which is so important for our estimates,  and which produces a degenerate parabolic approximation, which we term the approximate $ \kappa $-problem.

\subsection{Smoothing the initial data}\label{subsec::initdata}
For the purpose of constructing solutions, we will smooth the initial velocity field $u_0$.  We will also smooth the initial density field
$\rho_0$ while preserving the conditions that $\rho(x) >0$ for $ x \in \Omega$, and that $\rho_0$ satisfies (\ref{degen}) near $\Gamma$.

For $\vartheta>0$, 
let $0 \le \varrho_ \vartheta \in C^ \infty _0( \mathbb{R}  ^3)$ denote the standard family of mollifiers with $\spt (\varrho_ \vartheta )
\subset \overline{ B(0, \vartheta)}$,  and   let $\mathcal{E}_ \Omega  $ denote a
Sobolev extension operator mapping $ H^s(\Omega)$ to $H^s( \mathbb{R}^3  )$ for $ s\ge 0$. 

We set $u_0^ \vartheta = \varrho_ \vartheta*\mathcal{E}_ \Omega (u_0)$, so that for $\vartheta >0$, $u_0^ \vartheta \in C^ \infty (\overline 
\Omega )$.   The smoothed initial density function $\rho_0 ^ \vartheta$ is defined as the solution of the fourth-order elliptic equation
\begin{subequations}
  \label{rhozero}
\begin{alignat}{2}
\Delta ^2 \rho_0^ \vartheta &= \varrho_ \vartheta*\mathcal{E}_ \Omega (\Delta ^2 \rho_0)  \ \ \ &&\text{in} \ \ \Omega
\,, \label{rhozero.a}\\
\rho_0^ \vartheta &= 0 \ \ \ &&\text{on} \ \ \Gamma  \,, \label{rhozero.b}\\
\frac{ \p \rho_0 ^ \vartheta}{\p N} &=  \Lambda _\vartheta  \frac{ \p \rho_0 }{\p N}   &&\text{on } \Gamma \,, \label{rhozero.c} \\
(x_1,x_2) & \mapsto \rho_0(x_1,x_2,x_3) \text{ is $1$-periodic  }  \,. \label{rhozero.d}
\end{alignat}
\end{subequations}
$ \Lambda _ \vartheta$ is the boundary convolution operator defined in Section \ref{subsec_hor}.
By elliptic regularity,  $\rho_0^\vartheta \in C^ \infty (\overline \Omega)$, and by choosing $\vartheta>0$ sufficiently small, we see
that $\rho_0^\vartheta(x) >0$ for $x \in \Omega $, and that the physical vacuum condition (\ref{degen}) is satisfied near $\Gamma$.
This follows from the  fact that $ \frac{ \p \rho_0 ^ \vartheta}{\p N} < 0$ on $\Gamma$ for $\vartheta >0$ taken sufficiently small, which
implies that $\rho_0^\vartheta(x)>0$  for $ x \in \Omega $ very close to $\Gamma$.  On the other hand, $\rho_0^\vartheta(x) >0$ for
all $x \in \omega $ for any open subset $ \omega \subset \Omega $ by taking $ \vartheta$ sufficiently small.

Until Section \ref{subsec_optimal}, for notational convenience, we will denote $u_0^\vartheta$ by $u_0$ and $\rho_0^\vartheta$ by 
$\rho_0$.    In Section \ref{subsec_optimal},  we will show that Theorem \ref{theorem_main} holds with the
optimal regularity stated therein.

\subsection{The degenerate parabolic approximation to the compressible Euler equations: the  $ \kappa $-problem}

\begin{definition}[The approximate $ \kappa $-problem]
   For $ \kappa > 0$,
we consider the following sequence of  degenerate parabolic approximate  $ \kappa $-problems:
\begin{subequations}
  \label{approx}
\begin{alignat}{2}
 \rho_0 v_t^i + a^k_i ( \rho_0 ^2   J^{-2} ),_k + \kappa \p_t\bigl[a^k_i ( \rho_0^2   J^{-2} ),_k \bigr]&=0  &&\text{in} \ \ \Omega \times (0,T_ \kappa ] 
\,, \label{approx.a}\\
(\eta,v)&= (e,u _0) \ \ \ &&\text{on} \ \ \Omega \times \{t=0\} \,, \label{approx.b}\\
\rho_0  &= 0   &&\text{on } \Gamma \,. \label{approx.c}
\end{alignat}
\end{subequations}
\end{definition}

 Solutions to (\ref{ce0}) will be found in the limit as $ \kappa \to 0$.
 
Note that (\ref{approx.a}) can be equivalently written in a form that is  essential for the curl estimates 
 that we shall present below:
\begin{equation} 
v_t^i + 2 A^k_i (\rho_0 J^{-1} ),_k +2\kappa  \p_t\bigl[A^k_i (\rho_0 J^{-1} ),_k\bigr] =0 \,. \tag{\ref{approx.a}'}
\end{equation}

\begin{remark} There appear to be few other possible choices for the artificial viscosity term given in (\ref{approx.a}).  Our
choice, $ \kappa \p_t\bigl[a^k_i ( {\rho_0 }^2   J^{-2} ),_k \bigr]$, preserves the structure of the energy estimates and also, thanks
to Lemma \ref{kelliptic}, the structure of the elliptic-type estimates that we use to bound normal derivatives.   On the other hand,
the addition of this artificial parabolic term does not exactly preserve the transport structure of vorticity, but instead produces
error terms that we can nevertheless control.  
\end{remark} 

\begin{remark} Note that we do not require any compatibility conditions on the initial data in order to solve the Euler
equations (\ref{ce0}) (or in Eulerian form (\ref{ceuler})), and the same remains true for our approximate $ \kappa $-problem
(\ref{approx}).  The lack of compatibility conditions stems from the degeneracy condition (\ref{degen}) which allows us to
solve for $\eta$ and $v$ without prescribing any boundary conditions on displacements or velocities.
\end{remark} 

\subsection{Time-differentiated velocity fields at $t=0$}  
Given $u_0$ and $\rho_0$,  and using the fact that $\eta(x,0)=x$, the quantity $v_t|_{t=0}$ for the degenerate parabolic $ \kappa $-problem is computed using (\ref{approx.a}'):
$$
v^i_t|_{{t=0}} =- \left.\left(  2 \kappa \p_t[ A^k_i (\rho_0 J^{-1} ),_k] + 2 A^k_i (\rho_0 J^{-1} ),_k  \right)\right|_{t=0} = 
 (2 \kappa \rho_0 \operatorname{div} u_0  - 2 \rho_0 ),_i + 2 \kappa u_0^k,_i \rho_0,_k \,.
$$
Similarly, for all $k \in \mathbb{N}  $,
$$
\partial_t^k v^i |_{{t=0}} = \left. \frac{\partial^{k-1}}{\partial t^{k-1}}\left(   -2\kappa \p_t[A^k_i (\rho_0 J),_k] -  2 A^k_i (\rho_0 J^{-1} ),_k      \right)\right|_{t=0} \,.
$$
These formulae make it clear that  each $\partial_t^k v|_{t=0}$ is a function of space-derivatives of $u_0$ and $\rho_0$.



\subsection{Introduction of the $X$ variable and the $ \kappa $-problem as a function of $X$}  \label{subsec_X} We 
consider a heat-type equation which arises by letting  $a^j_i \partial_{x_j}$ act upon equation (\ref{approx.a}), and using the 
Piola identity (\ref{a3}):
\begin{align} 
&a^j_iv_t^i,_j + \kappa  \Bigl[   a^j_i a^k_i \rhoi(\rho_0^2 \p_t J^{-2} ),_k\Bigr],_j = - \kappa [ a^j_i \p_t a^k_i  \rhoi (\rho_0 ^2 J^{-2} ),_k],_j
    - 2 [a^j_i A^k_i (\rho_0 J^{-1} ),_k],_j \,. \label{heatJ1}
\end{align} 
Since $\p_t J^{-2} = -2J^{-3} J_t $, we write (\ref{heatJ1}) as 
\begin{align} 
&a^j_iv_t^i,_j - 2 \kappa  \Bigl[   a^j_i a^k_i \rhoi(\rho_0^2 J^{-3}  J_t ),_k\Bigr],_j = - \kappa [ a^j_i \p_t a^k_i  \rhoi (\rho_0 ^2 J^{-2} ),_k],_j
    - 2 [a^j_i A^k_i (\rho_0 J^{-1} ),_k],_j \,. \label{heatJ}
\end{align} 
\begin{definition}[The $X$ variable] \label{defn_X}
We set
\begin{equation}\label{X}
X= \rho_0 J^{-3}  J_t = \rho_0  J^{-3} a ^s_r v^r,_s = \rho_0 J^{-2} \operatorname{div} _\eta v \,.
\end{equation} 
\end{definition}

Using (\ref{X}), we see that
\begin{align*} 
a^j_i v_t^i,_j  &= J_{tt} - \p_t a^j_i \, v^i,_j 
 = \frac{J^3 X_t}{ \rho_0} + 3 J^{-1} (J_t)^2 - \p_t a^j_i \, v^i,_j \,,
\end{align*} 
so that we can rewrite (\ref{heatJ}) as the following nonlinear heat-type equation for $X$:
\begin{align} 
\frac{J^3 X_t}{\rho_0}- 2 \kappa  \Bigl[   a^j_i a^k_i \rhoi(\rho_0 X),_k\Bigr],_j &= - \kappa [ a^j_i \p_t a^k_i  \rhoi (\rho_0 ^2 J^{-2} ),_k],_j
    - 2 [a^j_i A^k_i (\rho_0 J^{-1} ),_k],_j  \nonumber \\
&\qquad\qquad  - 3 J^{-1} (J_t)^2 + \p_t a^j_i \, v^i,_j\,.  \label{heatX}
\end{align} 

It follows from (\ref{X}) that
\begin{equation}\label{ldivv}
\operatorname{div}_\eta v = \frac{ (X J^2)}{\rho_0} \,,
\end{equation} 
so that time-differentiating (\ref{ldivv}), we see that
\begin{equation}\label{ldivvt}
\operatorname{div}_ \eta v_t = \frac{ (X J^2)_t}{\rho_0} - \p_t A^j_i v^i,_j \,.
\end{equation} 

\subsection{The nonlinear Lagrangian  vorticity equation} The analogue of (\ref{lagcurl}) for our approximate $ \kappa $-problem
takes the form, with $f = \rho_0 J^{-1} $,
\begin{align} 
\operatorname{curl} _ \eta v_t &  = 2 \kappa \varepsilon_{\cdot ji}  v^r,_s   A^s_i \bigl[ f,_l A^l_r\bigr],_m A^m_j   \n \\
& = 2 \kappa \varepsilon_{\cdot ji}  v^r,_s   A^s_i \bigl[ \rho,_{rj}(\eta)\bigr]\,,
\label{lcurlvt}
\end{align} 
where $f=\rho(\eta)$.

We now explain how the formula (\ref{lcurlvt}) is obtained.    We have that the $k$th component of the Lagrangian curl is
\begin{align*} 
[\operatorname{curl} _ \eta v_t]^k &  = -2 \kappa \varepsilon_{k ji}  [   \p_t A^l_i f,_l  + A^l_i \p_t f,_l],_r A^r_j \\
&  = -2 \kappa \varepsilon_{k ji}  [   \p_t A^l_i f,_l ],_r A^r_j  \,,
\end{align*} 
where we have used the Lagrangian version of the fact that the curl operator annihilates the gradient operator; namely
$\varepsilon_{k ji}  (A^r_i F,_r),_j =0$ for all differentiable $F$.

It is now convenient to switch back to Eulerian variables.  We  expand $\p_t A^l_i$, and write
\begin{align*} 
 \p_t A^l_i f,_l = -  v^r,_s A^s_i  f,_l A^l_r = - [u^r,_i \rho,_r] \circ \eta \,.
\end{align*} 
Now we can compute the standard curl
operator of this quantity to find that
\begin{align*} 
\varepsilon_{ kji} [u^r,_i \rho,_r] ,_j & = \varepsilon_{ kji} u^r,_{ij} \rho,_r + \varepsilon_{ kji} u^r,_i \rho,_{rj} \\
& =  \varepsilon_{ kji} u^r,_i \rho,_{rj}  \,.
\end{align*} 
Reverting back to Lagrangian variables yields the identity (\ref{lcurlvt}).

\subsection{A boundary identity for the approximate $ \kappa $-problem} For the purposes of constructing solutions to (\ref{approx}) we
will need the formula for the normal (or vertical) component of $v_t$ on $\Gamma$:
\begin{align} 
v_t ^3 & = -2J^{-2} a^3_3  \rho_0,_3 - 2 \kappa \p_t[J^{-2} a^3_3]  \rho_0,_3   \n\\
            & =   -2J^{-2} a^3_3  \rho_0,_3 - 2 \kappa J^{-2} \p_t a^3_3 \rho_0,_3  - 2 \kappa \p_t J^{-2} \, a^3_3 \rho_0,_3 \,,
            \label{bidentity}
\end{align} 
where
\begin{align} 
a^3_3 & =  ( \eta,_1 \times \eta,_2) \cdot e_3 \label{a33} \,, \\
\p_t a^3_3 &  = ( v,_1 \times \eta,_2 + \eta,_1 \times v,_2) \cdot e_3
\label{at33} \,.
\end{align} 
We note for later use that linearizing (\ref{at33}) about $\eta =e$ produces $ \operatorname{div} _\Gamma v$ as the linearized
analogue of $\p_ta^3_3$.

\section{Solving the parabolic $\kappa$-problem (\ref{approx}) by a fixed-point method}\label{sec_kproblem}


\subsection{Functional framework for the fixed-point scheme and some notational conventions}

For $T>0$, we shall denote by $\boldsymbol{X}_T$ and $\boldsymbol{Y}_T$ the following Hilbert spaces:
\begin{align*}
\boldsymbol{X}_T &=\Bigl\{v\in L^2(0,T;H^4(\Omega))|\ 
 \p_t^a v \in L^2(0,T;H^{4-a}(\Omega))\,, \ \ a=1,2,3
 \Bigr\} \,, \\
 \boldsymbol{Y}_T& =\Bigl\{y\in L^2(0,T;H^3(\Omega))|\ 
 \p_t^a y \in L^2(0,T;H^{3-a}(\Omega))\,, \ \ a=1,2,3  \Bigr\}\,, \\
  \boldsymbol{Z}_T& =\Bigl\{v \in \boldsymbol{X}_T  \  | \  \rho_0 Dv \in \boldsymbol{Y}_T
 \Bigr\}\,, 
\end{align*}
endowed with their natural Hilbert norms:
\begin{align}
&\|v\|_{\boldsymbol{X}_T}^2  = \sum_{a=0}^3 \|\p_t^a v \|^2_{L^2(0,T;H^{4-a}(\Omega))}\,, \ \ \ 
\|y\|_{\boldsymbol{Y}_T}^2  = \sum_{a=0}^3 \|\p_t^a v \|^2_{L^2(0,T;H^{3-a}(\Omega))}\,,\n  \\
&\qquad \qquad \qquad \text{ and } \|v\|_{\boldsymbol{Z}_T}^2 = \| v\|_\XT^2 + \| \rho_0 Dv \|_\YT^2
\,. \label{XYnorms}
\end{align}  

For $M>0$, we define the following closed, bounded, convex subset of $\boldsymbol{X}_T$:
\begin{align}\label{ctm}
\mathcal{C}_T(M)=&\{ v\in \boldsymbol{Z}_T \  : 
 \  \|v\|_{\boldsymbol{Z}_T}^2\le M \}\,,  
 \end{align} 
where we  define the polynomial function $ \mathcal{N} _0$  of norms of the  initial data as follows:
\begin{equation}\label{N0}
\mathcal{N} _0 = P( \|  u_0\|_{100}, \|\rho_0\|_{100}) \,.
\end{equation} 
Since we have smoothed the initial data $u_0$ and $\rho_0$, we can  use the artificially high $H^{100}(\Omega)$-norm  in $ \mathcal{N} _0$.
  Later, in Section \ref{subsec_optimal}, we produce the optimal regularity for this initial data.

Henceforth, we assume that $T>0$ is given such that independently
of the choice of $v \in \mathcal{C} _T(M)$,  
$$
\eta(x,t) =x + \int_0^t v(x,s)ds
$$
is injective for $t \in [0,T]$, and that 
$${\frac{1}{2}} \le J(x,t) \le {\frac{3}{2}} \text{ for } \in [0,T] \text{ and } x \in \overline \Omega \,.$$
This can
be achieved by taking $T>0$ sufficiently small:  with $e(x)=x$, notice that
$$
\| J( \cdot ,t) -1\|_{L^ \infty ( \Omega )} \le C\| J( \cdot ,t) -1\|_2    = \| \int_0^t a^s_r(\cdot ,s)  v^r,_s ( \cdot ,s)ds\|_2 \le C\sqrt{T}M \,.
$$

In the same fashion, we can take $T>0$  small enough to ensure that on $[0,T]$ and for some $\lambda >0$,
\begin{equation}\label{pos_def}
2\lambda | \xi |^2 \le  a ^j_i(x,t) a^k_i(x,t) \xi _j \xi _k  \ \ \forall \ \xi \in \mathbb{R}^3  \,, x \in \Omega \,.
\end{equation} 
The space $ \ZT$  will be appropriate for our fixed-point methodology to prove existence of a solution to our degenerate parabolic
 $  \kappa$ problem (\ref{approx}).

\begin{theorem}[Solutions to the $\kappa$-problem]\label{thm_ksoln4}
Given  smooth initial data with $\rho_0$ satisfying $\rho_0(x)>0$ for $x\in \Omega $ and verifying the physical vacuum condition
(\ref{degen}) near $\Gamma$,  for $T_ \kappa>0 $ sufficiently small,  there exists a unique solution $v \in \boldsymbol{Z}_{T_{ \kappa} }$ to the  degenerate
parabolic $ \kappa $-problem (\ref{approx}).
\end{theorem} 
The remainder of Section \ref{sec_kproblem} will be devoted to the proof of Theorem \ref{thm_ksoln4}.

\subsection{Implementation of the fixed-point scheme for the $ \kappa $-problem (\ref{approx})}
Given $\bar v \in \mathcal{C}_T(M)$,  we define	 $\bar \eta (t) =e + \int_0^t \bar v (t')dt'$,   and set
$$\bar A = [D \bar \eta] ^{-1}\,, \ \  \bar J = \det D\bar \eta\,, \ \  \text{ and }  \bar a = \bar J \bar A\,. $$
Next, we set
$$
\bar B^{jk} =  \bar a^j_i \bar a^k_i  \ \text{ the positive definite, symmetric coefficient matrix} \,.
$$
Linearizing  (\ref{heatX}), we define $\bar X$ to be the solution of the following linear and degenerate
parabolic problem:
\begin{subequations}
\label{ssX}
\begin{alignat}{2}
\frac{\bar J^3 \bar X _t}{\rho_0} - 2\kappa  \Bigl[\bar B^{jk} \rhoi ( \rho_0 \bar X)  ,_k\Bigr],_j 
&= \bar G
 &&  \text{ in }   \Omega \times (0,T_{ \kappa } ] \,, \label{ssX.a}\\
\bar X & =0 &&  \text{ on }   \Gamma \times (0,T_{ \kappa } ] \,, \label{ssX.b}\\
(x_1,x_2) & \mapsto \bar X(x_1,x_2,x_3,t) \text{ is $1$-periodic}\,,  \label{ssX.c} \\
\bar X & = X_0:=\rho_0 \operatorname{div} u_0\ \ \ &&  \text{ on }   \Omega  \times \{t=0\} \,, \label{ssX.d}
\end{alignat} 
\end{subequations}
where the forcing function $ \bar G $ is defined as
\begin{align} 
\bar G & = - \kappa [ \bar a^j_i \p_t \bar a^k_i  \rhoi (\rho_0 ^2 \bar J^{-2} ),_k],_j
    - 2 [\bar a^j_i \bar A^k_i (\rho_0 \bar J^{-1} ),_k],_j  
 - 3 \bar J^{-1} (\bar J_t)^2 + \p_t \bar a^j_i \, \bar v^i,_j  \,. \label{Gforce}
\end{align} 
We shall establish the following
\begin{proposition} \label{prop_X} For $T>0$ taken sufficiently small, 
there exists a unique solution to (\ref{ssX}) satisfying
$$
\|\bar X \|^2_\XT \le \mathcal{N} _0 +  \pbv  + \pcurl \,,
$$
with the norms $\XT$, $\YT$, and $\ZT$ defined in (\ref{XYnorms}), and once again $P$ denotes a generic polynomial function
of its arguments.  (Generic constants are absorbed by the constants in our generic polynomial function $P$.)
\end{proposition} 
The proof of Proposition \ref{prop_X} will be given in Sections \ref{sec832}--\ref{sec836}.

\subsection{The definition of the velocity field $v$}  We will define a linear  elliptic system  of equations for $v$ which should be
viewed as the linear analogue of equations (\ref{ldivvt}), (\ref{lcurlvt}), and (\ref{bidentity}).

\begin{definition}[The linear system for the velocity-field $v(t)$]
With $\bar v \in \mathcal{C} _T(M)$ given,  and $\bar X$  obtained by solving the linear problem  (\ref{ssX}), we are
now in a position to define $v(t)$ on  $ [0,T_ \kappa ]$ by   specifying its divergence and curl in $\Omega$, as well as
the trace of its normal component on the boundary $\Gamma$:

\begin{subequations}\label{defv}
\begin{alignat}{2}
\operatorname{div} v_t & = \operatorname{div} \bar v_t  - \operatorname{div} _{ \bar \eta} \bar v_t + \frac{[\bar X \bar J^2]_t}{\rho_0} -
\p_t \bar A^j_i \bar v^i,_j
\  && \text{in}\ \Omega \,,        \label{defva}\\
\operatorname{curl} v_t  &=  \operatorname{curl} \bar v_t - \operatorname{curl} _{\bar \eta} \bar v_t
+2\kappa \varepsilon_{ \cdot ji}
\bar v,_s^r   \bar A^s_i\  \bar\Xi,_r^j(\bee) + \bar {\mathfrak{C}} \ && \text{in}\ \Omega \,,\label{defvb}\\
v_t^3 + 2 \kappa \rho_0,_3 \operatorname{div} _\Gamma v&=2 \kappa \rho_0,_3 \operatorname{div} _\Gamma \bar v   - 2 \rho_0,_3 \bar J^{-2} \bar a^3_3 \n \\
 & \qquad  \qquad  -2 \kappa  \rho_0,_3 \bar J^{-2} \p_t \bar a^3_3 -2 \kappa \rho_0,_3 \bar a^3_3 \p_t \bar J^{-2} + \bar c(t) N^3
\qquad  &&   \text{on}\ \Gamma  \,, \label{defvc} \\
\int_ \Omega v_t^ \alpha dx& = -2 \int_ \Omega \bar A^k_ \alpha (\rho_0 \bar J^{-1} ),_k dx 
-2 \kappa  \int_ \Omega \p_t[ \bar A^k_ \alpha (\rho_0 \bar J^{-1} ),_k] dx  \,,  \label{defvd} \\
(x_1,x_2)& \mapsto v_t(x_1,x_2,x_3,t)  \text{ is $1$-periodic}   \,,   \label{defv.e} 
\end{alignat}
\end{subequations}
where the presence of
$ \operatorname{div} _\Gamma \bar v$  in (\ref{defvc})  represents the linearization of $\p_t \bar a^3_3$ about $\bar \eta =e$, 
and where
\begin{align} 
\bar a^3_3 &  = e_3 \cdot ( \bar \eta,_1 \times \bar \eta,_2) \,, \n \\
\p_t \bar a^3_3 &  = e_3 \cdot ( \bar v,_1 \times \bar \eta,_2 + \bar \eta,_1 \times \bar v,_2) \,, \label{pta33}
\end{align} 
 the function $\bar c(t)$  (a constant in $x$) on the right-hand side of (\ref{defvc}) is defined by
\begin{align}
\bar c(t)=& {\frac{1}{2}} \int_ \Omega ( \operatorname{div} \bar v_t - \operatorname{div} _{ \bar \eta} \bar v_t) dx  + {\frac{1}{2}} 
\int_ \Omega  \frac{[\bar X \bar J^2]_t}{\rho_0} dx - {\frac{1}{2}} \int_ \Omega \p_t \bar A^j_i \bar v^i,_j dx  + \int_ \Gamma \bar J^{-2} \bar a^3_3 \rho_0,_3 N^3 dS \n \\
& \ + \kappa \int_\Gamma  \bar J^{-2} \p_t \bar a^3_3 \rho_0,_3 N^3 dS + \kappa \int_\Gamma \p_t \bar J^{-2}  \bar a^3_3 \rho_0,_3 N^3 dS
+ \kappa \int_\Gamma \operatorname{div}_\Gamma  (v-\bar v) \rho_0,_3 N^3 dS \,,
\end{align}
and where
the vector field $\bar\Xi(\bee)$ on the right-hand side of (\ref{defvb}) is defined on
$[0,T]\times\Omega$ as the solution of the ODE
\begin{subequations}\label{ODEXi}
\begin{align}
\bar v_t+2\bar\Xi(\bee)+2\kappa [\bar\Xi(\bee)]_t&=0\,,\\
\bar\Xi(0)=D\rho_0\,.
\end{align}
\end{subequations}
 The vector field $\bar{\mathfrak C}$ on the right-hand side of (\ref{defvb}) is then defined on $[0,T] \times \Omega$  by
\begin{equation}
\label{may11.1}
\bar{\mathfrak C}^i=2\bar A_i^j\psi,_j+2\kappa [2\bar A_i^j\psi,_j]_t\,,
\end{equation}
where $\psi$ is solution of the following time-dependent elliptic-type problem for $t\in [0,T]$:
\begin{subequations}
\begin{align}
2[\bar A_i^j\psi,_j],_i+2\kappa \partial_t[\bar A_i^j\psi,_j]_i&=\operatorname{div}\bigl(\operatorname{curl} _{\bar \eta} \bar v_t
-2\kappa \varepsilon_{ \cdot ji}
\bar v,_s^r   \bar A^s_i\  \bar\Xi,_r^j(\bee)\bigr) \ \text{ in }\ \Omega\,,\\
\psi&=0\ \text{ on }\ \Gamma\,,\\
(x_1,x_2) \mapsto &\psi(x_1,x_2,x_3,t) \ \text{ is $1$-periodic}\,,\\
\psi|_{t=0}&=0\ \text{ in }\ \Omega\,,
\end{align}
\label{may11.2}
\end{subequations}
so that we have the compatibility condition for (\ref{defvb})
\begin{equation}
\label{may11.3}
\operatorname{div}\bigl(-\operatorname{curl} _{\bar \eta} \bar v_t
+2\kappa \varepsilon_{ \cdot ji}
\bar v,_s^r   \bar A^s_i\  \bar\Xi,_r^j(\bee)\bigr)+\bar{\mathfrak{C}}\bigr)=0\  \text{ in }\ \Omega \times [0,T]\,.
\end{equation}
\end{definition}

An integrating factor provides us with a closed-form solution to the ODE (\ref{ODEXi}), and by employing integration-by-parts
in the time integral, we find that
\begin{align}
\bar\Xi(\be)(t,\cdot)&=e^{-\frac{t}{\kappa}}D\rho_0(\cdot)-\int_0^t
\frac{e^{\frac{t'-t}{\kappa}}}{2\kappa} \bar v_t(t',\cdot) dt'\ ,\n\\
&=e^{-\frac{t}{2\kappa}}D\rho_0(\cdot)+\int_0^t
\frac{e^{\frac{t'-t}{\kappa}}}{2\kappa^2} \bar v(t',\cdot)
dt'-\frac{1}{2\kappa}\bar
v(t,\cdot)+\frac{e^{-\frac{t}{\kappa}}}{2\kappa} u_0(\cdot)\,.  \label{Xi}
\end{align}
The formula (\ref{Xi}) shows that $\bar\Xi(\bee)$ has the same regularity as $\bar v$.   This gain in regularity is remarkable and should
be viewed as one of the key reasons that permit us to construct solutions to (\ref{approx}) using the linearization (\ref{defv}) with
a fixed-point argument.  

Similarly, we notice that
\begin{align} 
2[\bar A_i^j\psi,_j],_i(t,\cdot)
&=
\int_0^t \frac{e^{\frac{t'-t}{\kappa}}}{2\kappa} \operatorname{div}\bigl(\operatorname{curl} _{\bar \eta} \bar v_t -2\kappa \varepsilon_{ \cdot ji}
\bar v,_s^r   \bar A^s_i\  \bar\Xi,_r^j(\bee)\bigr)(t',\cdot)\ dt' \n \\
&=
\int_0^t   \frac{e^{\frac{t'-t}{\kappa}}}{2\kappa}   \varepsilon_{ k ji}
 \bigl[ \bar v^i,_r  \bar A^r_ s \bar v^s,_l \bar A^l_j      - \frac{1}{\kappa} \bar v^i,_r  \bar A^r_j 
  -2\kappa 
\bar v,_s^r   \bar A^s_i\  \bar\Xi,_r^j(\bee)\bigr],_k (t',\cdot)\ dt'  \n \\
& \qquad\qquad +  \frac{e^{\frac{t'-t}{\kappa}}}{2\kappa} \operatorname{div} \operatorname{curl} _{\bar \eta} \bar v
\label{may11.7}
\end{align} 

Since we can rewrite  the left-hand side of (\ref{may11.7}) as 
$2\Delta \psi+2[(\bar A_i^j-\delta_i^j)\psi,_j],_i$, and 
 with $\bar v\in \mathcal{C}_T(M)$, the elliptic problem (\ref{may11.7}) is 
well-defined and together with the boundary condition (\ref{may11.2}.b) provides  the following estimates for any $t\in [0,T]$:
\begin{subequations}\label{may11.4}
\begin{align} 
\|\psi(t)\|_4& \le \mathcal{N} _0 +   C\Bigl( T \|\bar v(t) \|_4  + \int_0^t \| \bar v\|_4\Bigr) \,, \\
\|\psi_t(t)\|_3 &\le \mathcal{N} _0 +   C\Bigl(T \|\bar v_t(t) \|_3 +\|\bar v(t)\|_3 + \int_0^t \| \bar v\|_3\Bigr) \,.
\end{align} 
\end{subequations}
Using (\ref{may11.1}), the estimates (\ref{may11.4}) lead to
\begin{subequations}
\label{may11.5}
\begin{align} 
\|\bar{\mathfrak C}\|_{2} & \le  \mathcal{N} _0 +   C\Bigl(T \|\bar v_t(t) \|_3 +\|\bar v(t)\|_3 + \int_0^t \| \bar v\|_3\Bigr)\,, \\
\bigl\| \int_0^t \bar{\mathfrak C}\bigr\|_{3} & \le  \mathcal{N} _0 +   C\Bigl( T \|\bar v(t) \|_4  + \int_0^t \| \bar v\|_4\Bigr) \,.
\end{align} 
\end{subequations}

\begin{remark} 
The function $\bar c(t)$
 is added to the right-hand side of (\ref{defvc}) to ensure that the solvability condition for the elliptic system (\ref{defv}) is satisfied;
in particular, the solvability condition is obtained from an application of the divergence theorem to equation (\ref{defva}).
\end{remark} 

\begin{remark} 
Condition (\ref{defvd}) is only necessary because of the periodicity of our domain in the directions $e_1$ and $e_2$.  In particular,
our elliptic system is defined modulo a constant vector, and the addition of $\bar c(t)N^3$ to the right-hand side of (\ref{defvc}) fixes
the constant in the vertical direction, while the condition (\ref{defvd}) fixes the two constants in the tangential directions.  The particular
choice for the average of $v_t^ \alpha $, $ \alpha =1,2$,  permits us to close the fixed-point argument, and obtain the unique solution of (\ref{approx}).
\end{remark}


\subsection{Construction of solutions  and regularity theory for $\bar X$ and its time derivatives}
This section will be devoted to the proof of Proposition \ref{prop_X}.
\subsection{Smoothing $\bar v$}
We will proceed with a two stage process.  First, we smooth $\bar v$ and obtain strong solutions to the linear equation (\ref{ssX})
in the case that the forcing function $\bar G$ and the coefficient matrix $\bar B^{jk}$ are $C^ \infty(\overline \Omega)$-functions.
Second, having strong solutions to (\ref{ssX}), with bounds that depend on the smoothing parameter of $\bar v$, we  use nonlinear
estimates to conclude the proof of Proposition \ref{prop_X}.

Using the notation of  Section \ref{subsec::initdata}, for each $t \in[0,T_ \kappa ]$ and for $\nu > 0$, we define
$$
\bar v^\nu( \cdot, t) = \varrho_\nu * \mathcal{E} _ \Omega ( \bar v ( \cdot , t)) \,,
$$
so that for each $ \nu >0$, $\bar v^ \nu ( \cdot , t) \in C^ \infty (\overline \Omega )$.   We define $ \bar G^ \nu$ by replacing
$\bar A$,   $\bar a$, $\bar J$, and $\bar v$ in (\ref{Gforce}) with $\bar A^ \nu$, $\bar a^\nu$, $\bar J^ \nu$, and $\bar v^\nu$, respectively.
The quantities $\bar A^ \nu$, $\bar a^\nu$, $\bar J^ \nu$ are defined just as their unsmoothed analogues from the map $\bar \eta^\nu = e +
\int_0^t \bar v^\nu$.  We also define $ [\bar B^\nu]^{jk}  =( \bar a^\nu)^j_i ( \bar a^\nu)^k_i  $; according to (\ref{pos_def}), we can choose
$\nu>0$ sufficiently small so that for $t \in [0,T_ \kappa ]$,
\begin{equation}\label{pos_def2}
 \lambda | \xi |^2 \le  [\bar B^\nu]^{jk}(x,t) \xi _j \xi _k  \ \ \forall \ \xi \in \mathbb{R}^3  \,, x \in \Omega \,.
\end{equation} 

Until Section \ref{sec836}, we will use $\bar B^\nu$ and $\bar G^\nu$ as the coefficient matrix and forcing function, respectively,
but for notational convenience we will not explicitly write the superscript $\nu$.

\subsubsection{$L^2(0,T;\h)$  regularity for $\bar X_{ttt}$}\label{sec832}
We will use the definition of the constant $ \lambda >0$ given in (\ref{pos_def}).

\begin{definition}[Weak Solutions of (\ref{ssX})]  A function $ \bar X  \in L^2(0,T; \h)$ with
$\frac{ \bar X _t}{\rho_0} \in H ^{-1} ( \Omega )$  is a weak solution of (\ref{ssX}) if
\begin{itemize}
\item[(i)] for all $ \mathcal{W} \in \h$, 
\begin{equation}\label{weakX}
\langle \frac{\bar J^3 \bar X _t}{\rho_0} \,, \mathcal{W} \rangle  + 2 \kappa \int_ \Omega \frac{ \bar B^{jk}}{\rho_0} (\rho_0 \bar X)  ,_k \mathcal{W} ,_j dx
= \langle \bar G, \mathcal{W} \rangle  \ \ \text{ a.e. } [0,T]\,,
\end{equation} 
\item[(ii)]$\bar X (0) = X _0 \,.$
\end{itemize}
The duality pairing between $\h$ and $ H ^{-1} ( \Omega )$ is denoted by $\langle \cdot , \cdot \rangle $, and $\bG\in L^2(0,T;H^{-1}(\Omega)) $.
\end{definition} 

Recall that if $ \bG \in H^{-1}(\Omega) $, then 
$\| \bG \|_{ H^{-1}(\Omega) } = 
\sup\{ \langle \bG, \mathcal{W}  \rangle \ | \ \mathcal{W} \in  \h\,, \| \mathcal{W} \|_{ \h} =1 \}$.
Furthermore, there exist functions $\bG_0,\bG_1,\bG_2,\bG_3$ in $ L^2(\Omega) $ such that $ \langle \bG, \mathcal{W} \rangle= \int_ \Omega 
\bG_0 \mathcal{W}  +  \bG_i \mathcal{W} ,_i\, dx $, so that $\| \bG \|_{ H^{-1}(\Omega) }^2 = \inf \sum_{a=0}^3\|  \bG_a\|^2_0$, the
infimum being taken over all such functions $\bG_a$.

\begin{lemma}\label{X_weak}
If  $\bG \in L^2(0,T;H^{-1}(\Omega))   $ and $\frac{X _0}{\sqrt{\rho_0}} \in L^2(\Omega) $, then for $T>0$ taken sufficiently small so that (\ref{pos_def}) holds, there exists a unique weak
solution to (\ref{ssX}) such that for constants $C_p>0$ and $C_{ \kappa \lambda } >0$, 
\begin{align} 
 &\left\| \frac{\bar X_t}{\rho_0}\right\|^2_{L^2(0,T; H^{-1}(\Omega) )} +
\sup_{t \in [0,T]} C \left\| \frac{\bar X(t)}{\sqrt{\rho_0}} \right\|^2_0  +C_p  \left\| \bar X\right\|^2_ {L^2(0,T; \h)} \n  \\
& \qquad \qquad  \qquad \qquad 
\le  \left\| \frac{X_0}{\sqrt{\rho_0}} \right\|^2_0 +  C_{ \kappa \lambda }  \left\| \bG\right\|^2_{ L^2(0,T;H^{-1}(\Omega)) }  \,.\n
\end{align} 
\end{lemma} 
\begin{proof}
Let $(e_n)_{n\in\N}$ denote a Hilbert basis of $\h$, with each $e_n$ being smooth. Such a choice of basis is indeed possible as we can take for instance the eigenfunctions of the Laplace operator on $\Omega$ with vanishing Dirichlet boundary conditions on $\Gamma$ and $1$-periodic in $e_1$ and $e_2$. We then define the Galerkin approximation at order $n\ge 1$ of (\ref{weakX}) as being under the form $X _n=\sum_{i=0}^n \lambda_i^n(t) e_i$ such that: $ \forall \ell\in\{0,...,n\}$,
\begin{subequations}
\label{divn}
\begin{align}
\bigl( \bar J^3 \frac{{X _n}_t}{\rho_0},  e_\ell\bigr)_{ L^2(\Omega) }  
+2 \kappa  \bigl(\frac{\bar B^{jk} }{\rho_0} (\rho_0 X _n),_k \,,   e_\ell,_j \bigr)_{L^2(\Omega )} &=(\bG_0 ,\ e_\ell)_{L^2(\Omega )}
- (\bG_i , \frac{ \p e_\ell}{\p x_i})_{L^2(\Omega )}
\ \hbox{in}\ [0,T],\\
\lambda_\ell^n(0)&=(X_0,e_\ell)_{L^2(\Omega )}.
\end{align}
\end{subequations}
Since each $e_\ell$ is in $H^{k+1}(\Omega )\cap\h$ for every $k \ge 1$, we have by our high-order Hardy-type inequality (\ref{Hardy}) that 
$$\frac{e_\ell}{\rho_0}\in H^k( \Omega )\ \text{ for } k\ge 1 \,;$$
therefore, each integral written in (\ref{divn}) is well-defined. 

Furthermore, as the $e_\ell$ are linearly independent, so are the $\frac{e_\ell}{\sqrt\rho_0}$ and therefore the determinant of the matrix 
$$
\Bigl[\bigl( \frac{e_i}{\sqrt\rho_0},\frac{e_j}{\sqrt\rho_0}\bigr)_{L^2(\Omega )}\Bigr]_{(i,j)\in\N_n=\{1,...,n\}}
$$ 
is nonzero. This implies that our finite-dimensional Galerkin approximation (\ref{divn}) is a well-defined first-order differential system of order $n+1$, which  therefore has a solution on a time interval $[0,T_n]$, where $T_n$ a priori depends on the rank $n$ of
the Galerkin approximation. 
In order to prove that $T_n=T$, with $T$ independent of $n$, we notice that since $X_n$ is a linear combination of the $e_\ell$ ($\ell\in \{1,...,n\}$), we have that on $[0,T_n]$,
\begin{equation}
\n
\left( \frac{\bar J^3 {X _n}_t}{\rho_0},  X_n\right)_{ L^2(\Omega) }  + 2 \kappa  \left(\frac{ \bar B^{jk}}{\rho_0} \frac{ \p(\rho_0 X_n)}{\p x_k} , \frac{\p X _n}{\p x_j} \right)_{L^2(\Omega )} = (\bG_0 ,\ X _n)_{L^2(\Omega )}
- (\bG_i , \frac{ \p X_n}{\p x_i})_{L^2(\Omega )} \,.
\end{equation}
Since 
\begin{align*} 
 2 \kappa  \int_ \Omega  \frac{ \bar B^{jk}}{\rho_0}(\rho_0 X_n),_k \,  X _n,_j \, dx
  & =
2 \kappa  \int_ \Omega  \bar B^{jk} X_n,_k \,  X _n,_j \, dx +   2 \kappa  \int_ \Omega \frac{ \bar B^{jk}}{\rho_0}\rho_0,_k X_n \,  X _n,_j\, dx 
 \end{align*} 
 and
 \begin{align*} 
 2 \kappa \int_ \Omega \frac{ \bar B^{jk}}{\rho_0}\rho_0,_k X_n \,  X _n,_j\, dx 
 & =
 - \kappa   \int_ \Omega \frac{ \rho_0,_{jk} }{\rho_0} \bar B^{jk} |X_n|^2 \, \, dx +  \kappa   \int_ \Omega \frac{ \rho_0,_{k} \rho_0,_j }{\rho_0^2} \bar B^{jk} |X_n|^2 \, \, dx  \\
 & \qquad \qquad  - \kappa  \int_ \Omega \frac{ \rho_0,_{k} }{\rho_0} \bar B^{jk},_j |X_n|^2 \, \, dx \,,
 \end{align*} 
 it follows that on $[0,T_n]$
 \begin{align} 
&{\frac{1}{2}} \frac{d}{dt} \int_ \Omega \bar J^3  \frac{ |X_n|^2}{\rho_0}dx+ 2 \kappa  \int_ \Omega  \bar B^{jk} X_n,_k \,  X _n,_j \, dx 
+  \kappa   \int_ \Omega \frac{ \rho_0,_{k} \rho_0,_j }{\rho_0^2} \bar B^{jk} |X_n|^2 \, \, dx \n \\
&  \qquad = {\frac{1}{2}} \int_ \Omega (\bar J^3)_t \frac{ |X_n|^2}{ \rho_0} dx +  \kappa   \int_ \Omega \frac{ \rho_0,_{jk} }{\rho_0} \bar B^{jk} |X_n|^2 \, \, dx + \kappa  \int_ \Omega \frac{ \rho_0,_{k} }{\rho_0} \bar B^{jk},_j |X_n|^2 \, \, dx \n \\
&  \qquad  \qquad + \int_ \Omega \bar G_0 \, X_n\, dx - \int_ \Omega \bar G_i\, X_n,_i \, dx \,. \n
\end{align} 
Using (\ref{pos_def}), we see that
\begin{align} 
&{\frac{1}{2}} \frac{d}{dt} \int_ \Omega \bar J^3  \frac{ |X_n|^2}{\rho_0}dx+ 2 \kappa  \lambda  \int_ \Omega |D X_n|^2 \, dx 
+  \kappa  \lambda   \int_ \Omega \frac {|D \rho_0|^2}{\rho_0^2} |X_n|^2 \, \, dx  \n \\
&  \qquad  \le    \|  {\frac{1}{2}}   (\bar J^3)_t  +  \kappa  \rho_0,_{jk} \bar B^{jk} +  \kappa \rho_0,_{k} \bar B^{jk},_j \|_{ L^ \infty ( \Omega )}
      \int_ \Omega \frac{1 }{\rho_0}  |X_n|^2 \, \, dx 
+C \| \bar G\|_{H^{-1}(\Omega) } \|D X_n\|_0   \,. \label{need100}
\end{align} 
Using the Sobolev embedding theorem and the Cauchy-Young inequality, we see that
\begin{align*} 
&{\frac{1}{2}} \frac{d}{dt} \int_ \Omega \bar J^3  \frac{ |X_n|^2}{\rho_0}dx+ \kappa  \lambda  \int_ \Omega |D X_n|^2 \, dx 
+  \kappa  \lambda   \int_ \Omega \frac {|D\rho_0|^2}{\rho_0^2} |X_n|^2 \, \, dx \\
&  \qquad  \le C   \|  {\frac{1}{2}}   (\bar J^3)_t  +  \kappa  \rho_0,_{jk} \bar B^{jk} +  \kappa \rho_0,_{k} \bar B^{jk},_j \|_2
      \int_ \Omega \frac{1 }{\rho_0}  |X_n|^2 \, \, dx 
+C_{ \kappa  \lambda } \| \bar G\|_{H^{-1}(\Omega) }^2 \,,
\end{align*} 
where the constant $C_{ \kappa \lambda }$ depends inversely on $ \kappa \lambda $.  Since $\bar v \in \mathcal{C} _T(M)$,
we have that on $[0,T]$,
$$
\int_0^t   \|  {\frac{1}{2}}   (\bar J^3)_t  +  \kappa  \rho_0,_{jk} \bar B^{jk} +  \kappa \rho_0,_{k} \bar B^{jk},_j \|_2dt \le C_M \sqrt{t} 
$$
for a constant $C_M$ depending on $M$,
so that  Gronwall's inequality shows  that $T_n=T$ (with $T$ independent of $n \in \mathbb{N}  $),  and  with $ {\frac{1}{2}} \le \bar J$ for all
$\bar v \in \mathcal{C} _T(M)$, we see that
\begin{equation}\nonumber
\sup_{t \in [0,T]} C \left\| \frac{X _n(t)}{\sqrt{\rho_0}} \right\|^2_0  + \kappa \lambda  \int_0^T  \left\|D X _n(t)\right\|^2_0 \le
\left\| \frac{ X(0)}{\sqrt{\rho_0}} \right\|^2_0 + C_{ \kappa \lambda }  \int_0^T  \left\| \bG(t)\right\|^2_{ H^{-1}(\Omega) } \,.
\end{equation}
Setting $C_p  = \frac{ \kappa \lambda }{\text{Poincar\'{e} constant}}$, we see that
\begin{equation}\nonumber
\sup_{t \in [0,T]}C  \left\| \frac{X _n(t)}{\sqrt{\rho_0}} \right\|^2_0  + C_p  \int_0^T  \left\|X _n(t)\right\|^2_1 \le
\left\| \frac{ X(0)}{\sqrt{\rho_0}} \right\|^2_0 + C_{ \kappa \lambda }  \int_0^T  \left\| \bG(t)\right\|^2_{ H^{-1}(\Omega) } \,.
\end{equation}
Thus, there exists a  subsequence $\{ {X _n}_m\} \subset \{X _n\}$ which converges weakly to some $ \bar X $ in $ L^2(0,T; \h )$, which satisfies
\begin{equation}\nonumber
\sup_{t \in [0,T]} C\left\| \frac{\bar X (t)}{\sqrt{\rho_0}} \right\|^2_0  + C_p  \int_0^T  \left\|\bar X (t)\right\|^2_1 \le
\left\| \frac{ X(0)}{\sqrt{\rho_0}} \right\|^2_0 + C_{ \kappa \lambda }  \int_0^T  \left\| \bG(t)\right\|^2_{ H^{-1}(\Omega) } \,.
\end{equation} 
Furthermore, it can also be shown using standard arguments that this $\bar X$ 
verifies the identity (\ref{weakX}). It then immediately follows that
\begin{equation}\nonumber
\frac{\bar X _t}{\rho_0 }\in L^2(0,T; H ^{-1} (\Omega)) \,,
\end{equation} 
and that $\bar X(0)=X _0$.   Uniqueness follows by letting $ \mathcal{W} =\bar X$ in (\ref{weakX}). 
\end{proof}

Since  $\| \bG \|^2_{ L^2(0,T; L^2(\Omega) )} \le P(\|\bar v\|^2_\XT )$,
  it thus follows from Lemma \ref{X_weak} and (\ref{N0}) that
\begin{equation}\label{Xest1}
 \left\| \frac{\bar X_t}{\rho_0}\right\|^2_{L^2(0,T; H^{-1}(\Omega) )} +
\sup_{t \in [0,T]}  \left\| \frac{\bar X(t)}{\sqrt{\rho_0}} \right\|^2_0  +C_p  \left\| \bar X\right\|^2_ {L^2(0,T; \h)}  
\le C\,.
\end{equation}

In order to build regularity for $X$, we construct  weak solutions for the time-differentiated version of (\ref{ssX}).   It is convenient to proceed 
from the first to third time-differentiated problems.
We begin with the first  time-differentiated version of (\ref{ssX}):
\begin{subequations}
\label{ssXt}
\begin{alignat}{2}
\frac{ \bar J^3 \bar X _{tt}}{\rho_0} - 2\kappa  \Bigl[\frac{\bar B^{jk}}{\rho_0} (\rho_0 \bar X_t)  ,_k\Bigr],_j 
&=\bar G_t + \mathcal{G} _1 &&  \text{ in }   \Omega \times (0,T_{ \kappa } ] \,, \label{ssXt.a}\\
 \bar X_t & =0 &&  \text{ on }   \Gamma \times (0,T_{ \kappa  } ] \,, \label{ssXt.b}\\
\bar X_t & = X_1\  &&  \text{ on }   \Omega  \times \{t=0\} \,,\label{ssXt.c}
\end{alignat} 
\end{subequations}
where the initial condition $X_1$ is given as
\begin{equation}\label{X1}
X _1= 2 \kappa \rho_0\Bigl[\frac{(\rho_0 X_0)  ,_i}{\rho_0}\Bigr],_i   + \rho_0 \bG(0) \,,
\end{equation} 
 the additional forcing term  $ \mathcal{G} _1$ is defined by
\begin{equation}\label{G1}
\mathcal{G} _1 =  2\kappa   \Bigl[\bar B_t^{jk} \rhoi ( \rho_0 \bar X) ,_k\Bigr],_j - \frac{ (\bar J^3)_t \bar X _{t}}{\rho_0} \,,
\end{equation} 
$X_0 = \rho_0 \operatorname{div} u_0$, and
$$
\bG(0) = -2 \kappa \operatorname{curl} \operatorname{curl} u_0 \cdot D \rho_0 - 2 \kappa \operatorname{div} u_0\, \Delta \rho_0
+ 2 \kappa u_0^j,_i \rho_0,_{ij} -2 \Delta \rho_0 -2 ( \operatorname{div} u_0)^2 -u_0^i,_j u_0^j,_i \,.
$$
According to the estimate (\ref{Xest1}),  $\| \bG_t  \mathcal{G} _1 \|^2_{ L^2(0,T; H^{-1}(\Omega) )} \le C$;  hence  by Lemma \ref{X_weak} (with $\bar X_t$, $\bG_t+ \mathcal{G} _1$, $X_1$ replacing $\bar X$, $\bG$, $X_0$, respectively),
\begin{equation}\label{Xest2}
 \left\| \frac{\bar X_{tt}}{\rho_0}\right\|^2_{L^2(0,T; H^{-1}(\Omega) )} +
\sup_{t \in [0,T]}  \left\| \frac{\bar X_t(t)}{\sqrt{\rho_0}} \right\|^2_0  +C_p  \left\| \bar X_t\right\|^2_ {L^2(0,T; \h)}  
\le C\,.
\end{equation} 

Next, we consider  the  second time-differentiated version of (\ref{ssX}):
\begin{subequations}
\label{ssXtt}
\begin{alignat}{2}
\frac{ \bar J^3 \bar X _{ttt}}{\rho_0} - 2\kappa  \Bigl[\frac{\bar B^{jk}}{\rho_0} (\rho_0 \bar X_{tt})  ,_k\Bigr],_j 
&= \bG_{tt} +  \mathcal{G} _2 &&  \text{ in }   \Omega \times (0,T_{ \kappa } ] \,, \label{ssXtt.a}\\
 \bar X_{tt} & =0 &&  \text{ on }   \Gamma \times (0,T_{ \kappa  } ] \,, \label{ssXtt.b}\\
\bar X_{tt} & = X_2\  &&  \text{ on }   \Omega  \times \{t=0\} \,,\label{ssXtt.c}
\end{alignat} 
\end{subequations}
where the initial condition $X_2$ is given as
\begin{equation}\label{X2}
X _2= 2 \kappa \rho_0\Bigl[\frac{(\rho_0 X_1)  ,_i}{\rho_0}\Bigr],_i   + \rho_0  \mathcal{G} _1(0) \,,
\end{equation} 
 the forcing function $ \mathcal{G} _2$ is defined by
\begin{equation}\label{G2}
\mathcal{G} _2 =  \p_t \mathcal{G} _1 + 2\kappa   \Bigl[\bar B_t^{jk} \rhoi ( \rho_0 \bar X_t) ,_k\Bigr],_j - \frac{ (\bar J^3)_t \bar X _{tt}}{\rho_0} \,.
\end{equation} 
We do not precise $\mathcal{G} _1(0)$, but note that its highest-order terms scale like either $D^3 u_0$ or $\rho_0 D^4 u_0$ or $D^3 \rho_0$,
so that $\| \sqrt{\rho_0} \mathcal{G} _1(0)\|_0^2  \le \mathcal{N} _0$.
Using the estimate (\ref{Xest2}), we see that $\bG_{tt} + \|\mathcal{G} _2 \|^2_{ L^2(0,T; H^{-1}(\Omega) )} \le C$. 
It thus follows from Lemma \ref{X_weak} (with $\bar X_{tt}$, $\bar G_{tt}+  \mathcal{G} _2$, $X_2$ replacing $X$, $\bG$, $X_0$, respectively),
\begin{align} 
 \left\| \frac{\bar X_{ttt}}{\rho_0}\right\|^2_{L^2(0,T; H^{-1}(\Omega) )} +
\sup_{t \in [0,T]}  \left\| \frac{\bar X_{tt}(t)}{\sqrt{\rho_0}} \right\|^2_0  +C_p  \left\| \bar X_{tt}\right\|^2_ {L^2(0,T; \h)}  \le C \,.
\label{Xest3}
\end{align} 

Finally, we consider  the  third time-differentiated version of (\ref{ssX}):
\begin{subequations}
\label{ssXttt}
\begin{alignat}{2}
\frac{ \bar J^3 \bar X _{tttt}}{\rho_0} - 2\kappa  \Bigl[\frac{\bar B^{jk}}{\rho_0} (\rho_0 \bar X_{ttt})  ,_k\Bigr],_j 
&=\bar G_{ttt}+ \mathcal{G} _3 &&  \text{ in }   \Omega \times (0,T_{ \kappa } ] \,, \label{ssXttt.a}\\
 \bar X_{ttt} & =0 &&  \text{ on }   \Gamma \times (0,T_{ \kappa  } ] \,, \label{ssXttt.b}\\
\bar X_{ttt} & = X_3\  &&  \text{ on }   \Omega  \times \{t=0\} \,,\label{ssXttt.c}
\end{alignat} 
\end{subequations}
where the initial condition $X_3$ is given as
\begin{equation}\label{X3}
X _3= 2 \kappa \rho_0\Bigl[\frac{(\rho_0 X_2)  ,_i}{\rho_0}\Bigr],_i   + \rho_0  \mathcal{G} _2(0) \,,
\end{equation} 
 the forcing function $ \mathcal{G} _3$ is defined by
\begin{equation}\label{G3}
\mathcal{G} _3 =   \p_t \mathcal{G} _2 +  2\kappa   \Bigl[\bar B_t^{jk} \rhoi ( \rho_0 \bar X_{tt}) ,_k\Bigr],_j    -  \frac{ (\bar J^3)_t  \bar X _{ttt} }{\rho_0} \,.
\end{equation} 
Once again, we do not precise $\mathcal{G} _2(0)$, but note that its highest-order terms scale like either $D^4 u_0$ or $\rho_0 D^5 u_0$ or $D^4 \rho_0$,
so that $\| \sqrt{\rho_0} \mathcal{G} _2(0)\|_0^2  \le \mathcal{N} _0$.
Using the estimate (\ref{Xest3}), we see that $\|\bar G_{ttt}+ \mathcal{G} _3 \|^2_{ L^2(0,T; H^{-1}(\Omega) )}$ is
bounded by a constant $C$.  (Note that this constant $C$ crucially depends on $\nu >0$.)
We see that 
Lemma \ref{X_weak} (with $X_{ttt}$, $\bG_{ttt}+ \mathcal{G} _3$, $X_3$ replacing $X$, $\bG$, $X_0$, respectively) yields the following
estimate:
\begin{align} 
\left\| \frac{\bar X_{tttt}}{\rho_0}\right\|^2_{L^2(0,T; H^{-1}(\Omega) )} +
\sup_{t \in [0,T]}  \left\| \frac{\bar X_{ttt}(t)}{\sqrt{\rho_0}} \right\|^2_0  +C_p  \left\|\bar X_{ttt}\right\|^2_ {L^2(0,T; \h)}   \le C
\,.  \label{Xest4}
\end{align}

\subsubsection{$L^2(0,T;H^2(\Omega))$  regularity for $\bar X_{tt}$}\label{sec833} 
We expand the time-derivative in the 
 definition of $ \mathcal{G} _2$ in (\ref{G2}) and write
 \begin{equation}\label{G22}
 \mathcal{G} _2 = 4 \kappa \Bigl[ \bar B^{jk}_t \rhoi \bigl( \rho_0 \bar X_t),_k\Bigr],_j + 2 \kappa \Bigl[ \bar B^{jk}_{tt} \rhoi \bigl( \rho_0 \bar X),_k\Bigr],_j - 2 \frac{(\bar J)^3_t \bar X_{tt}}{\rho_0} -  \frac{(\bar J)^3_{tt} \bar X_{t}}{\rho_0} \,.
 \end{equation} 
 According to the estimates  (\ref{Xest4}),  together with the Hardy inequality and the smoothness of $\bar v$,
 $$
\int_0^T \left(\| \bar G_{tt}\|_0^2+ \left\| \frac{(\bar J)^3_{t} \bar X_{ttt}}{\rho_0}\right\|_0^2  + \left\|2 \frac{(\bar J)^3_t \bar X_{tt}}{\rho_0}\right\|_0^2 + \left\| \frac{(\bar J)^3_{tt} \bar X_{t}}{\rho_0} \right\|_0^2\right)dt  \le C \,;
 $$
 hence, (\ref{ssXtt.a}) and (\ref{G22}) show that

\begin{align} 
 \kappa ^2\Bigl\|  \p_{tt} \Bigl[\frac{\bar B^{jk}}{\rho_0} (\rho_0 \bar X)  ,_k\Bigr],_j  \Bigr\|^2_{L^2(0,T; L^2(\Omega) )}
 \le C\,. \label{fest01}
\end{align} 
The bound (\ref{fest01}) together with the fundamental theorem of calculus then provides the bound on $[0,T]$:
\begin{equation}\label{fest10}
\| [\frac{\bar B^{jk}}{\rho_0} (\rho_0 \bar X)  ,_k],_j \|_0^2 + \| \p_t[\frac{\bar B^{jk}}{\rho_0} (\rho_0 \bar X)  ,_k],_j \|_0^2 \le C \,.
\end{equation} 
From this bound, we will infer that $\|\bar X\|_2^2 \le C$  and that $\|\bar X_t\|_2^2 \le C$, and finally that $\int_0^T \|\bar X_{tt}\|_2^2dt \le C$.
We will begin this analysis by estimating horizontal derivatives of $D\bar X$.

\begin{definition} [Horizontal difference quotients]  For $h>0$, we set 
$$\bar \p_ \alpha ^h u (x) = \frac{ u(x+ h e_ \alpha ) - u(x)}{h} \ \ \ ( \alpha =1,2) \,,
$$
and  $ \bar \p^h = (\bar \p^h_1, \bar \p^h_2)$.
\end{definition} 

The variational form of the statement that 
$$ \|[\frac{\bar B^{jk}}{\rho_0} (\rho_0 \bar X)  ,_k],_j \|_0^2 \text{ is bounded}$$
 takes the following form:  almost everywhere on $[0,T]$ and for $f(t)$ bounded  in $L^2 (\Omega)$,
\begin{equation}\label{sstemp03}
 2 \kappa \int_ \Omega \bar B^{jk} \bar X  ,_k \ \phi,_j \, dx + 2\kappa \int_ \Omega  \bar B^{jk} \frac{\rho_0,_k}{\rho_0} \bar X \, \phi,_j \, dx
= \int_ \Omega f \, \phi \, dx \ \ \forall \phi \in \h \,.
\end{equation} 

We substitute $\phi = - \bar \p^{-h} \bar \p^h \bar X$ into (\ref{sstemp03}), and using the discrete product rule
$$
\bar \p^h_ \alpha ( pq) = p^h \bar \p^h_ \alpha q +  \bar \p^h_ \alpha p \, q \,,  \ \ \  p^h(x) = p( x + h e_ \alpha ) \,,
$$
to find that
\begin{align*} 
& \underbrace{ 2 \kappa \int_ \Omega \bar B^{jk,h}\, \bar \p^h_ \alpha \bar X,_k \ \bar \p^h_ \alpha\bar  X,_j  \, dx }_{ \mathfrak{i} _1}
+ \underbrace{ 2 \kappa \int_ \Omega \bar \p^h_ \alpha \bar B^{jk} \, \bar X,_k \ \bar \p^h_ \alpha \bar X,_j  \, dx }_{ \mathfrak{i} _2} \\
& \qquad + \underbrace{2 \kappa \int_ \Omega \bar B^{jk,h} \rho_0,_k^h\, \bar \p^h_ \alpha\Bigl( \frac{\bar X}{ \rho_0} \Bigr)\ \bar \p^h_ \alpha \bar X,_j  \, dx}_{ \mathfrak{i} _3}  
 + \underbrace{ 2 \kappa \int_ \Omega\bar \p^h_ \alpha[ \bar B^{jk} \rho_0,_k] \,
  \frac{ \bar X}{ \rho_0} \ \bar \p^h_ \alpha \bar X,_j \, dx }_{ \mathfrak{i} _4} \\
 &\qquad\qquad\qquad\qquad = \underbrace{-   \int_ \Omega f \, \bar \p^{-h} \bar \p^h \bar X \, dx}_{ \mathfrak{i} _5} \,.
\end{align*} 
We proceed to the analysis of the integrals $\mathfrak{i} _a$, $a=1,...,5$.  By the uniform ellipticity condition (\ref{pos_def2}) obtained on
our time interval $[0,T]$, we see that
\begin{equation}\label{ell1}
2 \kappa \lambda \| \bar \p^h D \bar X\|^2_0 \le \mathfrak{i}_1  \,.
\end{equation} 
The integral $ \mathfrak{i} _2$ can be estimated using the $L^ \infty ( \Omega )$-$ L^2(\Omega) $- $ L^2(\Omega) $ H\"{o}lder's inequality:
\begin{align} 
\mathfrak{i} _2 &  \le  2 \kappa \| \bar \p^h_ \alpha \bar B^{jk} \|_{ L^ \infty ( \Omega )}  \,  \| \bar X_{tt},_k\|_0  \, \| \bar \p^h_ \alpha \bar X_{tt},_j  \|_0 \n 
\end{align} 
We then see that using
 the Cauchy-Young inequality for $ \epsilon >0$ we obtain
\begin{equation}\label{ell2}
\mathfrak{i} _2 \le  C \|\bar X\|_1^2  + \epsilon \| \bar \p^h_ \alpha \bar X,_j  \|_0^2 \,.
\end{equation}
The integral $ \mathfrak{i} _3$ requires us to form an exact derivative and integrate by parts.   With
\begin{align} 
\bar \p^h_ \alpha\Bigl( \frac{ \bar X}{ \rho_0} \Bigr) & = \bar \p_ \alpha ^h( \rho_0 ^{-1} ) \bar X^h + \rho_0 ^{-1} \bar \p^h _ \alpha \bar X  
 =-\frac{1}{\rho_0} \bar \p^h_ \alpha  \rho_0 \frac{\bar X^h }{\rho_0^h} + \frac{ \bar \p^h_ \alpha  \bar X}{ \rho_0} \,, \label{nice}
\end{align} 
 we write $ \mathfrak{i} _3$ as
\begin{align*} 
\mathfrak{i} _3 =
\underbrace{2 \kappa \int_ \Omega \frac{ \bar B^{jk,h} \rho_0,_k^h}{\rho_0}\,       \bar \p^h_ \alpha  \bar X     \,\bar \p^h_ \alpha \bar X,_j  \, dx}_{ 
 { \mathfrak{i} _3}_a  } 
-
\underbrace{2 \kappa \int_ \Omega \bar B^{jk,h} \rho_0,_k^h\,   (  \frac{1}{\rho_0} \bar \p^h_ \alpha  \rho_0) \, \frac{\bar X^h }{\rho_0^h}         \,\bar \p^h_ \alpha \bar X,_j  \, dx}_{  { \mathfrak{i} _3}_b  }  \,.
\end{align*} 
Integration by parts with respect to $x_j$ shows that
\begin{align*} 
{ \mathfrak{i} _3}_a &= 
\underbrace{ \kappa \int_ \Omega  \bar B^{jk,h} \rho_0,_k^h \rho_0,_j  \,      \left|\frac{\bar \p^h  \bar X}{\rho_0}    \right|^2 \, dx}
_{ {{\mathfrak{i} _3}_a}_i    }
- \underbrace{ \kappa \int_ \Omega ( \bar B^{jk,h} \rho_0,_k^h),_j\,      \frac{\bar \p^h  \bar X}{{\rho_0}}    \bar \p^h \bar X \, dx}
_{ {{\mathfrak{i} _3}_a}_{ii}} \,.
\end{align*} 
Next, we  notice that
\begin{equation*}
{{\mathfrak{i} _3}_a}_i=\kappa \int_ \Omega  \bar B^{jk,h} \rho_0,_k \rho_0,_j  \,      \left|\frac{\bar \p^h  \bar X}{\rho_0}    \right|^2 dx +
\kappa \int_ \Omega  \bar B^{jk,h} (\rho_0,_k^h-\rho_0,_k) \rho_0,_j  \,      \left|\frac{\bar \p^h  \bar X}{\rho_0}    \right|^2
\, dx
\end{equation*}
and therefore, according to (\ref{pos_def2}),
\begin{align}\label{ell3ai}
{{\mathfrak{i} _3}_a}_i&\ge -C\, h\, \kappa \| D^2\rho_0\|_{L^\infty} C_M \int_\Omega \left|\frac{\bar \p^h  \bar X}{\rho_0}    \right|^2\, dx +\kappa \lambda \int_ \Omega   |D \rho_0|^2  \,  \left|\frac{\bar \p^h \bar X}{\rho_0}    \right|^2  \, dx   \,\nonumber\\
&\ge -C\, h\, \kappa  C \left\|{\bar \p^h  \bar X}    \right\|_1^2 +\kappa \lambda \int_ \Omega   |D \rho_0|^2  \,  \left|\frac{\bar \p^h \bar X}{\rho_0}    \right|^2  \, dx
\end{align}
On the other hand, for $ \epsilon >0$, 
\begin{align} 
 {{\mathfrak{i} _3}_a}_{ii}  & \le    \kappa  \| ( \bar B^{jk,h} \rho_0,_k^h ),_j\|_{ L^ \infty} \, \| \frac{\bar \p^h \bar X_{tt}}{{\rho_0}}  \|_0 \, 
 \|  \bar \p^h \bar X_{tt} \|_0  \n \\
& \le C    \|  \bar \p^h \bar X \|_0^2   +\epsilon     \| \frac{\bar \p^h  \bar X}{{\rho_0}}  \|_0^2 \n \\
& \le C   \|  \bar \p^h \bar X \|_0^2   +\epsilon   C  \| \bar \p^h D  \bar X  \|_0^2 \,,
 \label{ell3aii}
\end{align}
where we have used the Hardy and Poincar\'{e} inequalities for the last inequality in (\ref{ell3aii}).      Furthermore,
\begin{align} 
{ \mathfrak{i} _3}_b &\le C
\|\bar B^{jk,h} \rho_0,_k^h\,   (  \frac{1}{\rho_0} \bar \p^h_ \alpha  \rho_0)\|_{ L^ \infty (\Omega )} \,  \| \frac{ \bar X^h }{\rho_0^h}  \|_0       \,
\| \bar \p^h_ \alpha X,_j \|_0 \n \\
& \le C_  \|   \frac{\bar \p^h_ \alpha  \rho_0}{\rho_0} \|_2 \, \| \bar X\|_1 \, \| \bar \p^h_ \alpha \bar X,_j \|_0 \n \\
& \le C  \|     \rho_0\|_4 \, \|\bar X\|_1 \, \| \bar \p^h_ \alpha \bar X,_j \|_0 \n \\
& \le C\| X\|_1^2 + \epsilon   \| \bar \p^h_ \alpha \bar X,_j \|_0^2 \,.
\label{ell3b}
\end{align} 
Similarly, we have that
\begin{align} 
\mathfrak{i} _4 & \le C   \|  \bar \p^h_ \alpha[ \bar B^{jk} \rho_0,_k] \|_{L^ \infty ( \Omega )} \,
\|  \frac{ \bar X}{ \rho_0} \|_0 \, \| \bar \p^h_ \alpha\bar X,_j \|_0  \n \\
& \le C\|\bar X\|_1^2 +\epsilon   \| \bar \p^h_ \alpha \bar X,_j \|_0^2 \,,
\label{ell4}
\end{align} 
and finally
\begin{align} 
\mathfrak{i} _5 & \le C\| f\|_0^2 +  \epsilon   \| \bar \p^h_ \alpha \bar X,_j \|_0^2
\label{ell5}
\end{align} 
Combining the estimates 
(\ref{ell1})--(\ref{ell5}),  and taking $  \epsilon  >0$ sufficiently
small, we find that for $h>0$ small enough
$$
 \| \bar \p^h D\bar X\|^2_0 
\le    C(\| \bar X_{tt}\|_1^2  + \| f\|_0^2) \le C \,,
$$
with $C$ independent of $h$.  It thus follows that
\begin{equation}
 \|  \bar \p  \bar X\|^2_1 \le  C  \,. \label{Xtan_est}
\end{equation}

Since $\rho_0$ is strictly positive on any open interior subdomain of $ \Omega $, standard regularity theory shows that solutions
$\bar X$ (and its time derivatives) are smooth in the interior, and hence the equation (\ref{ssX}) holds in the classical sense in the
interior of $\Omega$.
It remains to estimate $  \|  \bar X,_3\|^2_{ L^2(0,T; H^1(\Omega) )}$ and for this purpose we expand
$ \Bigl(\frac{\bar B^{jk}}{\rho_0} (\rho_0 \bar X)  ,_k\Bigr),_j$ to find that
\begin{align} 
\bar X,_{33} + \rho_0,_3 \Bigl( \frac{\bar X}{\rho_0}\Bigr),_3 & = \frac{1}{\bar B^{33}} \Bigl( \Bigl[\frac{\bar B^{jk}}{\rho_0} (\rho_0 \bar X)  ,_k\Bigr],_j -\bar B^{ \alpha 3} \bar X,_{3 \alpha }
- \bar B^{3j},_j \bar X,_3 - (\bar B^{j \alpha } \bar X,_ \alpha ),_j \n \\
&\qquad \qquad - ( \bar B^{jk} \rho_0,_k),_j \frac{\bar X}{\rho_0} - \bar B^{ j \alpha } \rho_0,_j \Bigl( \frac{\bar X}{\rho_0} \Bigr),_ \alpha 
- \bar B^{  \alpha 3} \rho_0,_\alpha \Bigl( \frac{\bar X}{\rho_0} \Bigr),_ 3 \Bigr) \,. \label{need2}
\end{align} 
Since
$$
 \bar B^{  \alpha 3} \rho_0,_\alpha \Bigl( \frac{\bar X_{tt}}{\rho_0} \Bigr),_ 3 =  \bar B^{  \alpha 3}  \frac{\rho_0,_ \alpha }{\rho_0} \Bigl(
 \bar X,_3 - \frac{\bar X}{\rho_0} \Bigr) \,,
$$
and since 
$\| \rho_0,_ \alpha / \rho_0 \|_{ L^ \infty ( \Omega ) }$ is bounded by a constant thanks to our higher-order Hardy-type inequality
Lemma \ref{Hardy}, we see from (\ref{Xtan_est}) and (\ref{Xest3}) that
\begin{equation}\label{s701}
\| \bar X,_{33} + \rho_0,_3 \Bigl( \frac{\bar X}{\rho_0}\Bigr),_3  \|^2_0  \le C \,.
\end{equation}

We introduce the variable $Y$ defined by
\begin{equation}
Y(t,x_1,x_2,x_3)=\int_0^{x_3} \frac{\bar X(t,x_1,x_2,y_3)}{\rho_0(x_1,x_2,y_3)} dy_3 \,, \label{varY}
\end{equation}
so that $Y$ vanishes at $x_3=0$, and will allow us to employ the Poincar\'{e} inequality with this variable.
It is easy to see that
\begin{equation}
\bar X=\rho_0 Y,_3 \,. \label{XttY}
\end{equation}
Thanks to the standard Hardy inequality, we thus have that
\begin{equation*}
\|Y,_3\|^2_0\le C \|\bar X\|^2_1 \le  C \,.
\end{equation*}
The estimate (\ref{Xtan_est}) then shows that
\begin{equation*}
\|DY\|^2_0  \le     C  \,,
\end{equation*}
and hence by Poincar\'e's inequality,
\begin{equation}
\|Y\|^2_ 0  \le     C   \,, \label{ssg1}
\end{equation}
We notice that
\begin{align*} 
 \bar X,_{33} + \rho_0,_3 \Bigl( \frac{\bar X}{\rho_0}\Bigr),_3 & = (\rho_0 Y,_3),_{33} + \rho_0,_3 Y,_{33} \\
 & =\rho_0 Y,_{333} + 3 \rho_0,_3 Y,_{33} + \rho_0,_{33} Y,_3 \,.
\end{align*} 
so that
\begin{equation*} 
\| \rho_0 Y,_{333} + 3 \rho_0,_3 Y,_{33} \|^2_0   \le     C \,.
\end{equation*}
The product rule then implies that
\begin{equation}\n
\| (\rho_0 Y),_{333} \|^2_0  \le     C   \,.
\end{equation} 
The same $ L^2(\Omega) $ can easily be established for the lower-order terms  $\rho_0Y$,
$(\rho_0 Y),_{3}$, and $(\rho_0 Y),_{33}$; for instance the identity (\ref{XttY}) shows that
$$ \| (\rho_0 Y,_{3}),_3 \|^2_0   \le     
 C \,. $$
By (\ref{ssg1}), we see that $ \| \rho_0 Y,_{33} \|^2_0$ enjoys the same bound.
Since
$$
(\rho_0 Y),_{33} = \rho_0 Y,_{33} + 2 \rho_0,_3 Y,_3 + \rho_0,_{33} Y
$$
 the estimate (\ref{ssg1}) proves that $ \| (\rho_0 Y),_{33} \|^2_0 \le C $.
By definition of the $H^3(0,1)$-norm, we then see that
\begin{equation}\label{ssg3}
\int_{ \mathbb{T}  ^2}  \| \rho_0 Y ( x_1,x_2, \cdot )\|^2_{H^3(0,1)} dx_1 dx_2  \le     
C   \,.
\end{equation} 

Now, thanks to our high-order Hardy-type inequality set on the 1-D domain $(0,1)$, we infer from  (\ref{ssg3}) that
\begin{equation}\label{ssg4}
\int_{ \mathbb{T}  ^2}  \| Y ( x_1,x_2, \cdot )\|^2_{H^2(0,1)} dx_1 dx_2  \le     
C   \,.
\end{equation} 
From (\ref{ssg3}), 
\begin{equation}\n
 \int_{ \mathbb{T}  ^2}  \| \rho_0 Y,_3 ( x_1,x_2, \cdot ) + \rho_0,_3 Y(x_1,x_2, \cdot )\|^2_{H^2(0,1)} dx_1 dx_2  \le     
C   \,,
\end{equation} 
from which it follows, with (\ref{XttY}), that
\begin{equation}\n
 \int_{ \mathbb{T}  ^2}  \| \bar X ( x_1,x_2, \cdot ) \|^2_{H^2(0,1)} dx_1 dx_2  \le     C   \,;
\end{equation}
hence,
\begin{equation}
\|\bar X,_{33}\|^2_ 0  \le     
C   \,. \label{ssg5}
\end{equation}
Combining the inequalities (\ref{ssg5}) and (\ref{Xtan_est}), we see that
\begin{equation}
\|\bar X\|^2_ 2  \le     
C    \,. \label{XttH2}
\end{equation}

The estimate (\ref{XttH2}) together with (\ref{fest10}) then shows that
$$ \|[\frac{\bar B^{jk}}{\rho_0} (\rho_0 \bar X_t)  ,_k],_j \|_0^2 \text{ is bounded.}$$
By identically repeating for $\bar X_t$ the $H^2( \Omega )$-regularity estimates that we just detailed for $\bar X$, we obtain that
\begin{equation}
\|\bar X_t\|^2_ 2  \le     
C    \,. \label{XttH2b}
\end{equation}
The estimates (\ref{XttH2}) and (\ref{XttH2b}) together with (\ref{fest01}) then prove that
$$ \|[\frac{\bar B^{jk}}{\rho_0} (\rho_0 \bar X_{tt})  ,_k],_j \|^2_{ L^2(0,T; L^2(\Omega) )}
\le C \,.
$$
Once again we repeat the estimates for $\bar X_{tt}$ which we just explained for $\bar X$, this time $L^2$-in-time, and we obtain the
desired result; namely, 
\begin{equation}
\|\bar X_{tt}\|^2_ { L^2(0,T; H^2(\Omega) )}  \le     
C    \,. \label{XttH2c}
\end{equation}
It follows that (\ref{ssXt}) holds almost everywhere.

\subsubsection{$L^2(0,T;H^3(\Omega))$ regularity for $\bar X_{t}$}\label{sec834}  Using (\ref{G1}),  we write (\ref{ssXt.a}) as
$$
- 2 \kappa \p_t [\frac{\bar B^{jk}}{\rho_0} (\rho_0 \bar X)  ,_k],_j = \bar G_t - \frac{\bar J^3 \bar X_{tt}}{\rho_0} \,.
$$
The estimate (\ref{XttH2c}) together with our higher-order Hardy inequality shows that
\begin{equation}\label{need3}
 \| \p_t [\frac{\bar B^{jk}}{\rho_0} (\rho_0 \bar X)  ,_k],_j \|^2_{ L^2(0,T; H^1(\Omega) )}  \le C \,,
\end{equation} 
and hence by the fundamental theorem of calculus, 
$$
\| [\frac{\bar B^{jk}}{\rho_0} (\rho_0 \bar X)  ,_k],_j \|^2_1  \le C \,.
$$

We employ the identity (\ref{nice}) to find that
\begin{align*} 
\bar \p^h \bigl[ \frac{\bar B^{jk} }{\rho_0} ( \rho_0 \bar X),_k\Bigl],_j & = 
 \bigl[ \frac{\bar B^{jk,h} }{\rho_0} ( \rho_0 \bar \p^h\bar X),_k\Bigl],_j + \Bigl[ \bar \p^h \bar B^{jk}( \bar X,_k + \rho_0,_k \frac{\bar X}{\rho_0}) \Bigr],_j\\
 & \qquad + \Bigl[  \bar B^{jk,h} \bar \p^h\rho_0,_k \frac{\bar X}{\rho_0} \Bigr],_j -  \Bigl[  \bar B^{jk,h} \rho_0,_k  {\frac{\bar \p^h \rho_0}{\rho_0}} \frac{\bar X^h}{\rho_0^h} \Bigr],_j \,.
\end{align*} 
Since the last three term on the right-hand side are bounded in $ L^2(\Omega) $ thanks to our higher-order Hardy inequality and
(\ref{XttH2}),  we see that
$$
\left\|  \bigl[ \frac{\bar B^{jk,h} }{\rho_0} ( \rho_0 \bar \p^h\bar X),_k\Bigl],_j \right\|^2_0 \le C \,.
$$
Now repeating the argument which lead to (\ref{XttH2}) with $\bar \p^h \bar X$ replacing $\bar X$, we find that
\begin{equation}\label{Xtan2}
\| \bar \p \bar X \|_2^2 \le C \,.
\end{equation} 

By differentiating the relation (\ref{need2}) with respect to $x_3$ and using the estimate (\ref{Xtan2}), we see that
$$
\| \bar X,_{333} + \rho_0,_3 \left( \frac{\bar X}{\rho_0}\right),_{33} \|^2_0 \le C \,.
$$
Now by using the variable $Y$ defined by (\ref{varY}), we can repeat our  argument to find that $\|\bar X,_{333}\|_0^2 \le C$ and hence
that 
\begin{equation}\label{XH3}
\|\bar X\|^2_3 \le C \,.
\end{equation} 
From (\ref{need3}), we then easily infer that
$$
\int_0^T\| [\frac{\bar B^{jk}}{\rho_0} (\rho_0 \bar X_t)  ,_k],_j \|^2_1 dt \le C \,,
$$
so that the argument just given allows us to conclude that
\begin{equation}
\|\bar X_{t}\|^2_ { L^2(0,T; H^3(\Omega) )}   \le     C
   \,. \label{XtH3}
\end{equation}

\subsubsection{$L^2(0,T;H^4(\Omega))$ regularity for $\bar X$} \label{sec835}  By repeating the argument of Section \ref{sec834}, we
find that
\begin{equation}
\|\bar X\|^2_ { L^2(0,T; H^4(\Omega) )}   \le     
C   \,. \label{XH4}
\end{equation}

We have thus established existence and regularity of our solution $\bar X$; however, the bounds and time-interval of existence
depend on $\nu >0$.   We next turn to better Sobolev-type estimates to establish bounds for $\bar X$ and its time-derivatives which
are independent of $\nu$ and are useful for our fixed-point scheme.

\subsubsection{Estimates for $ \| \bar X\|^2_\XT$ independent of $\nu$}\label{smoothparam} \hspace{4 in}

\noindent
{\bf Step 1.}
We begin this section by getting $\nu$-independent energy estimates for the third time-differentiated problem (\ref{ssXttt}).


\begin{lemma} \label{lem_G3}  For  $T>0$ taken sufficiently small and $ \delta >0$, 
\begin{align} 
&
 \left\| \frac{\bar X_{tttt}}{\rho_0}\right\|^2_{L^2(0,T; H^{-1}(\Omega) )} +
\sup_{t \in [0,T]}  \left\| \frac{\bar X_{ttt}(t)}{\sqrt{\rho_0}} \right\|^2_0  +C_p  \left\|\bar X_{ttt}\right\|^2_ {L^2(0,T; \h)}   \n \\
& \qquad\qquad\qquad
\le  \mathcal{N} _0 + TP(\|\bar X\|^2_\XT) +C \delta \|\bar X\|^2_\XT +  \pbv +  \pcurl
 \,.  \label{sXest4}
\end{align} 
\end{lemma} 
\begin{proof}
We write the forcing function $\bar G_{ttt}+ \mathcal{G} _3$ as 
\begin{equation}\label{sG3}
\bar G_{ttt}+ \mathcal{G} _3 = \underbrace{\bar G_{ttt}+ \p_t \mathcal{G} _2}_{ \mathcal{T} _1} + \underbrace{ 2\kappa   \Bigl[\bar B_t^{jk} \rhoi ( \rho_0 \bar X_{tt}) ,_k\Bigr],_j}_{ \mathcal{T} _2} - \underbrace{ \frac{ (\bar J^3)_t  \bar X _{ttt}}{\rho_0}}_{ \mathcal{T} _3} \,.
\end{equation} 
We test (\ref{ssXttt.a}) with $\bar X_{ttt}$.  In the identical fashion that we obtained (\ref{need100}), we see that
\begin{align} 
&{\frac{1}{32}} \frac{d}{dt} \int_ \Omega   \frac{ |\bar X_{ttt}|^2}{\rho_0}dx+ 2 \kappa  \lambda  \int_ \Omega |D \bar X_{ttt}|^2 \, dx 
+  \kappa  \lambda   \int_ \Omega \frac {|D \rho_0|^2}{\rho_0^2} |\bar X_{ttt}|^2 \, \, dx  \n \\
&  \qquad  \le   \|  {\frac{1}{2}}   (\bar J^3)_t  +  \kappa  \rho_0,_{jk} \bar B^{jk} +  \kappa \rho_0,_{k} \bar B^{jk},_j \|_{ L^ \infty ( \Omega )}
      \int_ \Omega \frac{1 }{\rho_0}  | \bar X_{ttt}|^2 \, \, dx 
+ \langle \bar G_{ttt}+ \mathcal{G} _3 \ , \ \bar X_{ttt} \rangle   \,. \n
\end{align} 
Integrating this inequality from $0$ to $t \in (0,T]$, we see that
\begin{align} 
&{\frac{1}{32}} \sup_{t \in [0,T]} \int_ \Omega  \rhoi |\bar X_{ttt}(t)|^2dx+ 2 \kappa  \lambda \int_0^T  \int_ \Omega |D \bar X_{ttt}|^2 \, dx dt  \le \mathcal{N} _0\n \\
&   + T \sup_{t \in [0,T]}   \|  {\frac{1}{2}}   (\bar J^3)_t  +  \kappa  \rho_0,_{jk} \bar B^{jk} +  \kappa \rho_0,_{k} \bar B^{jk},_j \|_{ L^ \infty ( \Omega )}
      \int_ \Omega \frac{1 }{\rho_0}  | \bar X_{ttt}(t)|^2 \, \, dx 
+\int_0^T \langle \bar G_{ttt}+ \mathcal{G} _3 \ , \ \bar X_{ttt} \rangle dt   \,. \n
\end{align} 
By the Soblev embedding theorem 
$$
 \|  {\frac{1}{2}}   (\bar J^3)_t  +  \kappa  \rho_0,_{jk} \bar B^{jk} +  \kappa \rho_0,_{k} \bar B^{jk},_j \|_{ L^ \infty ( \Omega )}
 \le C  \|  {\frac{1}{2}}   (\bar J^3)_t  +  \kappa  \rho_0,_{jk} \bar B^{jk} +  \kappa \rho_0,_{k} \bar B^{jk},_j \|_2 \,.
$$
The highest-order derivative in  the term $ {\frac{1}{2}}   (\bar J^3)_t$ scales like $Dv$, while the highest-order derivative in 
$\bar B^{jk},_j$ scales like $D^2\eta$, which means that we have to be able to bound $ \sup_{t \in [0,T]} \| v(t)\|_3$ as well as
$\sup_{t \in [0,T]} \| \eta(t)\|_4$, and these are clearly bounded by $C \sqrt{t} \| \bar v\|_\XT$.
 Therefore, by choosing $T$ sufficiently small and invoking the
 Poincar\'{e} inequality, we see that
 \begin{align} 
&C \sup_{t \in [0,T]} \int_ \Omega  \rhoi |\bar X_{ttt}(t)|^2dx+  C_{ \kappa \lambda } \int_0^T  \int_ \Omega |D \bar X_{ttt}|^2 \, dx dt  \le \mathcal{N} _0
+\int_0^T \langle \bar G_{ttt}+ \mathcal{G} _3 \ , \ \bar X_{ttt} \rangle dt   \,. \n
\end{align}

We proceed to the analysis of the terms in   $\int_0^T \langle \bar G_{ttt}+ \mathcal{G} _3 \ , \ \bar X_{ttt} \rangle dt$, and 
we begin with the term $ \mathcal{T} _3$ in (\ref{sG3}).  We have that 
$$\int_0^T \langle \mathcal{T} _3, \bar X_{ttt} \rangle  dt
\le  \sup_{t \in [0,T]}  \|   (\bar J^3)_t \|_{ L^ \infty ( \Omega )}  \int_0^T \int_ \Omega \frac{1 }{\rho_0}  | \bar X_{ttt}|^2 \, \, dx  dt
\le C_M T^ {\frac{3}{2}} \sup_{t \in [0,T]}  \int_ \Omega \frac{1 }{\rho_0}  | \bar X_{ttt}|^2 \, \, dx \,,
$$
where we have made use of the Soblev embedding theorem giving the inequality $ \|   (\bar J^3)_t \|_{ L^ \infty ( \Omega )} \le C 
 \|   (\bar J^3)_t \|_2 \le \sqrt{t} C_M$, where $C_M$ depends on $M$.


To estimate the term $ \mathcal{T} _2$ in (\ref{sG3}), notice that
\begin{align*} 
\langle \mathcal{T} _2, \bar X_{ttt} \rangle  & = -2 \kappa \int_\Omega  \bar B_t^{jk}  (\bar X_{tt },_k + \rho_0,_k \frac{ \bar X_{tt}}{\rho_0} ) \bar X_{ttt},_j dx \le C \| \bar B_t\|_2 \| \bar X_{tt}\|_1 \|\bar X_{ttt}\|_1  \\
 &\le \delta \|\bar X_{ttt}\|_1^2 + C \| \bar B_t\|_2^2  \| \bar X_{tt}\|_1^2 \le  \delta \|\bar X_{ttt}\|_1^2 + C \| \bar B_t\|_2^2 ( \|\bar X_{tt}(0)\|_1^2 +
 t  \| \bar X_{ttt}(t)\|_1^2 ) \,,
\end{align*} 
and thus
\begin{align*} 
\int_0^T \langle \mathcal{T} _2, \bar X_{ttt}  \rangle  dt
 & \le  \mathcal{N} _0 + \delta \| \bar X \|_\XT^2 + \pbv + TP( \| \bar X\|^2_\XT)
\,.
\end{align*}

It remains to estimate $\langle \mathcal{T} _1, \bar X_{ttt} \rangle $; we use the identity (\ref{G2}) defining $ \mathcal{G} _2$ to  expand $ \mathcal{T} _1$ as 
$$ \mathcal{T}_1 = \bar G_{ttt} + \p_t \mathcal{G} _2= \bar G_{ttt} + \p_{tt} \mathcal{G} _1 + \p_t \Bigl(2\kappa   \Bigl[\bar B_t^{jk} \rhoi ( \rho_0 \bar X_t) ,_k\Bigr],_j - \frac{ (\bar J^3)_t  \bar X _{tt}}{\rho_0}\Bigl)\,.$$

The terms
$\langle  \p_t \Bigl(2\kappa   \Bigl[\bar B_t^{jk} \rhoi ( \rho_0 \bar X_t) ,_k\Bigr],_j - \frac{ (\bar J^3)_t \bar X _{tt}}{\rho_0}\Bigl), \bar X_{ttt} \rangle $ are estimated in
the same way and have the same bounds  as  $\langle  \mathcal{T} _2, \bar X_{ttt} \rangle  $ and $ \langle \mathcal{T} _3, \bar X_{ttt} \rangle $ above, so we focus on estimating $ \langle \bG_{ttt}+ \p_{tt} \mathcal{G} _1, \bar X_{ttt} \rangle $.  To do so, we use the identity (\ref{G1}) defining $ \mathcal{G} _1$ and
write
$$
 \bG_{ttt}+ \p_{tt} \mathcal{G} _1 = \bG_{ttt} 
+\underbrace{\p_{tt}\Bigl( 2\kappa   \Bigl[\bar B_t^{jk} \rhoi ( \rho_0 \bar X) ,_k\Bigr],_j\Bigr)}_{ \mathcal{S} _1} 
- \underbrace{\p_{tt}\Bigl( \frac{ (\bar J^3)_t \bar X _{t}}{\rho_0} \Bigr)}_{ \mathcal{S} _2} \,. 
$$
Expanding $ \mathcal{S} _1$ as
$$
\mathcal{S} _1 = \underbrace{ 2\kappa   \Bigl[\bar B_{ttt}^{jk} \rhoi ( \rho_0 \bar X) ,_k\Bigr],_j}_{ { \mathcal{S} _1}_a} 
+ \underbrace{4\kappa   \Bigl[\bar B_{tt}^{jk} \rhoi ( \rho_0 \bar X_t) ,_k\Bigr],_j}_{ { \mathcal{S} _1}_b}
+\underbrace{2\kappa   \Bigl[\bar B_{t}^{jk} \rhoi ( \rho_0 \bar X_{tt}) ,_k\Bigr],_j}_{ { \mathcal{S} _1}_c} \,,
$$
we see that for $ \delta >0$,
\begin{align*} 
\langle { { \mathcal{S} _1}_a}, \bar X_{ttt} \rangle 
& =  - 2 \kappa \int_\Omega \bar B_{ttt}^{jk} ( \bar X,_k + \rho_0,_k \frac{\bar X}{\rho_0}) \bar X_{ttt} ,_jdx \\
& \le C \| \bar B_{ttt}^{jk}\|_0 \, (\| \bar X,_k\|_{2}+ \| \frac{\bar X}{\rho_0}\|_{2}) \,\|D \bar X_{ttt}\|_0 
 \le C \| \bar B_{ttt}^{jk}\|_0 \, \| \bar X\|_{3} \,\|D \bar X_{ttt}\|_0  \\
&  \le C \| \bar B_{ttt}^{jk}\|_0^2 \, ( \| \bar X(0)\|_{3}^2 +t \| \bar X_t\|^2_3 )  +   \delta \|\bar X_{ttt}\|_1^2   \,.
\end{align*} 
where we have used the Sobolev embedding theorem for the first inequality,  our higher-order Hardy inequality Lemma \ref{Hardy}
for the second inequality, the Cauchy-Young inequality  together with the
fundamental theorem of calculus for the third inequality.   We see that
\begin{align*} 
\int_0^T \langle { { \mathcal{S} _1}_a}, \bar X_{ttt} \rangle  dt 
& \le  \mathcal{N} _0 + \delta \| \bar X \|_\XT^2 + \pbv + TP( \| \bar X\|^2_\XT) \,.
\end{align*} 

 The duality pairing involving  ${ \mathcal{S} _1}_b$ and ${ \mathcal{S} _1}_c$ can be estimated in the same way to provide the estimate
\begin{align*} 
\int_0^T \langle { { \mathcal{S} _1}}, \bar X_{ttt} \rangle  dt 
& \le  \mathcal{N} _0 + \delta \| \bar X \|_\XT^2 + \pbv + TP( \| \bar X\|^2_\XT) \,.
\end{align*} 
The duality pairing involving
$ \mathcal{S} _2$  is estimated in the same manner as $ \mathcal{T} _3$  and $ \mathcal{S} _1$ to  yield 
$$
\int_0^T \langle { { \mathcal{S} _2}}, \bar X_{ttt} \rangle  dt 
 \le  \mathcal{N} _0 + \delta \| \bar X \|_\XT^2 + \pbv + TP( \| \bar X\|^2_\XT) 
 + C_M T^ {\frac{3}{2}} \sup_{t \in [0,T]}  \int_ \Omega \frac{1 }{\rho_0}  | \bar X_{ttt}|^2 \, \, dx \,.
$$

It thus remains to estimate the duality pairing $\int_0^T \langle \bG_{ttt}, \bar X_{ttt} dt$.  We write
\begin{equation}\n
{\bG}_{ttt} = -\underbrace{ \p_{ttt} \Bigl( 3 \bar J^{-1} (\bar J_t)^2 - \p_t \bar a^j_i \, \bar v^i,_j \Bigr)}_{ \mathcal{L} _1}  +
 - \underbrace{2 \p_{ttt} [\bar a^j_i \bar A^k_i (\rho_0 \bar J^{-1} ),_k],_j  }_{ \mathcal{L} _2}
\underbrace{\kappa \p_{ttt} [ \bar a^j_i \p_t \bar a^k_i  \rhoi (\rho_0 ^2 \bar J^{-2} ),_k],_j}_{ \mathcal{L} _3}
 \,.
\end{equation} 
Notice that   by the Cauchy-Schwarz inequality
\begin{align*} 
\int_0^T \langle  \mathcal{L} _1 \ , \ \bar X_{ttt} \rangle dt
&\le \int_0^T \| \sqrt{\rho_0} \p_{ttt} \Bigl( 3 \bar J^{-1} (\bar J_t)^2 - \p_t \bar a^j_i \, \bar v^i,_j \Bigr) \|_0  \  \| \frac{ \bar X_{ttt}}{\sqrt{\rho_0}}\|_0 dt\\
& \le \pbv + CT  \sup_{t \in [0,T]}  \int_ \Omega \frac{1 }{\rho_0}  | \bar X_{ttt}|^2 \, \, dx \,.
\end{align*} 
Next, we write
\begin{align*} 
\langle \mathcal{L} _2, \bar X_{ttt} \rangle =2 \int_ \Omega  \p_{ttt} [\bar a^j_i \bar A^k_i (\rho_0 \bar J^{-1} ),_k] \ \bar X_{ttt},_j\, dx
\end{align*} 

The higher-order derivatives in $\p_{ttt} [\bar a^j_i \bar A^k_i (\rho_0 \bar J^{-1} ),_k]$
scale like either $D (\rho_0 D\bar v_{tt})$ or  $D\bar v_{tt}$ so  the fundamental theorem of calculus and the Cauchy-Young inequality once again shows that for $ \delta >0$,
\begin{align*} 
\int_0^T \langle \mathcal{L} _2, \bar X_{ttt} \rangle dt  = \pbv + \delta \| \bar X\|^2_\XT \,.
\end{align*} 

A good estimate for $\mathcal{L} _3$   requires the curl structure of Lemma \ref{lem_curlcurl}. We write the highest-order term in 
$$ \kappa  \p_{ttt} [ \bar a^j_i \p_t \bar a^k_i  \rhoi (\rho_0 ^2 \bar J^{-2} ),_k],_j
= \kappa  \p_{ttt} [ 2\bar a^j_i \p_t \bar a^k_i  \rho_0,_k \bar J^{-2} + \rho_0 \bar a^j_i \p_t \bar a^k_i \bar J^{-2} ,_k)],_j
$$
as
\begin{equation}\label{sstemp02}
2\kappa \p_{ttt} [ (\bar a^j_i \p_t \bar a^k_i) \rho_0,_k \bar J^{-2}],_j  \,;
\end{equation} 
all of the other terms arising from the distribution of $\p_{ttt}$ are lower-order and can be estimated in the same way as $ \mathcal{L} _2$.
Now, using Lemma \ref{lem_curlcurl},  the highest-order term in (\ref{sstemp02}) is written as
\begin{align*} 
\p_{ttt} [ (\bar a^j_i \p_t \bar a^k_i) \rho_0,_k \bar J^{-2} ],_j  & 
= \underbrace{ [\operatorname{curl} \operatorname{curl}\bar v_{ttt}]^k \rho_0,_k \bar J^{-2} }_{ \mathfrak{t}  _1} \\
&\qquad
+\underbrace{ \bar v^r_{ttt},_{sj} \Bigl(\bar J^{-1} [\bar a^s_r \bar a^k_i - \bar a^s_i \bar a^k_r] \bar a^j_i - [\delta^s_r \delta ^k_i - \delta ^s_i \delta ^k_r] \delta ^j_i \Bigr) \rho_0,_k \bar J^{-2} }_{ \mathfrak{t}  _2}\\
& \qquad\qquad\qquad\qquad +  \underbrace{\bar v_{ttt}^r,_{s} \Bigl( \bar J^{-1} [\bar a^s_r \bar a^k_i - \bar a^s_i \bar a^k_r] \Bigr),_j \bar a^j_i \rho_0,_k \bar J^{-2} }_
{ \mathfrak{t}  _3}+\mathcal{R}\,,
\end{align*} 
with $\mathcal R$ being lower-order and once again estimated as  $ \mathcal{L} _2$.   Integration by parts with respect to the curl operator in
the term $ \mathfrak{t}  _1$, we see that
\begin{align*} 
\langle \mathfrak{t}  _1,  \bar X_{ttt} \rangle & = \int_ \Omega \operatorname{curl} \bar v_{ttt} \cdot  D \times ( D\rho_0 \bar  J^{-2} \bar X_{ttt}) dx \\
& \le   \| \operatorname{curl} \bar v_{ttt}(t) \|_0 \ \Bigl( \| D\rho_0 \bar J^{-2}\|_{L^ \infty (\Omega )   } \ \| D\bar X_{ttt}\|_0 +
   \| \operatorname{curl} (D\rho_0\bar J^{-2} )\|_{L^ 3 (\Omega )   } \ \| \bar X_{ttt}\|_{L^ 6 (\Omega )   }\Bigr) \\
& \le   \| \operatorname{curl} \bar v_{ttt}(t) \|_0 \ \Bigl( \| D\rho_0 \bar J^{-2}\|_2 +   \| \operatorname{curl} (D\rho_0\bar J^{-2} )\|_1\Bigr)\| \bar X_{ttt}\|_1 \\
& \le C    \| \operatorname{curl} \bar v_{ttt}(t) \|^2_0 \ \Bigl( \| D\rho_0 \bar J^{-2}\|^2_2 +   \| \operatorname{curl} (D\rho_0\bar J^{-2} )\|^2_1\Bigr)
+ \delta \| \bar X_{ttt}\|_1^2 \,.
\end{align*} 
It follows that 
\begin{align*} 
\int_0^T \langle \mathfrak{t}  _1,  \bar X_{ttt} \rangle dt
& \le \pcurl + \pbv  + \delta \| \bar X\|^2_\XT \,.
\end{align*} 

For $ \mathfrak{t}  _2$, 
\begin{align*} 
\langle \mathfrak{t}  _2, \bar X_{ttt} \rangle & = - \int_ \Omega  \bar v^r_{ttt},_{s} \left[ \Bigl(\bar J^{-1} [\bar a^s_r \bar a^k_i - \bar a^s_i \bar a^k_r] \bar a^j_i - [\delta^s_r \delta ^k_i - \delta ^s_i \delta ^k_r] \delta ^j_i \Bigr) \rho_0,_k \bar J^{-2} \bar X_{ttt} \right],_j dx \,.
\end{align*} 
Given that $\| a(t) - \operatorname{Id}\|_3 = \| \int_0^t a_t (t')dt' \|_3 \le \sqrt{t} P(\| \bar v \|_\ZT)$, we see that
\begin{align*} 
\int_0^T \langle \mathfrak{t}  _2, \bar X_{ttt} \rangle  dt \le \pbv  + \delta \| \bar X\|^2_\XT  \,.
\end{align*} 
The duality pairing involving $ \mathfrak{t} _3$ can be estimated in the same way.  

Summing together the above inequalities and taking $T>0$ sufficiently small concludes the proof.
\end{proof}

It is easy to see that we have the same estimates for the  weak solutions $\bar X$, $\bar X_t$, and $\bar X_{tt}$ solving  (\ref{ssX}), (\ref{ssXt}), and (\ref{ssXtt}), respectively:
\begin{align} 
&
 \sum_{a=0}^3\left\| \p_t^a \frac{\bar X_t}{\rho_0}\right\|^2_{L^2(0,T; H^{-1}(\Omega) )} +
\sup_{t \in [0,T]}  \left\| \frac{ \p_t^a\bar X(t)}{\sqrt{\rho_0}} \right\|^2_0  +C_p  \left\|  \p_t^a\bar X\right\|^2_ {L^2(0,T; \h)}   \n \\
& \qquad\qquad\qquad
\le \mathcal{N} _0 + TP(\|\bar X\|^2_\XT) +C \delta \|\bar X\|^2_\XT +  \pbv +  \pcurl \,.  \label{allbound}
\end{align}

\noindent
{\bf Step 2.} Returning to the definition of $ \mathcal{G} _2$ in (\ref{G2}), by using the estimate (\ref{allbound}) together with the Hardy inequality, we see that 
$$
\| \bG_{tt} + \mathcal{G} _2 \|^2_{L^2(0,T; L^2(\Omega) )} \le \mathcal{N} _0 + TP(\|\bar X\|^2_\XT) +C \delta \|\bar X\|^2_\XT +  \pbv +  \pcurl \,.
$$
Combining this with the estimate (\ref{sXest4}), the equation  (\ref{ssXtt.a}) shows that
\begin{align} 
&4 \kappa ^2\Bigl\|   \Bigl[\frac{\bar B^{jk}}{\rho_0} (\rho_0 \bar X_{tt})  ,_k\Bigr],_j  \Bigr\|^2_{L^2(0,T; L^2(\Omega) )}
=\Bigl\|- \frac{ \bar J^3\bar  X _{ttt}}{\rho_0} + \mathcal{G} _2  \Bigr\|^2_{L^2(0,T; L^2(\Omega) )}    \n \\
& \qquad\qquad\qquad \qquad \le \mathcal{N} _0 + TP(\|\bar X\|^2_\XT) +C \delta \|\bar X\|^2_\XT +  \pbv +  \pcurl\,.\nonumber
\end{align} 
By  repeating our argument of Section \ref{sec833} but this time using Soblev-type estimates for the horizontal-derivative
estimates (replacing the difference quotient estimates as we already have regularity), we obtain the desired bound:
\begin{equation}\label{sXttH2}
\| \bar X_{tt}\|^2_{L^2(0,T; L^2(\Omega) )} \le \mathcal{N} _0 + TP(\|\bar X\|^2_\XT) +C \delta \|\bar X\|^2_\XT +  \pbv +  \pcurl \,.
\end{equation} 

\noindent
{\bf Step 3.} From the definition of $ \mathcal{G} _1$ in (\ref{G1}), we similarly see that 
$$
\|\bG_t +  \mathcal{G} _1 \|^2_{L^2(0,T; H^1(\Omega) )} \le\mathcal{N} _0 + TP(\|\bar X\|^2_\XT) +C \delta \|\bar X\|^2_\XT +  \pbv +  \pcurl  \,.
$$
Following our argument for the regularity of $\bar X_{tt}$, we obtain the estimate
\begin{equation}
\|\bar X_{t}\|^2_ { L^2(0,T; H^3(\Omega) )}   \le     
\mathcal{N} _0 + TP(\|\bar X\|^2_\XT) +C \delta \|\bar X\|^2_\XT +  \pbv +  \pcurl     \,. \label{sXtH3}
\end{equation}

\noindent
{\bf Step 4.} Finally, 
$$
\| \bG \|^2_{L^2(0,T; H^2(\Omega) )} \le \mathcal{N} _0 + TP(\|\bar X\|^2_\XT) +C \delta \|\bar X\|^2_\XT +  \pbv +  \pcurl \,.
$$
Following the argument of Step 3 and using Sobolev-type estimates for the horizontal-derivative bounds (which replace
the horizontal difference-quotient estimates), we finally conclude that
\begin{equation}
\|\bar X\|^2_ { L^2(0,T; H^4(\Omega) )}   \le     
 \mathcal{N} _0 + TP(\|\bar X\|^2_\XT) +C \delta \|\bar X\|^2_\XT +  \pbv +  \pcurl  \,. \label{sXH4}
\end{equation}

\subsubsection{The proof of Proposition \ref{prop_X}} \label{sec836} Summing the inequalities
(\ref{sXest4}), (\ref{sXttH2}), (\ref{sXtH3}), and (\ref{sXH4}), we obtain the estimate
$$
 \|\bar X\|^2_\XT   \le     
\mathcal{N} _0 + TP(\|\bar X\|^2_\XT) +C \delta \|\bar X\|^2_\XT +  \pbv +  \pcurl \,.
$$
Choosing $ \delta >0$ and $T>0$  sufficiently small (and readjusting the constants), we see
that  
$$
 \|\bar X\|^2_\XT   \le   \mathcal{N} _0 +\pbv +  \pcurl \,.
$$
 As the right-hand
side does not depend on $\nu >0$, we can pass to the limit as $\nu \to 0$ in (\ref{ssX}).  This completes
the proof of Proposition \ref{prop_X}.

\begin{remark}
Suppose that $\bar v_1$ and $\bar v_2 $ are both elements of $ \mathcal{C} _T(M)$. 
For $a=1,2$, let $\bar X_a$ denote the solution of (\ref{ssX}) with coefficient matrix $\bar B_a$ and forcing function $\bar G_a$ formed
from $\bar v_a$ rather than $\bar v$.      Our proof of  Proposition \ref{prop_X} then shows that
\begin{equation} \label{rem_diff}
 \|\bar X_1 - \bar X_2\|^2_\XT   \le     T P( \| \bar v_1 - \bar v_2\|_\ZT^2)  +  P( \| \operatorname{curl} \bar v_1 - \operatorname{curl}  \bar v_2\|_\YT^2)  \,.
\end{equation} 
We will make use of this inequality in our iteration scheme below.
\end{remark}

\subsection{Existence of the fixed-point and the proof of Theorem \ref{thm_ksoln}}

\subsubsection{The boundary convolution operator $ \Lambda _ \epsilon $ on $\Gamma$} \label{subsec_hor}

For $\epsilon >0$, let $ 0 \le \rho_ {\epsilon }  \in C^\infty_0( \mathbb{R}^2)$ with $\spt (\rho_{\epsilon }) \subset \overline{B(0, \epsilon )} $ denote a standard family of mollifiers on $\mathbb{R}^2$ . With $x_h =(x_1,x_2)$, we define the operation of {\it  convolution on the
boundary} as follows:
$$
\Lambda_\epsilon  f ( x_h) =  \int_{\mathbb{R}^2} \rho_ \epsilon  (x_h -y_h) f(y_h)  dx_h \  \text{ for } f \in L^1_{loc}( \mathbb{R}^2  ) \,.
$$

By standard properties of convolution, there exists a constant $C$ which is independent of $\epsilon $, such that for
$s \ge 0$,
$$
| \Lambda _ \epsilon  F|_s \le C |F|_s \ \ \ \forall \ F \in H^s(\Gamma) \,.
$$
Furthermore,
\begin{equation}\label{conv_est}
\epsilon  | \bar \p \Lambda _ \epsilon  F|_0 
\le C |F|_0 \ \ \ \forall \ F \in L^2(\Omega) \,.
\end{equation} 

\subsubsection{Construction of solutions to (\ref{defv}) via intermediate $ \epsilon $-regularization}
We shall establish the existence of a solution $v$ to (\ref{defv}) by first considering for any $\epsilon>0$ the $ \epsilon $-regularized system, where the higher-in-space order term in (\ref{defvc}) is smoothed via two boundary convolution operators on $\Gamma$:
\begin{subequations}\label{defvepsilon}
\begin{alignat}{2}
\operatorname{div} v_t^\epsilon  & = \operatorname{div} \bar v_t  - \operatorname{div} _{ \bar \eta} \bar v_t + \frac{[\bar X \bar J^2]_t}{\rho_0} -
\p_t \bar A^j_i \bar v^i,_j
\  && \text{in}\ \Omega \,,        \label{defvepsilona}\\
\operatorname{curl} v_t^\epsilon   &=  \operatorname{curl} \bar v_t - \operatorname{curl} _{\bar \eta} \bar v_t
+ 2\kappa \varepsilon_{ \cdot ji} 
\bar v,_s^r   \bar A^s_i\  \bar\Xi,_r^j(\bee) + \bar{\mathfrak C}  \ && \text{in}\ \Omega \,,\label{defvepsilonb}\\
(v^ \epsilon )_t^3&+ 2 \kappa \rho_0,_3 \Lambda _ \epsilon ^2 \operatorname{div} _\Gamma v^ \epsilon  
=2 \kappa \rho_0,_3 \Lambda _ \epsilon \operatorname{div} _\Gamma \bar v  - 2 \rho_0,_3 \Lambda _ \epsilon [ \bar J^{-2} \bar a^3_3] \n \\
 & \qquad   -2 \kappa  \rho_0,_3\Lambda _ \epsilon [ \bar J^{-2} \p_t \bar a^3_3] -2 \kappa \rho_0,_3  \Lambda _ \epsilon [\bar a^3_3 \p_t \bar J^{-2}] + \bar c_ \epsilon (t) N^3 
\qquad  &&   \text{on}\ \Gamma\,, \label{defvepsilonc} \\
\int_ \Omega (v^ \epsilon )_t^ \alpha dx& = -2 \int_ \Omega \bar A^k_ \alpha (\rho_0 \bar J^{-1} ),_k dx 
-2 \kappa  \int_ \Omega \p_t[ \bar A^k_ \alpha (\rho_0 \bar J^{-1} ),_k] dx  \,,  \label{defvepsilond} \\
(x_1,x_2)& \mapsto v^ \epsilon _t(x_1,x_2,x_3,t)  \text{ is $1$-periodic }    \,,   \label{defvepsilone} 
\end{alignat}
\end{subequations}
where the vector $\bar\Xi(\eta)$ is defined in (\ref{ODEXi}) and the function $\bar c_ \epsilon (t)$  (a constant in $x$) on the right-hand side of (\ref{defvepsilonc}) is defined by
\begin{align}
\bar c_ \epsilon (t)=& {\frac{1}{2}} \int_ \Omega ( \operatorname{div} \bar v_t - \operatorname{div} _{ \bar \eta} \bar v_t) dx  + {\frac{1}{2}} 
\int_ \Omega  \frac{[\bar X \bar J^2]_t}{\rho_0} dx - {\frac{1}{2}} \int_ \Omega \p_t \bar A^j_i \bar v^i,_j dx  + \int_ \Gamma \Lambda _ \epsilon [ \bar J^{-2} \bar a^3_3 ] \rho_0,_3 N^3 dS \n \\
& \ + \kappa \int_\Gamma \Lambda _ \epsilon [  \bar J^{-2} \p_t \bar a^3_3] \rho_0,_3 N^3 dS + \kappa \int_\Gamma \Lambda _ \epsilon [ \p_t \bar J^{-2}  \bar a^3_3 ] \rho_0,_3 N^3 dS \n \\
& \qquad \qquad + \kappa \int_\Gamma \operatorname{div}_\Gamma  ( \Lambda _ \epsilon ^2 v^\epsilon - \Lambda _ \epsilon \bar v) \rho_0,_3 N^3 dS \,.   \label{ceps}
\end{align}

We now outline the steps remaining in this section.
We shall first prove, by a fixed-point approach, that for a small time $T_\epsilon>0$ depending {\it a priori} on $\epsilon$, we have the existence of a solution to this problem. We shall then prove, via $ \epsilon $-independent energy estimates on the solutions of (\ref{defvepsilon}),  that $T_\epsilon=T$, with $T$ independent of $ \epsilon $, and that the sequence $\ve$ converges in an appropriate space to a solution $v$ of (\ref{defv}), which also satisfies the same energy estimates. These estimates will allow us to conclude the existence of a fixed-point $v=\bar v$.

\noindent{\bf Step 1: Solutions to (\ref{defvepsilon}) via the contraction mapping  principle.}
For 
\begin{align}
w\in \mathcal{X} ^3_T&=\{w \in  L^2(0,T;H^3(\Omega))  \ : \  w_t\in L^2(0,T;H^3(\Omega)),\  w_{tt}\in L^2(0,T;H^2(\Omega)), \n \\
&\qquad  w_{ttt}\in L^2(0,T;H^1(\Omega)), \ (x_1,x_2) \mapsto w \text{ is $1$-periodic}\} \label{csX3} \,,
\end{align}
with norm
$$
\|w\|_{ \mathcal{X} ^3_T}^2 = \| w\|^2_{L^2(0,T;H^3(\Omega))} + \| w_t\|^2_{L^2(0,T;H^3(\Omega))} + \| w_{tt}\|^2_{L^2(0,T;H^2(\Omega))}
+ \| w_{ttt}\|^2_{L^2(0,T;H^1(\Omega))}\,,
$$
we set $\Phi(w)=u_0+\int_0^t \p_t \Phi(w)$ where $\p_t \Phi(w)$ is defined by the elliptic system which specifies the divergence, curl, and normal
trace of the vector field $\p_t \Phi(w)$:
\begin{subequations}
\label{defw}
\begin{alignat}{2}
\operatorname{div} \p_t \Phi(w)&=\operatorname{div}  \bar v_t-\operatorname{div}_{\bee}  \bar v_t+\frac{[\bar X\bje^2]_t}{\rho_0}- \p_t\bAe_i^j\bar v,_j^i \qquad  && \text{in}\ \Omega\,,\label{defwa}\\
\operatorname{curl} \p_t \Phi(w)&=\operatorname{curl} \p_t \bar v-\operatorname{curl}_{\bee} \p_t \bar v  + 2\kappa \varepsilon_{\cdot ji}   \bve,_s^r   \bAe^s_i\  \bxi,_r^j(\bee) + \bar{\mathfrak C} \qquad &&\text{in}\ \Omega\,,\label{defwb}\\
\p_t \Phi(w) \cdot e_3&=-2\kappa \rho_0,_3 \Lambda _ \epsilon ^2 \operatorname{div} _\Gamma w
+2\kappa \rho_0,_3 \Lambda _ \epsilon \operatorname{div} _\Gamma \bar v
 \n\\
&\qquad -2\kappa  {\rho_0,_3} \Lambda _ \epsilon [ \bar J^{-2} \p_t \bar a^3_3]-2\rho_0,_3 \Lambda _ \epsilon [ \bar J^{-2} \bar a^3_3] \n \\
&\qquad  -2\kappa \rho_0,_3 \Lambda _ \epsilon [\bar a^3_3 \p_t \bar J^{-2} ]  +\bar c(w) N^3 && \text{on}\ \Gamma\,, \label{defwc}\\
\int_ \Omega \p_t\Phi(w)^\alpha dx& = -2 \int_ \Omega \bar A^k_ \alpha (\rho_0 \bar J^{-1} ),_k dx 
-2 \kappa  \int_ \Omega \p_t[ \bar A^k_ \alpha (\rho_0 \bar J^{-1} ),_k] dx  \,,  \label{defwd} \\
\forall t\in [0,T],\ \ \p_t \Phi(w)(t)&\ \text{is 1-periodic in the directions $e_1$ and $e_2$ }.
\end{alignat}
\end{subequations}
The function $\bar c (w)(t)$ in (\ref{defwc})  is defined by
\begin{align}
[\bar c(w)] (t)=& {\frac{1}{2}} \int_ \Omega ( \operatorname{div} \bar v_t - \operatorname{div} _{ \bar \eta} \bar v_t) dx  + {\frac{1}{2}} 
\int_ \Omega  \frac{[\bar X \bar J^2]_t}{\rho_0} dx - {\frac{1}{2}} \int_ \Omega \p_t \bar A^j_i \bar v^i,_j dx  \n \\
& \  + \int_ \Gamma \Lambda _ \epsilon [ \bar J^{-2} \bar a^3_3 ] \rho_0,_3 N^3 dS+ \kappa \int_\Gamma \Lambda _ \epsilon [  \bar J^{-2} \p_t \bar a^3_3] \rho_0,_3 N^3 dS  \n \\
& \qquad + \kappa \int_\Gamma \Lambda _ \epsilon [ \p_t \bar J^{-2}  \bar a^3_3 ] \rho_0,_3 N^3 dS + \kappa \int_\Gamma \operatorname{div}_\Gamma  ( \Lambda _ \epsilon ^2 w - \Lambda _ \epsilon \bar v) \rho_0,_3 N^3 dS \,,  \label{cw}
\end{align}
and is introduced so that the elliptic system (\ref{defw}) satisfies all of the solvability conditions.
Thus,
due to the definition (\ref{cw}) , the problem (\ref{defw}) defining $\p_t \Phi(w)$ is perfectly well-posed.  Applying Proposition \ref{prop1} to (\ref{defw}) and its first, second, and third time-differentiated versions, we find that
\begin{equation}\label{csgg1}
\|\p_t \Phi(w)-\p_t \Phi(\tilde w)\|_{\mathcal{X} ^3_{T_ \epsilon }}\le C(M,\epsilon) \|w-\tilde w\|_{\mathcal{X} ^3_{T_ \epsilon }}+ C_M T_\epsilon \| w-\tilde w\|_{\mathcal{X} ^3_{T_ \epsilon }}\,,
\end{equation}
the $ \epsilon $ dependence in the constant $C(M, \epsilon )$ coming from repeated use of (\ref{conv_est}).    Note that the lack of $w$
on the right-hand sides of (\ref{defwa}) and (\ref{defwb}) implies that both the divergence and curl
of $\p_t \Phi(w)-\p_t \Phi(\tilde w)$ vanish, and that  on $\Gamma$, 
$$[\p_t \Phi(w)-\p_t \Phi(\tilde w)] \cdot e_3 = 2\kappa \rho_0,_3 \Lambda _ \epsilon ^2 \operatorname{div} _\Gamma (\tilde w-w) + [\bar c(w) -\bar c(\tilde w)]N^3\,.$$
It follows from (\ref{csgg1})
 that
\begin{equation*}
\| \Phi(w)- \Phi(\tilde w)\|_{\mathcal{X} ^3_{T_ \epsilon }}\le T_\epsilon C(M,\epsilon) \|w-\tilde w\|_{\mathcal{X} ^3_{T_ \epsilon }}\,,
\end{equation*}
and therefore the mapping $\Phi:{\mathcal{X} ^3_{T_ \epsilon }} \to {\mathcal{X} ^3_{T_ \epsilon }}$ is a contraction if $T_\epsilon$ is taken sufficiently small, leading to the existence and uniqueness of a fixed-point $\ve=\Phi(\ve)$,  which is therefore a solution of (\ref{defvepsilon}) on $[0,T_\epsilon]$.

\noindent {\bf Step 2: $ \epsilon $-independent energy estimates for $\ve$.}  Having obtained a unique solution to (\ref{defvepsilon}),
we now proceed with 
 $\epsilon-$independent estimates on this system.   We integrate the divergence (\ref{defvepsilona}) and curl (\ref{defvepsilonb}) relations in time, and we now view the PDE for the normal trace (\ref{defvepsilonc})  as a parabolic equation for $v^ \epsilon $ on $\Gamma$:
\begin{subequations}
\label{defve}
\begin{align}
\operatorname{div} \ve&=  \operatorname{div} \bar v-\operatorname{div}_{\bee} \bar v+\frac{\bar X\bje^2}{\rho_0}\qquad  \text{in} \ \Omega \ ,\label{defvea}\\
\operatorname{curl}  \ve&=\operatorname{curl} u_0 + \operatorname{curl}  \bar v-\operatorname{curl}_{\bee}\bar v  + 2\kappa \int_0^t\varepsilon_{\cdot ji}   \bve,_s^r   \bAe^s_i\  \bxi,_r^j(\bee)  \n \\
& \qquad  + \int_0^t (\varepsilon_{\cdot ji}   \bve,_s^i   \p_t\bAe^s_j  +\bar{\mathfrak C} )
 \ \   \text{in}\ \Omega ,\label{defveb}\\
(v^ \epsilon )_t^3  +2 \kappa \rho_0,_3 \Lambda _ \epsilon ^2 \operatorname{div} _\Gamma v^ \epsilon&=      +
2 \kappa \rho_0,_3 \Lambda _ \epsilon \operatorname{div} _\Gamma \bar v  - 2 \rho_0,_3 \Lambda _ \epsilon [ \bar J^{-2} \bar a^3_3]
 \n \\
 & \qquad -2 \kappa  \rho_0,_3\Lambda _ \epsilon [ \bar J^{-2} \p_t \bar a^3_3]     -2 \kappa \rho_0,_3  \Lambda _ \epsilon [\bar a^3_3 \p_t \bar J^{-2}]  \n \\
 & \qquad \qquad + \bar c_ \epsilon (t) N^3    
 \qquad   \qquad  \text{on}\ \Gamma\,,  \label{defvec}\\
\ve(0)&=u_0,\\
\int_ \Omega (v^ \epsilon )^ \alpha dx& = \int_\Omega u_0^\alpha dx -2 \int_0^t\int_ \Omega \bar A^k_ \alpha (\rho_0 \bar J^{-1} ),_k dx 
-2 \kappa  \int_ \Omega [ \bar A^k_ \alpha (\rho_0 \bar J^{-1} ),_k] dx  \,,  \label{defvee} \\
\forall t\in [0,T],\ \  \ve(t)&\ \text{is 1-periodic in the directions $e_1$ and $e_2$ }\,,
\end{align}
\end{subequations}
where $\bar c^\epsilon(t)$ is defined in (\ref{ceps}).

We will establish the existence of a fixed-point in the metric space $ \mathcal{C}_{T_ \kappa }(M)$ defined in (\ref{ctm}), but to do
so we will make use of the space (depending on $\epsilon$)
\begin{align*}
 \mathcal{X} _T^4=\{& w\in L^\infty(0,T;H^{3.5}(\Omega)) \cap L^2(0,T; \h) \ : \  w_t\in L^2(0,T;H^3(\Omega)),   \\
 & \ w_{tt}\in L^2(0,T;H^2(\Omega)), \ 
 w_{ttt}\in L^2(0,T;H^1(\Omega)), \ \Lambda _ \epsilon  w \in L^2(0,T;H^4(\Omega)),\\
 &\ w(0)= u_0\} \,,
\end{align*}
with norm
\begin{align*} 
\| w\|^2_{ \mathcal{X} _T^4} &  = \| \Lambda _ \epsilon w\|^2_{ L^2(0,T;H^4(\Omega))} + \| w_t\|^2_{ L^2(0,T;H^3(\Omega))}
+ \| w_{tt}\|^2_{ L^2(0,T;H^2(\Omega))}
+ \| w_{ttt}\|^2_{ L^2(0,T;H^1(\Omega))}   \,.
\end{align*} 

Since $\bar v\in \mathcal{C} _{T_ \kappa }(M)$,   equations (\ref{defvea}) and (\ref{defveb}) show that both $\operatorname{div}\ve$ and $\operatorname{curl}\ve$ are  in  $L^2(0,T_ \epsilon ;H^3(\Omega))$;  additionally,  from (\ref{defvec}) and (\ref{conv_est}), we see that  
$(\ve)^3$ is in $L^\infty(0,T_\epsilon;H^{3.5}(\Gamma))$, and hence according to Proposition \ref{prop1}, 
$\ve\in L^2(0,T_\epsilon;H^4(\Omega))$, with a bound that {\it a priori} depends on $\epsilon$. We next show that, in fact, we can control 
$\Lambda _ \epsilon \ve$ in $\mathcal{C} _T(M)$ independently of $\epsilon$, on a time interval $[0,T]$ with $T>0$ independent of $\epsilon$.

We proceed by letting  $\bar \p^3$ act on each side of (\ref{defvec}), multiplying this equation by $-\frac{N^3}{\rho_0,_3} \bar \p^3 (\ve)^3$,
and then integrating over $\Gamma$.   This yields the following identity:
\begin{align}
&-\frac{1}{2}\frac{d}{dt} \int_{\Gamma}\frac{N^3}{\rho_0,_3} |\bar \p^3 (\ve)^3|^2 dS +2\kappa \int_\Gamma \bar \p^3 \Lambda _ \epsilon ^2\operatorname{div} _\Gamma v^ \epsilon \,  \bar \p^3 (\ve)^3 N^3 dS\n\\
&\qquad \qquad\qquad
= \int_\Gamma G\, \bar \p^3 (\ve)^3 N^3 dS -\int_\Gamma \bar \p^3 \bigl[ \rho_0,_3 \Lambda _ \epsilon F\bigr] \bar \p^3 (\ve)^3 \frac{N^3}{\rho_0,_3}dS \label{Eeps1}
\end{align}
where
\begin{align}
G&= -2 \kappa  \Bigl[ \bar \p^3 \rho_0,_3 \, \Lambda _ \epsilon ^2 \operatorname{div} _\Gamma v^ \epsilon 
+ 3 \bar \p^2 \rho_0,_3 \, \bar \p \Lambda _ \epsilon ^2 \operatorname{div} _\Gamma v^ \epsilon 
+ 3 \bar \p \rho_0,_3 \, \bar \p^2 \Lambda _ \epsilon ^2 \operatorname{div} _\Gamma v^ \epsilon 
\Bigr] \,, \label{csG} \\
F&=2\kappa\operatorname{div} _\Gamma \bar v
-2 \bar J^{-2} \bar a^3_3  -2\kappa \bar J^{-2} \p_t \bar a^3_3  -2\kappa  \bar a^3_3 \p_t \bar J^{-2} \,.  \label{csF}
\end{align}

Since $G$ contains lower-order terms, we immediately see that for any $t\in [0,T_\epsilon]$:
\begin{align}
 \int_\Gamma G \bar \p ^3 (\ve)^3 N^3\, dS \le & C |\bar \p^3 (\ve)|_{0}^2 \,.\label{G0}
\end{align} 
We then write 
\begin{equation}\label{ssneed78}
\bar \p^3 \bigl[ \rho_0,_3 \Lambda _ \epsilon F\bigr] = \rho_0,_3 \bar \p^3  \Lambda _ \epsilon F + \bar \p^3 \rho_0,_3 \Lambda _ \epsilon F
+ 3 \bar \p^2 \rho_0,_3 \bar \p \Lambda _ \epsilon F + 3 \bar \p \rho_0,_3 \bar \p^2 \Lambda _ \epsilon F \,,
\end{equation} 
and notice that since the last three term on the right-hand side are lower-order,
 we easily obtain the estimate
\begin{align} 
&\bigl|\int_\Gamma    \big[( \bar \p^3 \rho_0,_3 \Lambda _ \epsilon F
+ 3 \bar \p^2 \rho_0,_3 \bar \p \Lambda _ \epsilon F + 3 \bar \p \rho_0,_3 \bar \p^2 \Lambda _ \epsilon F \bigr)
\bar \p^3 (\ve)^3 \frac{N^3}{\rho_0,_3}dS\bigr|
 \le
 C  |\bar \p^3 (\ve)^3|_0 |\bar \p^2 D \bar v|_0\,.\label{G}
\end{align}

Next, in order to estimate  the highest-order term $\int_\Gamma \bar \p^3  \Lambda _ \epsilon F\, \bar \p^3 (\ve)^3 N^3dS$, we notice 
that by the standard properties of the boundary convolution operator $ \Lambda _ \epsilon $, we have that
\begin{align}
&\int_\Gamma \bar \p^3 \Lambda _ \epsilon F\,  \bar \p^3(\ve)^3  {N^3} dS=\int_\Gamma \bar \p^3 F \, \bar \p^3 \Lambda _ \epsilon (\ve)^3N^3
dS\n\\
& =
-\underbrace{2 \kappa \int_\Gamma \bar \p^3 ( \bar J^{-2} \p_t \bar a^3_3) \, \bar \p^3 \Lambda _ \epsilon (\ve)^3N^3 dS}_{ \mathcal{J}_1 }
-\underbrace{2 \kappa \int_\Gamma \bar \p^3 ( \bar a^3_3 \p_t \bar J^{-2}) \, \bar \p^3 \Lambda _ \epsilon (\ve)^3N^3 dS}_{ \mathcal{J}_2 } \n\\
& \qquad\qquad\qquad 
+\underbrace{ 2 \kappa \int_\Gamma \bar \p^3  \operatorname{div} _\Gamma \bar v
    \, \bar \p^3 \Lambda _ \epsilon (\ve)^3N^3 dS}_{ \mathcal{J}_3  } 
-\underbrace{ 2 \int_\Gamma \bar \p^3 \bigl[ 
 \bar J^{-2} \bar a^3_3   \bigr] \, \bar \p^3 \Lambda _ \epsilon (\ve)^3N^3 dS}_{ \mathcal{J}_4  }
 \,. \label{csmainJ}
\end{align} 
In order to estimate the integral $ \mathcal{J} _1$, we recall the formula for $ \p_t \bar a^3_3$ given in (\ref{pta33}), and write
\begin{align} 
\mathcal{J} _1 &= -2 \kappa \int_\Gamma \bar \p^3 ( \bar J^{-2}  [ \bar v,_1 \times \bar \eta,_2 + \bar \eta,_1 \times \bar v,_2]^3    ) \, \bar \p^3 \Lambda _ \epsilon (\ve)^3N^3 dS \n \\
& =\underbrace{-2\kappa\int_\Gamma \bigl[\bar \p^3\bar v,_1\times\bigl(\frac{\be,_2}{\bje^2} -\be,_2(0)\bigr)\bigr]^3 \bar \p^3\Lambda _ \epsilon (\ve)^3N^3 dS}_{ {\mathcal{J} _1}_a}
-
\underbrace{2\kappa\int_\Gamma \bar \p^3\bar v^1,_1 \, \bar \p^3\Lambda _ \epsilon (\ve)^3N^3 dS}_{ { \mathcal{J} _1}_b}
\n\\ 
&\ \underbrace{-2\kappa\int_\Gamma\bigl[\bigl(\frac{\be,_1}{\bje^2}- \be,_1(0)\bigr)\times \bar \p^3\bar v,_2\bigr]^3 \bar \p^3 \Lambda _ \epsilon  (\ve)^3 N^3 dS}_{ {\mathcal{J} _1}_c}
-
\underbrace{2\kappa\int_\Gamma \bar \p^3\bar v^2,_2 \, \bar \p^3\Lambda _ \epsilon (\ve)^3N^3 dS}_{{ \mathcal{J} _1}_d}
+\mathcal R_1 \,,
\label{F1}
\end{align}
where $ \mathcal{R} _1$ is a lower-order integral over $\Gamma$ that  contains all of the remaining terms from the action of $\bar \p^3$, so that there
are at most three space derivatives on $\bar v$ on $\Gamma$.     The trace theorem combined with  the Cauchy-Schwarz   inequality
easily show that $| \mathcal{R} _1| \le C |(\ve)^3|_3 \|\bar v\|_4$.     The next crucial observation is that
\begin{equation}\label{csobserve}
{ { \mathcal{J} _1}_b} + { { \mathcal{J} _1}_d} +  \mathcal{J} _3 =0 \,.
\end{equation} 
We use the fundamental theorem of calculus, 
$$
\frac{\be,_2}{\bje^2}(t) -\be,_2(0) = \int_0^t \p_t\frac{\be,_2}{\bje^2} \text{ and }
\frac{\be,_1}{\bje^2}(t)- \be,_1(0) =\int_0^t \p_t \frac{\be,_1}{\bje^2} \,,
$$
to estimate the integrals ${ { \mathcal{J} _1}_a}$ and  ${ { \mathcal{J} _1}_c}$  so that (\ref{F1}) and (\ref{csobserve}) show that
$$
| \mathcal{J} _1 + \mathcal{J} _3| \le C t \| \Lambda _ \epsilon (\ve)^3\|_4 \|\bar v\|_4 +C |(\ve)^3|_3 \|\bar v\|_4\,.
$$

Next, we write the integral $ \mathcal{J}_2$ as
$$
\mathcal{J} _2 = 4 \kappa \int_ \Gamma \bar a^3_3 \bar J^{-3} \bar a^s_r \bar \p^3 \bar v^r,_s \, \bar \p^3 \Lambda _ \epsilon (v^ \epsilon )^3 N^3\, dS + \mathcal{R} _2 \,,
$$
where 
$$ \mathcal{R} _2   \sim \int_\Gamma \bar \p^3 D\bar \eta\, \bar \p^3 \Lambda _ \epsilon (v^ \epsilon )^3 N^3 dS
$$
where the symbol $ \sim$ is used here to mean that $ \mathcal{R} _2$
 is comprised of integrands which have the derivative count of the integrand $\bar \p^3 D\bar \eta\, \bar \p^3 \Lambda _ \epsilon 
(v^ \epsilon )^3$.    It follows that
$$
|\mathcal{R} _2| \le | D\bar \eta|_{2.5} |\Lambda _ \epsilon (v^ \epsilon )^3|_{3.5} \le C \| \bar \eta\|_4 \| \Lambda _ \epsilon v^ \epsilon \|_4 \,,
$$
the last inequality following from the trace theorem.   Since $\bar \eta(t) = e + \int_0^t \bar v$, we see that for some $ \delta >0$,
$$
|\mathcal{R} _2|  \le \mathcal{N} _0 +( \delta + Ct^2) \| \Lambda _ \epsilon v^ \epsilon \|_4^2 +   Ct^2 \| \bar v\|_4^2  \,,
$$
Returning to the remaining term in $\mathcal{J} _2$, we write
\begin{align*} 
&4 \kappa \int_ \Gamma \bar a^3_3 \bar J^{-3} \bar a^s_r \bar \p^3 \bar v^r,_s \, \bar \p^3 \Lambda _ \epsilon (v^ \epsilon )^3 N^3\, dS
=
\underbrace{4 \kappa \int_ \Gamma  \bar \p^3 \operatorname{div}  \bar v \, \bar \p^3 \Lambda _ \epsilon (v^ \epsilon )^3 N^3\, dS}_{ { \mathcal{J} _2}_a} \\
& \qquad\qquad\qquad
+
\underbrace{4 \kappa \int_ \Gamma \bar (a^3_3 \bar J^{-3} \bar a^s_r  - \delta ^s_r)  \bar \p^3 \bar v^r,_s \, \bar \p^3 \Lambda _ \epsilon (v^ \epsilon )^3 N^3\, dS}_{ { \mathcal{J} _2}_b} \,.
\end{align*} 

We estimate the integral $ {\mathcal{J} _2}_a$ with the Cauchy-Schwarz inequality.
The integral $ {\mathcal{J} _2}_b$ can be estimated using the fundamental theorem of calculus:
$$
 |{\mathcal{J} _2}_b| \le C t \| \Lambda _ \epsilon (\ve)^3\|_4 \|\bar v\|_4  \,,
 $$
so that we have established the following estimate:
\begin{align*} 
&| \mathcal{J} _2| + | \mathcal{J} _1 + \mathcal{J} _3|
\le \mathcal{N} _0+ T C_M +( \delta + Ct^2) \| \Lambda _ \epsilon v^ \epsilon \|_4^2 +   Ct^2 \| \bar v\|_4^2   \\
& \qquad \qquad + C t \| \Lambda _ \epsilon (\ve)^3\|_4 \|\bar v\|_4 +C |(\ve)^3|_3 \|\bar v\|_4+  C \|\bar X\|_4^2 + C \|\operatorname{div} \bar v \|_3 \| \Lambda _ \epsilon v^ \epsilon \|_4 \,,
\end{align*} 
where the bound on the integral $ {\mathcal{J} _2}_a$ in contained in the right-hand side.
Finally, the integral $ \mathcal{J} _4$ can be estimated in the same way as $ \mathcal{R} _2$ above, so that with the
identities (\ref{csmainJ}), (\ref{G}), and (\ref{ssneed78}) we have shown that
\begin{align}
& \bigl|\int_\Gamma \bar \p^3 (\rho_0,_3\Lambda _ \epsilon  F) \bar \p^3 (\ve)^3 \frac{N^3}{\rho_0,_3}dS\bigr|
\le \mathcal{N} _0 + TC_M +( \delta + Ct^2) \| \Lambda _ \epsilon v^ \epsilon \|_4^2 +   Ct^2 \| \bar v\|_4^2 \n  \\
& \qquad \qquad + C t \| \Lambda _ \epsilon (\ve)^3\|_4 \|\bar v\|_4 +C |(\ve)^3|_3 \|\bar v\|_4+  C\| \bar X \|_4^2 + C \|\operatorname{div} \bar v \|_3 \| \Lambda _ \epsilon v^ \epsilon \|_4\,.
\label{Eeps2}
\end{align}
We now turn our attention to the second term on the left-hand side of (\ref{Eeps1}), which will  give us as a sign-definite energy term plus a small perturbation.
We first notice that by the properties of the boundary convolution operator $ \Lambda _ \epsilon $,
\begin{equation}
\label{Eeps3}
\int_\Gamma \bar \p^3\bigl[ \Lambda _ \epsilon ^2(\ve,_1^1+\ve,_2^2)\bigr] \bar \p^3 (\ve)^3 N^3\,dS=
\underbrace{\int_\Gamma \bar \p^3 \Lambda _ \epsilon (\ve,_1^1+\ve,_2^2)\,  \bar \p^3  \Lambda _ \epsilon \ve \cdot  N \ dS}_{ \mathcal{I} }\,.
\end{equation}

The divergence theorem applied to the integral $\mathcal{I} $ (as our domain $ \Omega = \mathbb{T}  ^2 \times (0,1)$) implies that
\begin{equation}
\n
 \mathcal{I} =
 \underbrace{\int_\Omega \bar \p^3 \Lambda _ \epsilon (\ve,_1^1+\ve,_2^2)\, \bar \p^3 \Lambda _ \epsilon (\ve)^3,_3 dx }_{ \mathcal{I} _1}
 +\underbrace{\int_\Omega \bar \p^3\bigl[\Lambda _ \epsilon (\ve,_{13}^1+\ve,_{23}^2)\bigr] \bar \p^3 \Lambda _ \epsilon (\ve)^3dx}_{ \mathcal{I} _2}\,.
\end{equation}
Now,
\begin{align*}
\mathcal{I} _1=&-\int_\Omega \bigl| \bar \p^3  \Lambda _ \epsilon (\ve)^3,_3\bigr|^2 dx +\int_\Omega \Lambda _ \epsilon  \bar \p^3 \operatorname{div}\ve\,  \bar \p^3  \Lambda _ \epsilon \ve,_3^3 dx \,,
\end{align*} 
and
\begin{align*} 
\mathcal{I} _2= &-\int_\Omega \bar \p^3 \Lambda _ \epsilon (\ve)^1,_{3}\, \bar \p^3  \Lambda _ \epsilon (\ve)^3,_1dx
-\int_\Omega \bar \p^3 \Lambda (\ve)^2,_{3}\,  \bar \p^3  \Lambda _ \epsilon (\ve)^3,_2 dx\,,
\end{align*} 
from which it follows that
\begin{align} 
\mathcal{I} 
=&-\int_\Omega \bigl| \bar \p^3 \Lambda_ \epsilon D(\ve)^3\bigr|^2 dx +\int_\Omega  \Lambda _ \epsilon  \bar \p^3 \operatorname{div}\ve\,
 \bar \p^3  \Lambda _ \epsilon (\ve)^3,_3 dx \n\\
&+\int_\Omega  \Lambda _ \epsilon \bar \p^3[\operatorname{curl}\ve \cdot e_2] \bar \p^3  \Lambda _ \epsilon (\ve)^3,_1dx
 -\int_\Omega \Lambda _ \epsilon  \bar \p^3 [\operatorname{curl}\ve  \cdot e_1] \bar \p^3 \Lambda _ \epsilon (\ve)^3,_2 dx \label{Eeps6} \,.
\end{align}

Now, thanks to (\ref{defvea}), we have for all $t\in [0,T_\epsilon]$
\begin{align}
\|\operatorname{div} \ve\|_3^2\le & C t^2 \|\bar v\|_4^2 +C\|\bar v\|_3^2 +C \bigl\|\frac{\bar X}{\rho_0}\bigr\|_3^2+Ct\|\bar v\|^2_{L^2(0,t;H^4(\Omega))}\n\\
\le &  C t^2 \|\bar v\|_4^2 +C\|\bar v\|_3^2 +Ct\|\bar v\|^2_{L^2(0,t;H^4(\Omega))} + C\| \bar X \|_4^2  \,,\label{Eeps7}
\end{align} 
where we have used our higher-order Hardy inequality Lemma \ref{Hardy}  for the second
inequality.

Next, with (\ref{defveb}), we see that for all $t\in [0,T_\epsilon]$
\begin{align} 
\|\operatorname{curl} \ve\|_3\le & C t \|\bar v\|_4 +C\|\bar v\|_3 +C\|u_0\|_4+C\sqrt{t}\|D(\bar\Xi(\bee))\|_{L^2(0,t;H^3(\Omega))} +C\sqrt{t}\|\bar v\|_{L^2(0,t;H^4(\Omega))}\n\\
\le & C t \|\bar v\|_4 + C_\kappa \|u_0\|_4+C\sqrt{t}\|\bar v\|_{L^2(0,t;H^4(\Omega))}\ ,\label{Eeps8}
\end{align}
where we have used (\ref{may11.5}) and  the identity (\ref{Xi}), relating $\Xi(\bar \eta)$ to $\bar v$
and where we have relied crucially on the chain-rule which shows that
$$\Xi^j,_r(\be)=\bAe^l,_r [\Xi(\be)]^j,_l\,.$$
Note that (\ref{Xi}) provides us with a bound which is $ \epsilon $-independent,
but which  indeed depends on $ \kappa $.

The  action of the boundary convolution operator $ \Lambda _ \epsilon $  does not affect these  estimates; thus, 
using Proposition \ref{prop_X}, we see that
\begin{align}
\int_0^T \|\operatorname{div}  \Lambda _ \epsilon \ve\|_3^2dt & \le  \mathcal{N} _0 +   \pbv +  \pcurl \,,\label{Eeps9}\\
\int_0^T \|\operatorname{curl}  \Lambda _ \epsilon \ve\|_3^2dt & \le   \mathcal{N} _0 +   \pbv \,.\label{Eeps10}
\end{align}

We integrate the inequality (\ref{Eeps1}) from $0$ to $t$, and we use the estimates
 (\ref{Eeps2}), (\ref{G0}), (\ref{Eeps6}), (\ref{Eeps9}), (\ref{Eeps10}) together with
 the fact that $\frac{N_3}{\rho_0,_3}\ge C>0$ on $\Gamma$,
to obtain that for any $t\in [0,T]$,
\begin{align*}
& |\bar \p^3 (\ve)^3(t)|_0^2+\int_0^t \|\bar \p^3D \Lambda _ \epsilon (\ve)^3\|_0^2
\le   \mathcal{N} _0 
+ C t |(\ve)^3|^2_3   +C t\|\bar v\|^2_\ZT +C t \int_0^t \| \Lambda _ \epsilon (\ve)^3\|^2_4\\
&   +C\sqrt{t}\int_0^t \|\bar v\|^2_4 + \frac{C}{\sqrt{t}}\int_0^t|(\ve)^3|^2_3
+ \delta  \| \Lambda _ \epsilon v^ \epsilon \|_4^2 +  \pbv +  \pcurl + C \|  \operatorname{div} \bar v \|^2_\YT\,,
\end{align*}
 where we have used the Cauchy-Young
inequality $ab \le \frac{C}{\sqrt{t}}a^2 + \sqrt{t} b^2$ for $a,b\ge 0$.
By taking $ \delta >0$ sufficiently small, 
and using the relations (\ref{Eeps9}) and (\ref{Eeps10}), this implies the following inequality:
\begin{align*}
& |\ve(t)|_{3}^2+\int_0^t \| \Lambda _ \epsilon \ve\|_4^2   
 \le   \mathcal{N} _0 
+ C t |(\ve)^3|^2_3   +C t\|\bar v\|^2_\ZT +C t \int_0^t \| \Lambda _ \epsilon (\ve)^3\|^2_4  +C\sqrt{t}\int_0^t \|\bar v\|^2_4 \\
&   \qquad  \ \   +C\sqrt{t}\int_0^t \|\bar v\|^2_4 + \frac{C}{\sqrt{t}}\int_0^t|(\ve)^3|^2_3
 +  \pbv +  \pcurl + C \|  \operatorname{div} \bar v \|^2_\YT \\
& \qquad\le  \mathcal{N} _0 +C (t+\sqrt{t}) \sup_{t \in [0,T]}  |\ve|^2_3+C (t+\sqrt{t})\|\bar v\|^2_\ZT
+C t \int_0^t \| \Lambda _ \epsilon \ve\|^2_4  \\
&   \qquad  \ \   +\pbv +  \pcurl + C \|  \operatorname{div} \bar v \|^2_\YT \,.
\end{align*}
(Note that due to the presence of the convolution operators $ \Lambda _ \epsilon $  in the definition of  $(v_t^ \epsilon )^3$ on $\Gamma$ in formula (\ref{defvec}), it is clear that  $\sup_{[0,t]}|\ve|^2_3$  is bounded by some finite number, which  a priori depends on $\epsilon$.)

Since we are considering a bounded time interval, $0 \le t \le C \sqrt{t}$, and we will henceforth make use of this fact.
 The previous estimate then implies that 
\begin{align}
\sup_{t \in [0,T]} |\ve(t)|_{3}^2+\int_0^T \|\Lambda _ \epsilon  \ve\|_4^2dt & \le   \mathcal{N} _0 + \sqrt{T}P(\|\ve\|^2_{\mathcal{X}_T^4}) +  \pbvs \n \\
& \qquad
 +  \pcurl + C \|  \operatorname{div} \bar v \|^2_\YT \,.\label{Eeps11}
\end{align}

Estimates for the time-differentiated quantities are, in fact, very straightforward at this stage. By using the expression (\ref{defvec}), we see that thanks to the estimate (\ref{Eeps11}),
\begin{align}
\int_0^T |(\ve_t)^3|_{2.5}^2 dt &\le \mathcal{N} _0 + C \int_0^T | \Lambda _ \epsilon \ve|_{3.5}^2dt   + \pbvs    \n\\
&\le \mathcal{N} _0 + \sqrt{T}P(\|\ve\|^2_{\mathcal{X}_T^4}) +  \pbvs +  \pcurl + C \|  \operatorname{div} \bar v \|^2_\YT  \,.
\label{Eeps12}
\end{align}

Now, just as we estimated  the divergence and curl  of $\ve$ in  (\ref{Eeps7}) and (\ref{Eeps8}), we can repeat this procedure to estimate
the divergence and curl of $v_t^ \epsilon $.   By using (\ref{defvepsilona}) and (\ref{defvepsilonb}), we also have that
\begin{align*}
\int_0^T\|\operatorname{div} \ve_t\|_2^2dt &\le  \mathcal{N} _0 +  \pbvs+  \pcurl + C \|  \operatorname{div} \bar v \|^2_\YT\,,\\
\int_0^T \|\operatorname{curl} \ve_t\|_2^2 dt &\le\mathcal{N} _0 +  \pbvs+  \pcurl + C \|  \operatorname{div} \bar v \|^2_\YT\,,
\end{align*}
which in addition to the normal trace estimate (\ref{Eeps12}) provides the estimate:
\begin{align}
\int_0^T \|\ve_t\|_{3}^2 dt \le \mathcal{N} _0 + \sqrt{T}P(\|\ve\|^2_{\mathcal{X}_T^4}) +  \pbvs +  \pcurl + C \|  \operatorname{div} \bar v \|^2_\YT\,
\label{Eeps13}
\end{align}

We proceed in a similar fashion to estimate  $v^ \epsilon _{tt}$, by considering the time-differentiated version of (\ref{defvepsilon}), and using (\ref{Eeps11}) and (\ref{Eeps13}). This yields the following inequality:
\begin{align}
\int_0^T \|\ve_{tt}\|_{2}^2dt \le  \mathcal{N} _0 + \sqrt{T}P(\|\ve\|^2_{\mathcal{X}_T^4}) +  \pbvs +  \pcurl + C \|  \operatorname{div} \bar v \|^2_\YT \,.
\label{Eeps14}
\end{align}
Finally, by using the second time-differentiated version of (\ref{defvepsilon}) and using (\ref{Eeps11}), (\ref{Eeps13}) and (\ref{Eeps14}), we also have that 
\begin{align}
\int_0^T \|\ve_{ttt}\|_{1}^2 dt \le   \mathcal{N} _0 + \sqrt{T}P(\|\ve\|^2_{\mathcal{X}_T^4}) +  \pbvs +  \pcurl + C \|  \operatorname{div} \bar v \|^2_\YT \,.
\label{Eeps15}
\end{align}
The estimate (\ref{Eeps11}) together with (\ref{Eeps13}), (\ref{Eeps14}), and  (\ref{Eeps15}) provide us with
\begin{align}
\|\ve\|_{\mathcal{X}_T^4}^2 \le \mathcal{N} _0 + \sqrt{T}P(\|\ve\|^2_{\mathcal{X}_T^4}) +  \pbvs +  \pcurl + C \|  \operatorname{div} \bar v \|^2_\YT  \,,
\label{Eeps16}
\end{align}
where the polynomial functions $P$ on the right-hand side are  independent of $\epsilon$. 

Thanks to our  polynomial estimates  in Section \ref{subsec_poly}, we  infer from (\ref{Eeps16}) the existence of $T>0$ (which is independent of $\epsilon$) such that $\ve\in \mathcal{X}_T^4$ and satisfies the estimate:
\begin{align}
\|\ve\|_{\mathcal{X}_T^4}^2 \le  2 \mathcal{N} _0 +2\pbvs +  \pcurl + 2C \|  \operatorname{div} \bar v \|^2_\YT \,.
\label{Eeps17}
\end{align}
(We will readjust the constant $C$ and the polynomial functions $P$ to absorb the multiplication by $2$.)

\noindent {\bf Step 3: The limit as $\epsilon \to 0$ and the fixed-point  of the map $\bar v \mapsto v$}

We set  $\epsilon=\frac{1}{n}$,  $n\in \mathbb{N}=\{1,2,3,...\} $.  From (\ref{Eeps17}), there exists a subsequence (still denoted by $\epsilon$) and
a vector field $v\in \mathcal{X} ^3_T$, $V\in L^2(0,T;H^4(\Omega))\cap L^\infty (0,T;H^3(\Omega))$ such that 
\begin{subequations}
\label{Eeps18}
\begin{align}
\ve&\rightharpoonup v\ \ \text{in}\  \mathcal{X} ^3_T\ ,\\
\ve&\rightarrow v\ \ \text{in}\ \mathcal{X} ^2_T\ ,\label{Eeps18b}\\
\Lambda _ \epsilon  \ve &\rightharpoonup V\ \text{in}\ L^2(0,T;H^4(\Omega))\ .
\end{align}
\end{subequations}
where the space $ \mathcal{X} ^3_T$ is defined in (\ref{csX3}) and
\begin{align}
 \mathcal{X} ^2_T&=\{w \in  L^2(0,T;H^2(\Omega))  \ : \  w_t\in L^2(0,T;H^2(\Omega)),\  w_{tt}\in L^2(0,T;H^1(\Omega)), \n \\
&\qquad  (x_1,x_2) \mapsto w \text{ is $1$-periodic}\} \n \,.
\end{align}

Next, we notice that for any $\varphi\in  \mathcal{C} ^ \infty _0(\Omega)$, the space of smooth functions with compact support in $\Omega$, we have for each $i=1, 2, 3$ and $t\in [0,T]$, $T$ still depending
on $ \kappa >0$, that
\begin{equation}
\label{Eeps19}
\lim_{\epsilon\rightarrow 0} \int_{\Omega} \Lambda _ \epsilon (\ve)^i\cdot \varphi dx = \lim_{\epsilon\rightarrow 0} \int_{\Omega} (\ve)^i\cdot \Lambda _ \epsilon  \varphi dx =\int_\Omega v^i \varphi dx\,,
\end{equation}
where we used the fact that $ \Lambda _ \epsilon  \varphi \rightarrow \varphi$ in $L^2(\Omega)$. This shows us that $v=V$, and that
\begin{align}
\|v\|_\XT^2 \le   \mathcal{N} _0 + \pbvs +  \pcurl + C \|  \operatorname{div} \bar v \|^2_\YT \,.
\n
\end{align}
The estimates weighted by $\rho_0$ in the definition of the  $\ZT$-norm  follow immediately from 
multiplication by $\rho_0$ of the equations (\ref{defva}) and (\ref{defvb}); because $\rho_0$ vanishes on $\Gamma$ and using the
(unweighted) 
estimates already obtained, there is no need to consider the parabolic equation (\ref{defvc}), so that
\begin{align}
\|v\|_\ZT^2 \le   \mathcal{N} _0 + \pbvs +  \pcurl + C \|  \operatorname{div} \bar v \|^2_\YT \,.
\label{Eeps20}
\end{align}
Moreover, the convergence in  (\ref{Eeps18}) and the definition of the sequence of problems (\ref{defve}) easily show us that $v$ is a solution of the problem (\ref{defv});  furthermore, we see that we can obtain for the system (\ref{defv}) the same type of energy estimates as in Step 2 
above. This shows the uniqueness of the solution $v$ of (\ref{defv}), and hence allows us to define a mapping $\Theta: \bar v \in \ZT \rightarrow v\in \ZT$.

We next launch an iteration scheme.  We choose any $v^{(1)} \in \mathcal{C} _T(M)$ and define for $n \in \mathbb{N}  $,
$$
 v^{(n+1)} = \Theta ( v^{(n)}) \,, \ \ \ \ v^{(n)}|_{t=0} = u_0 \,.
$$
For each $n \in \mathbb{N}  $ we set $\eta^{(n)} (x,t) = x + \int_0^t v^{(n)}(x,t') dt'$, $A^{(n)} = [D \eta^{(n)}] ^{-1} $, $J^{(n)} = \det
D \eta^{(n)}$,  $a^{(n)} = J^{(n)} A^{(n)}$, $X^{(n)}$ is the solution to (\ref{ssX}) with $v^{(n)}$, $a^{(n)}$, $J^{(n)}$, and $A^{(n)}$
replacing $ \bar v$, $\bar a$, $\bar J$, and $\bar A$, respectively.  Similarly, we define $\Xi^{(n)}(\eta^{(n)})$ via equation (\ref{Xi})
with $v^{(n)}$ replacing $\bar v$; we define ${\mathfrak C}^{(n)}$ via equations (\ref{may11.1}) and (\ref{may11.7}) with $v^{(n)}$ replacing $\bar v$.

According to (\ref{Eeps20}), for 
\begin{equation}\label{need200}
\|v^{(n+1)}\|_\ZT^2 \le   \mathcal{N} _0+ \sqrt{T}P(\|v^{(n)}\|_\ZT^2 ) + P( \| \operatorname{curl} v^{(n)}\|^2_\YT) + C \|  \operatorname{div}  v^{(n)} \|^2_\YT \,.
\end{equation} 
From (\ref{defva}), 
$$
\operatorname{div} v^{(n)} = \operatorname{div}v^{(n-1)} -\operatorname{div}_{\eta^{(n-1)}}v^{(n-1)} + \frac{{J^{(n-1)}}^2 X^{(n-1)}}{\rho_0}
$$
so that
\begin{align} 
\| \operatorname{div} v^{(n)} \|^2_\YT  
& \le  \| \operatorname{div}v^{(n-1)} -\operatorname{div}_{\eta^{(n-1)}}v^{(n-1)} \|^2_\YT  + \| \frac{{J^{(n-1)}}^2 X^{(n-1)}}{\rho_0}\|^2_\YT \n\\
&  \le  
\mathcal{N} _0 +  \sqrt{T} P(\| v^{(n-1)}\|^2_\ZT) + P(\| \operatorname{curl}  v^{(n-1)}\|^2_\YT) \,, \label{need201}
\end{align} 
where we have used our higher-order Hardy inequality Lemma \ref{Hardy} and Proposition \ref{prop_X} for the second inequality.  Next,
we use (\ref{defvb}) and write
\begin{align*} 
&\operatorname{curl}  v^{(n)}  =\operatorname{curl} u_0 + \operatorname{curl}v^{(n-1)} -\operatorname{curl}_{\eta^{(n-1)}}v^{(n-1)} \\
&   \ \ \ \ + 2\kappa \int_0^t\varepsilon_{\cdot ji}   v^{(n-1)},_s^r  { A^{(n-1)}}^s_i\  \Xi^{(n-1)},_r^j(\eta^{(n-1)})  
 + \int_0^t \varepsilon_{\cdot ji}  v^{(n-1)},_s^i   \p_t{ A^{(n-1)}}^s_j  + \int_0^t {\mathfrak C}^{(n-1)}\,.
\end{align*} 
It then follows, using (\ref{Xi}) and (\ref{may11.5}), that
\begin{align} 
\| \operatorname{curl} v^{(n)} \|^2_\YT  \le  
\mathcal{N} _0 +  \sqrt{T} P(\| v^{(n-1)}\|^2_\ZT) \,. \label{need202}
\end{align} 
Combining the estimates (\ref{need200}), (\ref{need201}), and (\ref{need202}), we obtain the inequality
$$
\|v^{(n+1)}\|_\ZT^2 \le   \mathcal{N} _0 + \sqrt{T}P(\|v^{(n)}\|_\ZT^2 )  + \sqrt{T} P(\| v^{(n-1)}\|^2_\ZT) + \sqrt{T} P(\| v^{(n-2)}\|^2_\ZT)
$$
This shows that by choosing $T>0$ sufficiently small and $M >> \mathcal{N} _0$ sufficiently large,  the convex set $ \mathcal{C} _{T }(M)$ is stable under the action of $\Theta$.

 In order to see that $\Theta$ has a fixed-point, we simply notice that by proceeding in a similar fashion as in {\bf Step 2} above (for the
 $ \epsilon $-independent energy estimates), and by using the inequality (\ref{rem_diff})
\begin{align}
\| v^{(n+1)} - v^{(n)} \|^2_\ZT  & \le \sqrt{T}\ P(\| v^{(n)} - v^{(n-1)}\|^2_\ZT)+  \sqrt{T}\ P(\| v^{(n-1)} - v^{(n-2)}\|^2_\ZT)   \n \\
& \qquad + \sqrt{T}\ P(\| v^{(n-2)} - v^{(n-3)}\|^2_\ZT) \,,
\label{contraction}
\end{align}
where  the  polynomial function $P$ can be chosen under the form $P(z) =  \sum_{j=1}^m a_j z^j$ for some integer $m \ge 1$.

Although, the inequality (\ref{contraction}) is not exactly the usual hypothesis of the contraction mapping theorem, the identical
argument shows that for $T=T_ \kappa $ taken sufficiently small,  the map $\Theta$ is a contraction, and possesses a unique fixed-point
$v$ satisfying $v = \Theta(v)$.
 We will next
prove that this unique fixed-point $v$ is the  unique solution of the $ \kappa $-problem (\ref{approx}).

\subsubsection{The fixed-point of the map $\bar v\mapsto v$ is a
solution of the $\kappa$-problem} In a straightforward manner, we deduce
from (\ref{defv}) the following relations for our fixed-point $v=\bar v$:
\begin{subequations}
\label{defvf}
\begin{alignat}{2}
\operatorname{div}_{\eta} v_t&=\frac{[X  J^2]_t}{\rho_0}-\p_t A_i^j
v^i,_j\ \  && \text{ in }  \Omega\ ,\label{defvfa}\\
 \operatorname{curl}_{\eta}  v_t &=   2 \kappa \varepsilon_{\cdot ji}
v,_s^r   A^s_i\ \Xi,_r^j(\eta) + {\mathfrak C} \ && \text{ in } \Omega \ ,\label{defvfb}\\
v_t^3 &=   -2 J^{-2}  a^3_3  \rho_0,_3 - 2 \kappa [ J^{-2}  a^3_3]_t \rho_0,_3+ c(t) N^3  \qquad && \text{ on } \Gamma \ ,\label{defvfc}\\
\int_ \Omega v_t^ \alpha dx& = -2 \int_ \Omega  A^k_ \alpha (\rho_0  J^{-1} ),_k dx 
-2 \kappa  \int_ \Omega \p_t[  A^k_ \alpha (\rho_0  J^{-1} ),_k] dx  \,,  \label{defvfd} \\
(x_1,x_2)& \mapsto v_t(x_1,x_2,x_3,t)  \text{ is 1-periodic  }    \,,   \label{defvfe} 
\end{alignat}
\end{subequations}
where $X$ is a solution of (\ref{heatX}) and  where the function $c(t)$ in (\ref{defvfc})  is defined by
\begin{align*}
-2 c(t)=&\int_\Omega (-\operatorname{div} v_t  +\operatorname{div}_{\eta}
v_t) dx-\int_\Omega \frac{[X J^2]_t}{\rho_0}dx+ \int_\Omega \p_t A_i^j
v,_j^i  dx -2\int_\Gamma J^{-2} a_3^3 N^3\rho_0,_3dS\\
&\qquad -2\kappa\int_\Gamma  J^{-2} \p_t a^3_3 N_3 {\rho_0,_3}dS
 -2\kappa\int_\Gamma a_3^3 \p_t J^{-2}  N_3  \rho_0,_3 dS 
\\
 =&\int_\Omega (-\operatorname{div}   v_t +\operatorname{div}_{\eta}
 v_t) dx -\int_\Omega \frac{[X  J^2]_t}{\rho_0} dx + \int_\Omega \p_t A_i^j
v,_j^idx  \\
& \qquad  -2\int_\Gamma D\rho(\eta)\cdot N\,dS -2\kappa\int_\Gamma
[D\rho(\eta)]_t\cdot N\, dS\,,
\end{align*}
where $D \rho(\eta) = A^k_ \cdot (\rho_0 J^{-1} ),_k$, and $dS = dx_1dx_2$.
By using (\ref{defvfa}) and the divergence theorem, we therefore obtain the identity
\begin{equation}
\label{cf}
c(t) =  {\frac{1}{2}} \int_\Gamma [v_t+2D\rho(\eta)+2\kappa
[D\rho(\eta)]_t]\cdot N\, dS \,.
\end{equation}
The fixed-point of the map $\bar v \mapsto v$  (which we are labeling $v$ as well) also satisfies the equation
\begin{subequations}
\label{odef}
\begin{align}
v_t+2 \Xi(\eta)+2 \kappa [\Xi(\eta)]_t&=0,\label{odea}\\
\Xi(0)=D\rho_0.\label{odeb}
\end{align}
\end{subequations}
It is  thus clear from (\ref{defvfc}) and  (\ref{odef}) that  the fixed-point is a solution to the
$\kappa$-problem (\ref{approx}), {\bf if} we can prove that 
\begin{equation}\label{ss2eqs}
c(t)=0 \text{ and } \Xi=D\rho \,,
\end{equation} 
where
we remind the reader that
\begin{equation}
\rho(\eta)=\rho_0 J^{-1} \,.
\end{equation}
This is in fact the case, and we now  explain why (\ref{ss2eqs}) holds.

\vspace{.1in}
\noindent{\bf Step 1.}
We first apply $\operatorname{curl}_\eta$ to each term of (\ref{odea})
and compare the resulting equation with (\ref{defvfb}). This implies that
\begin{equation}
\operatorname{curl}_\eta [\Xi(\eta)] +\kappa [\operatorname{curl}_\eta
[\Xi_t(\eta)]+u_i(\eta) \operatorname{curl}_\eta [(\p_i\Xi)(\eta)]]=- {\frac{1}{2}} {\mathfrak C},
\end{equation}
which thanks to the fact that by definition of
$\operatorname{curl}_\eta$,
\begin{equation}
\operatorname{curl}_\eta
[\Theta(\eta)]=[\operatorname{curl}\Theta](\eta),
\end{equation}
for any vector field $\Theta$,
provides us with
\begin{equation}
[\operatorname{curl} \Xi](\eta) +\kappa
[[\operatorname{curl}\Xi_t](\eta)+u_i(\eta)
[\operatorname{curl}(\p_i\Xi)](\eta)]=- {\frac{1}{2}} {\mathfrak C},
\end{equation}
which shows that
\begin{equation}
\label{gronwallcurl}
[\operatorname{curl} \Xi](\eta) +\kappa
[\operatorname{curl}\Xi(\eta)]_t=-D\psi_e(\eta)-\kappa \bigl[D\psi_e(\eta)\bigr]_t,
\end{equation}
where $\psi_e$ denote the Eulerian version of $\psi$, given by
\begin{equation}
\label{may11.6}
\psi_e\circ\eta=\psi\,.
\end{equation}
According to  (\ref{odeb}), $\Xi(0)=D\rho_0$; thus, we have that $[\operatorname{curl}
\Xi](\eta)(0)=0$. Furthermore, by our definition (\ref{may11.2}), we have $D\psi_e(\eta)|_{t=0}=0$ in $\Omega$, which with (\ref{gronwallcurl}) allows us to
conclude that for $t\in [0,T]$,
\begin{equation*}
[\operatorname{curl} \Xi+D\psi_e](\eta)(t)=0.
\end{equation*}
We may therefore consider the following elliptic problem:
\begin{align*}
\Delta \psi_e=-\operatorname{div}(\operatorname{curl}\Xi)=0\ \text{ in }\  \eta(t)(\Omega)\,,\\
\psi_e=0 \ \text{ on }\ \eta(t)(\Gamma)\,, \\
(x_1,x_2) \mapsto \psi_e(x_1,x_2,x_3,t) \text{ is $1$-periodic} \,,
\end{align*}
which shows that $\psi_e=0$ and hence ${\mathfrak C}=0$.
Therefore, $\operatorname{curl} \Xi=0$ in $\eta(t,\Omega)$ and there exists a scalar function $Y(t,\cdot)$ defined on
$\eta(t,\Omega)$ such that
\begin{equation}
\label{Y}
\Xi=DY.
\end{equation}

It remains to establish that $D Y=D\rho$. We will first
prove that a Neumann-type boundary condition plus a small tangential
perturbation holds for $Y-\rho$; namely, we will show  that $(Y-\rho),_3=0$ on
$\eta(t,\Gamma)$.

\vspace{.1 in}
\noindent{\bf Step 2.} We  take the scalar product of (\ref{odea}) with
$e_3$ to find that
\begin{equation*}
v_t ^3+2DY(\eta)\cdot e_3 +2\kappa [DY)(\eta)]_t  \cdot e_3=0,
\end{equation*}
which, by comparison with
(\ref{defvfc}), yields the following identity on $\Gamma$:
\begin{align}
&2 \Bigl[D(Y-\rho) (\eta) +2\kappa [D(Y-\rho)(\eta)]_t\Bigr] \cdot e_3 =-c(t) N^3 \n \\
& \qquad\qquad\qquad = \int_\Gamma \Bigl[D(-\rho+Y)(\eta)+\kappa
[D(-\rho+Y)(\eta)]_t\Bigr]\cdot N \,dS \ N^3\,,\label{bc1}
\end{align}
where we have used the expression (\ref{cf}) for $c(t)$.

Therefore, since $N=(0,0,N^3)$ on $\Gamma$, this implies:
\begin{equation*}
\int_\Gamma [D(-\rho+Y)(\eta)+\kappa
[D(-\rho+Y)(\eta)]_t]\cdot N\, dS =\frac{1}{2}\int_\Gamma
[D(-\rho+Y)(\eta)+\kappa [D(-\rho+Y)(\eta)]_t]\cdot N\,dS \,,
\end{equation*}
and thus that
\begin{equation}
\label{czero}
-c(t)=\int_\Gamma [D(-\rho+Y)(\eta)+\kappa
[D(-\rho+Y)(\eta)]_t]\cdot N \,dS=0\,.
\end{equation}
The identity (\ref{czero}) together with (\ref{bc1}) then implies that
\begin{equation*}
\Bigl[ D (Y-\rho) (\eta) +\kappa [D(Y-\rho)(\eta)]_t\Bigr] \cdot e_3 =0\ \text{ on
} \Gamma\,.
\end{equation*}
Therefore,
\begin{equation*}
\bigl[(Y-\rho),_3 (\eta)\bigr](t,\cdot)=e^{-\frac{t}{\kappa}} (Y-\rho_0),_3(0)\ \text{ on }
\Gamma,
\end{equation*}
which with (\ref{Y}) and $\Xi(0)=D\rho_0$ shows that
\begin{equation}
\label{Ygamma}
(Y-\rho),_3=0,\ \ \text{on}\ \eta(t,\Gamma)\ .
\end{equation}

\vspace{.1 in}
\noindent{\bf Step 3.} We now apply $\operatorname{div}_\eta$ to
(\ref{odea}), and compare the resulting equation  with (\ref{ssX.a}). Using
(\ref{defvfa}) and the fact that  $X(0)= \rho_0 \operatorname{div} u_0$, we have that
\begin{equation}
X=\rho_0 J^{-2} \operatorname{div}_\eta v \,.
\end{equation}
By denoting
\begin{equation}
q=\rho-Y,
\end{equation}
this leads us to:
\begin{equation}
\operatorname{div}_\eta [D q(\eta)+\kappa D q_t(\eta)+\kappa
u_i(\eta) (D q),_i(\eta)]=0\ \ \text{in}\ \Omega.
\end{equation}
This is equivalent in $[0,T]\times \Omega$ to:
\begin{equation*}
\Delta  q(\eta)+\kappa \Delta  q_t (\eta)+\kappa u_i(\eta)\Delta 
q,_i (\eta)+
A_l^j v,_j^i q,_{li} (\eta)=0\,,
\end{equation*}
or equivalently,
\begin{equation}
\label{grlap}
\Delta  q(\eta)+\kappa[\Delta  q (\eta)]_t +
A_l^j v,_j^i q,_{li} (\eta)=0 \,.
\end{equation}
Now since $\rho_0=Y(0)$, we have that
\begin{equation}
\label{lap0}
\Delta q(0)=0\ \ \text{in}\ \Omega\,.
\end{equation}
Also, from (\ref{Ygamma}), we have the perturbed Neumann boundary
condition
\begin{equation}
\label{qgamma}
 q,_3=0,\ \ \text{on}\ \eta(\Gamma).
\end{equation}
By (\ref{defvfd}), (\ref{odef}), we obtain that for $\alpha=1,2$,
\begin{equation*}
\int_{\Omega} q,_\alpha(\eta)  dx +\kappa\int_\Omega [q,_\alpha(\eta)]_t  dx=0\ ,
\end{equation*}
or equivalently,
\begin{equation*}
\int_{\Omega} q,_\alpha(\eta) dx+\kappa\p_t\int_\Omega q,_\alpha(\eta) dx=0\,,
\end{equation*}
which together with the initial condition $\displaystyle\int_\Omega
q,_\alpha(0)dx=0$  implies that
\begin{equation}
\label{qaverage}
\int_{\Omega} q,_\alpha(\eta)dx=0\,.
\end{equation}
Therefore, by setting $f= \Delta q$, we have for all $t\in [0,T]$ the
system:
\begin{subequations}
\begin{align}
\Delta  q&=f\ \text{in}\ \eta(\Omega)\ ,\label{whya}\\
\int_{\eta(\Omega)}  J^{-1} q,_\alpha dx &=0\ ,\label{whyb}\\
q,_3&=0\ \text{on}\ \eta(\Gamma)\ ,\label{whyc}\\
D q &\text{ is 1-periodic in the directions $e_1$ and $e_2$}\
.\label{whyd}
\end{align}
\end{subequations}
Note that because of the periodicity of $v$, the domain $\eta(\Omega)$
is such that $\eta(1,x_2,x_3)=\eta(0,x_2,x_3)+(1,0,0)$, and
$\eta(x_1,1,x_3)=\eta(x_1,0,x_3)+(0,1,0)$, which explains why condition
(\ref{whyd}) holds.

We now take the vertical derivative $\p_3$ of (\ref{whya}), multiply the
resulting equation by $q,_{3}$ and integrate by parts in $\eta(\Omega)$,
using the conditions (\ref{whyc}) and (\ref{whyd}). This yields:
\begin{equation*}
\int_{\eta(\Omega)} |D q,_3|^2 dx=-\int_{\eta(\Omega)} f\ q,_{33} dx\,,
\end{equation*}
which provides us with the estimate
\begin{equation}
\label{why1}
\|D q,_3\|_{0,\eta(\Omega)}\le C \|f\|_{0,\eta(\Omega)}\,,
\end{equation}
where we are using the notation $\| \cdot \|_{0,\eta(\Omega)} = \| \cdot \|_{L^2(\eta(\Omega))}$.

We now write (\ref{whya}) under the form
\begin{equation}
\label{why1bis}
q,_{11}+q,_{22}=g \text{ where } g:= -q,_{33}+f \,.
\end{equation}
It follows from (\ref{why1bis}) that
$$
\int_{ \eta ( \Omega )} (q,_{11}+q,_{22})\, (q,_{11}+q,_{22})  dx = \int_{ \eta( \Omega )} |g|^2 dx \,.
$$
Integration by parts on the left-hand side of this equation, together with the periodicity of $Dq$ and
its  derivatives, shows that
\begin{equation}
\label{why2}
\int_{\eta(\Omega)} q,_{\alpha\beta}q,_{\alpha\beta} dx+\int_{\eta(\Gamma)}
q,_{11}q,_2 n_2(t) dS(t) -\int_{\eta(\Gamma)} q,_{12}q,_2
n_1(t)dS(t)=\int_{\eta(\Omega)} |g|^2 dx \,,
\end{equation}
where the notation $dS(t)$ denotes the naturally induced surface measure on $\eta(\Gamma)$.   Note that
the presence of these surface integrals is due to the fact that $\eta(\Gamma)$ is no longer necessarily horizontal
for $t>0$, so that integration by parts in purely horizontal directions produces boundary contributions.

We now notice that
\begin{equation*}
\int_{\eta(\Gamma)} q,_{11}q,_2 n_2(t) dS(t)=\int_{\eta(\Gamma)} q,_{11}q,_2
\frac{n_2(t)}{n_3(t)} n_3(t) dS(t)\,,
\end{equation*}
where division by $n_3(t)$ is bounded since for $t$ taken sufficiently small,
$|n_3(t)|$ is very close to $1$ on $\eta(\Gamma)$.

Next, for $i=1,2,3$, we smoothly extend $n_i(t)$ into $\eta(\Omega)$, and note the
integration-by-parts (with respect to $x_3$) identity
\begin{equation} 
\label{why3}
\int_{\eta(\Gamma)} q,_{11}q,_2 n_2(t) dS(t)=\int_{\eta(\Omega)} q,_{113}q,_2
\frac{n_2(t)}{n_3(t)} dx+\int_{\eta(\Omega)} q,_{11}\p_3\bigl[q,_2
\frac{n_2(t)}{n_3(t)}\bigr] dx\,.
\end{equation}

An integration by parts with respect to the variable $x_1$ for the first
integral on the right-hand side of (\ref{why3}) then yields:

\begin{align} 
\int_{\eta(\Gamma)} q,_{11}q,_2 n_2(t) dS(t) & =-\int_{\eta(\Omega)}
q,_{13}\p_1\bigl[q,_2 \frac{n_2(t)}{n_3(t)}\bigr] dx +\int_{\eta(\Gamma)}
q,_{13}q,_2 \frac{n_2(t)}{n_3(t)}n_1(t) dS(t) \n \\
& \qquad \qquad +\int_{\eta(\Omega)}
q,_{11}\p_3\bigl[q,_2 \frac{n_2(t)}{n_3(t)}\bigr] dx \,.
\label{why4}
\end{align} 
Similarly,  by integrating by parts  first with respect to $x_3$ and then with respect to $x_2$, we see 
that the third integral on the 
left-hand side of  (\ref{why2})  can be written as
\begin{align} 
\int_{\eta(\Gamma)} q,_{12}q,_2 n_1(t) dS(t) & =-\int_{\eta(\Omega)}
q,_{13}\p_2\bigl[q,_2 \frac{n_1(t)}{n_3(t)}\bigr]dx 
+\int_{\eta(\Gamma)} q,_{13}q,_2 \frac{n_1(t)}{n_3(t)}n_2(t) dS(t) \n \\
& \qquad \qquad +\int_{\eta(\Omega)}q,_{12}\p_3\bigl[q,_2 \frac{n_1(t)}{n_3(t)}\bigr] dx \,,
\label{why5}
\end{align} 
which shows that the boundary integrals over $\eta(\Gamma)$  cancel each
other when we substitute (\ref{why5}) and (\ref{why4}) into (\ref{why2});  thus, (\ref{why2})
takes the following form:
\begin{align*}
\int_{\eta(\Omega)} q,_{\alpha\beta}q,_{\alpha\beta} dx &=\int_{\eta(\Omega)} |g|^2 dx +
\int_{\eta(\Omega)} q,_{13}\p_1\bigl[q,_2
\frac{n_2(t)}{n_3(t)}\bigr] dx -\int_{\eta(\Omega)} q,_{11}\p_3\bigl[q,_2
\frac{n_2(t)}{n_3(t)}\bigr] dx \\
&\ \ \ \  -\int_{\eta(\Omega)} q,_{13}\p_2\bigl[q,_2
\frac{n_1(t)}{n_3(t)}\bigr] dx+\int_{\eta(\Omega)} q,_{12}\p_3\bigl[q,_2
\frac{n_1(t)}{n_3(t)}\bigr] dx \,,
\end{align*}
which thanks to the estimate (\ref{why1}) and the relations
$|n_\alpha(t)|_{W^{1,\infty}(\Omega)}\le Ct$, implies that
\begin{equation}
\label{why6}
\int_{\eta(\Omega)} q,_{\alpha\beta}q,_{\alpha\beta} dx  \le C
\|f\|^2_{0,\eta(\Omega)}+C t \|D^2 q\|^2_{0,\eta(\Omega)}+C t
\|D^2 q\|_{0,\eta(\Omega)}\|D q\|_{0,\eta(\Omega)}\ ,
\end{equation}
Combining this estimate with (\ref{why1}), we obtain that
\begin{equation}
\label{why7}
\|D^2 q\|^2_{0,\eta(\Omega)}\le C \|f\|^2_{0,\eta(\Omega)}+C t
\|D^2 q\|^2_{0,\eta(\Omega)}+C t \|D^2
q\|_{0,\eta(\Omega)}\|D q\|_{0,\eta(\Omega)}\,.
\end{equation}
Now, we notice that the conditions (\ref{whyc}),(\ref{whyd}), and (\ref{qaverage}) yield
Poincar\'e  inequalities for $q,_\alpha$ and $q,_3$,  so that
\begin{subequations}
\label{why8}
\begin{align}
\|q,_3\|_{0,\eta(\Omega)}&\le C \| D q,_3\|_{0,\eta(\Omega)}\,,  \\
\|q,_\alpha\|_{0,\eta(\Omega)}&\le C \| D q,_\alpha\|_{0,\eta(\Omega)}\,, \ \alpha =1,2 \,.
\end{align}
\end{subequations}
Therefore, (\ref{why7}) and (\ref{why8}) provide us with
\begin{equation*}
\|D q\|^2_{1,\eta(\Omega)}\le C \|f\|^2_{0,\eta(\Omega)}+C t
\|D q\|^2_{1,\eta(\Omega)}\,,
\end{equation*}
which, by taking $T>0$ small enough, yields:
\begin{equation}
\label{why9}
\|D q\|^2_{1,\eta(\Omega)}\le C \| \Delta  q\|^2_{0,\eta(\Omega)}\,.
\end{equation}

We thus have proved that
\begin{equation}
q,_{li} (\eta)=F^{li}(t, \Delta  q(\eta)),
\end{equation}
where $F^{li}(t,\cdot)$ denotes a linear and continuous operator from
$L^2(\Omega)$ into itself, whose norm depends in a smooth manner on
$v$ in $L^2(0,T;H^4(\Omega))$.

Therefore,  the ODE
\begin{equation}
\label{grlap2}
\Delta  q(\eta)+\kappa[\Delta q (\eta)]_t +
A_l^j v,_j^i F^{li}(t, \Delta  q(\eta))=0,
\end{equation}
 with the initial condition (\ref{lap0}) allow us to conclude by the
Gronwall inequality that on $[0,T]\times\Omega$,
 \begin{equation}
 \label{lap}
 \Delta  q (\eta)=0\ .
 \end{equation}
 From (\ref{lap}) and (\ref{why9}), we infer that
 $$D q=0\ \ \text{in}\ [0,T]\times\Omega,$$
 which finally proves that $D Y=D \rho$, and therefore that
$\Xi=D\rho$. Together with (\ref{czero}), this establishes that $v$
is a solution of the $\kappa$-problem (\ref{approx}) on a time interval $[0,T_ \kappa ]$.
This concludes the proof of Theorem \ref{thm_ksoln4}.

By considering more time-derivatives in our analysis, it is easy to see that as long as the initial data is
smooth, we can construct solutions which are smooth in both space and time.   We state this as the
following 
\begin{theorem}[Smooth solutions to the $\kappa$-problem]\label{thm_ksoln}
Given  smooth initial data with $\rho_0$ satisfying $\rho_0(x)>0$ for $x\in \Omega $ and verifying the physical vacuum condition
(\ref{degen}) near $\Gamma$,  for $T_ \kappa>0 $ sufficiently small,  there exists a unique  smooth solutions  to the  degenerate
parabolic $ \kappa $-problem (\ref{approx}).
\end{theorem}

\section{$ \kappa $-independent  estimates for (\ref{approx}) and solutions to the compressible Euler equations (\ref{ce0})}   \label{section_mainestimates}
In this section, we obtain estimates for the smooth solutions to (\ref{approx}), provided by Theorem \ref{thm_ksoln}, whose bounds and
time interval of existence are independent of the artificial viscosity parameter $\kappa$.   This permits
us to consider the limit of this sequence of solutions as $ \kappa \rightarrow 0$.  We prove that this limit exists, and that
it is the unique solution of (\ref{ce0}).  

\vspace{.1 in}
\noindent
{\bf Notation.}  For the remainder of Section \ref{section_mainestimates}, $\eta(t)$ denotes the solution of the $ \kappa $-problem
(\ref{approx}) on the time interval $[0,T_ \kappa ]$.  In particular, $\eta(t)$ is an element of  a sequence of solutions parameterized
by $ \kappa >0$, but in order to reduce the number of subscripts and superscripts that appear, we will not make this sequential dependence
explicit.   The reader should bear in mind that $\eta$ is really $\eta( \kappa )$.
 
\subsection{A continuous-in-time energy function appropriate for the asymptotic process $ \kappa \to 0$.} 
\begin{definition} \label{defn_E} We set on $[0,T_ \kappa ]$
\begin{align} 
\tilde E (t)  = &1
+\sum_{a=0}^4\bigl[\|\partial_t^{2a} \eta(t)\|_{4-a}^2 
+ \| \rho_0\,  \partial_t^{2a} D\eta(t)\|_{4-a}^2\bigr] 
+ \sum_{a=0}^4
\|\sqrt {\rho_0 }\, \bar\partial ^{4-a} \p_t^{2a}v(t)\|_0 ^2 
\nonumber \\
&  + \sum_{a=0}^4 \int_0^t \left[
 \| \sqrt{ \kappa }\rho_0\,  \partial_t^{2a} \bar \p^{4-a} Dv(s)\|_{0}^2 \right]\, ds
+ \| \operatorname{curl} _\eta v(t)\|^2_3 
 +  \| {\rho_0}\,  \bar\p^4 \operatorname{curl} _\eta v(t)\|_0^2     
             \,. \label{Energy}
\end{align} 
The function $\tilde E(t)$ is the higher-order energy function appropriate for obtaining $ \kappa $-independent estimates for the
degenerate parabolic approximation (\ref{approx}).
\end{definition}
According to Theorem \ref{thm_ksoln},  solutions to our approximate $ \kappa $-problem (\ref{approx}) are smooth, and hence
$ T\mapsto \sup_{t \in [0,T]} \tilde E(t)$ is a continuous function on $[0,T_ \kappa ]$ to which the polynomial-type inequality of
Section \ref{subsec_poly} can be applied. 

\begin{definition} For the remainder of the paper, we will use
 the constant $\tilde M_0$ to be a polynomial function of $\tilde E (0)$ so that
\begin{equation}
\tilde M_0 = P ( \t E (0)) \label{M0} \,.
\end{equation} 
\end{definition}

\begin{remark} Note the presence of $ \kappa $-dependent coefficients in $\t E (t)$ that indeed arise
as a necessity for our asymptotic study. The corresponding terms, without the $ \kappa $,
would of course not be asymptotically controlled.
\end{remark}

\begin{remark}
The $1$ is added to the norm to ensure that $\t E (t) \ge 1$, which will sometimes
be convenient, whereas not necessary.
\end{remark}

\begin{remark}\label{rem_interp}  Of all the terms on the right-hand side of (\ref{Energy}),  the sum
$\sum_{a=0}^4\bigl[\|\partial_t^{2a} \eta(t)\|_{4-a}^2 $ is the fundamental contribution, providing the basic
regularity for the solution.  Notice that only the even time derivatives of $\eta(t)$ appear in this norm.   While it
is possible to also obtain estimates for the odd time derivative of $\eta(t)$, we will instead rely on the following interpolation
estimate:
For $k \ge 1$,  given a vector field $r \in L^ \infty  ([0,T]; H^k(\Omega))$ with $r_{tt} \in L^ \infty ([0,T]; H^{k-1}(\Omega))$, it follows that 
 $r_t \in L^2(0,T; H^{k-{\frac{1}{2}} }(\Omega))$ and
\begin{align*} 
&\|r_t\|^2_{L^2(0,T; H^{k- {\frac{1}{2}} }(\Omega))}  \le C \bigl(  \|r_{tt}(t)\|_{k-1} \|r(t)\|_k \bigr)\Bigr|^T_0 +C \|r_{tt}\|_{L^2(0,T; H^{k-1}(\Omega))}   \|r\|_{L^2(0,T; H^k(\Omega))} \\
&\ \le P( \|r(0)\|_k, \|r_{t}(0)\|_{k-1}) + \delta \sup_{t \in [0,T]} \|r(t)\|_k^2+  C \, T\, \sup_{t \in [0,T]}  \Bigl(\|r(t)\|^2_k+ \|r_{tt}(t)\|^2_{k-1}\Bigr) \,.
\end{align*} 
Thus, with $L^2$-in-time control, we see that the odd time derivatives of $\eta$ verify the estimate
\begin{align*} 
&\sum_{a=0}^3
 \| \p_t^{2a} v\|^2_{L^2(0,T; H^{3.5-a} (\Omega))}
\\
& \qquad\qquad \qquad \qquad \le M_0( \delta ) + \delta \sup_{t \in [0,T]} E(t) + C\, T\, P( \sup_{t \in [0,T]} E(t)) \,.
\end{align*} 
See the interpolation inequality (\ref{interp1}) below for further details.
\end{remark}

\subsection{Assumptions on a priori bounds on $[0,T_ \kappa ]$}\label{subsec_assumptions}
For the remainder of this section, we assume that we have solutions $(\eta_ \kappa  ) \in \boldsymbol{X} _{T_ \kappa }$
on a time interval $[0,T_ \kappa ]$, and that for all such solutions, the time $T_ \kappa >0$ is taken sufficiently small so that for
$t \in [0,T_ \kappa ]$ and $ \xi \in \mathbb{R}^3  $,

\begin{equation}\label{assump1}
\left.
\begin{array}{ll}
 {\frac{1}{2}} \le J(t) \le {\frac{3}{2}}\,,  &  \lambda |\xi|^2  \le  a^j_l a^k_l \xi_j \xi_k\,, \\
\det  g(\eta(t))^{-1/2}  \le 2\det  g(\eta_0)^{-1/2} =2\,,  &  \| J^{-1}A^r_k A^s_k - \delta ^r_k \delta ^s_k \|_{L^ \infty ( \Omega )} < {\frac{1}{2}}  \,.
\end{array} \right\}
\end{equation} 
We further assume that our solutions satisfy the bounds
\begin{equation}\label{assump2}
\left.
\begin{array}{l}
\|  \eta(t)\|^2_ {H^{3.5 }(\Omega)}   \le 2| e|_{3.5}^2 + 1 \,, \\
 \| \p_t^a v(t)\|^2_ {H^{3-a/2}(\Omega)}   \le 2 \|\p_t^a v(0)\|^2_ {H^{3-a/2}(\Omega)}  + 1\, \ \  \  \text{ for } \ a=0,1,...,6\,, \\
 \|\rho_0 \p_t^a \eta (t)\|^2_ {H^{4.5-a/2}(\Omega)}   \le 2 \| \rho_0 \p_t^a \eta(0)\|^2_ {H^{4.5-a/2}(\Omega)}  + 1\, \ \  \  \text{ for } \ a=0,1,...,7 \,, \\ 
  \|\sqrt{ \kappa } \p_t^{2a+1} v(t)\|^2_ {H^{3-a}(\Omega)}   \le 2\|  \p_t^a v(0)\|^2_ {H^{3-a}(\Omega)}  + 1\, \ \  \  \text{ for } \ a=0,1,2,3 \,.
\end{array} \right\}
\end{equation} 
  
The right-hand sides appearing in the  inequalities (\ref{assump2}) shall be denoted by a
generic constant $C$ in the estimates appearing below. In what follows, we will
prove that this can be achieved in a time interval independent of  $\kappa $.

We continue to assume that $\rho_0$ is smooth coming from our approximation (\ref{rhozero}).

\subsection{Curl Estimates}\label{subsec_curlestimates}   We take $T \in (0,T_ \kappa )$.

\begin{proposition} \label{curl_est}
For all $t \in (0,T)$,
\begin{align}
&\sum_{a=0}^3 \|{\operatorname{curl}} \, \p_t^{2a}\eta(t)\|_{3-a}^2 +
\sum_{l=0}^4 \|\rho_0 \, \bar \p^{4-l}  {\operatorname{curl}}\, \p_t^{2l}\eta(t)\|_{0}^2  
+\sum_{l=0}^4 \int_0^t  \|\sqrt{\kappa} \rho_0  {\operatorname{curl}}_ \eta \,  \bar \p^{4-l}\p_t^{2l}v(s)\|_{0}^2ds  \nonumber  \\
& \qquad\qquad \qquad\qquad
\le \t M_0 + C\, T\, P({\sup_{t\in[0,T]}} \t E(t))\,.
\label{curl_estimate}
\end{align}
\end{proposition}

\begin{proof} Using the definition of the Lagrangian curl operator $\operatorname{curl} _\eta$ given by (\ref{lagcurl}), we let
$ \operatorname{curl} _ \eta $ act on
 (\ref{approx.a}') to obtain the identity
\begin{equation}\label{rst1}
(\operatorname{curl} _ \eta v_t)^k = - \kappa \varepsilon_{kji}  v^r,_s   A^s_j \bigl[ (\rho_0 J^{-1} ),_l A^l_r\bigr],_m A^m_i \,.
\end{equation} 
As described above, in the absence of the artificial viscosity term, the right-hand side is identically zero; we will have to
make additional estimates to control {\it error terms} arising from $\kappa$-right-hand side forcing.

It follows  from (\ref{rst1}) that
$$
\p_t( \operatorname{curl} _\eta v)^k= \varepsilon_{kji}{ A_t}^s_j v^i,_s  - \kappa \varepsilon_{kji}  v^r,_s   A^s_j \bigl[ (\rho_0 J^{-1} ),_l A^l_r\bigr],_m A^m_i \,.
$$
Defining the $k$th-component of the vector field $B(A,Dv)$ by
$$
B^k(A,Dv) = - \varepsilon_{kji} A^s_r v^r,_l A^l_j v^i,_s
$$
and defining the $k$-component of the vector field $F$ by
$$
F^k = - \kappa \varepsilon_{kji}  v^r,_s   A^s_j \bigl[ (\rho_0 J^{-1} ),_l A^l_r\bigr],_m A^m_i \,,
$$
we may write
\begin{equation}\label{curlv_3d}
\operatorname{curl} _\eta v(t) = \operatorname{curl} u_0 + \int_0^t [B(A(t'),Dv(t')) + F(t')] dt'\,.
\end{equation} 
Computing the gradient of this relation yields
\begin{equation}\nonumber
\operatorname{curl} _\eta D v(t) = D\operatorname{curl} u_0 - \varepsilon_{ \cdot ji}DA^s_j v^i,_s + \int_0^t [ DB(A(t'),Dv(t')) +  
DF(t')] dt'\,.
\end{equation} 
Applying the fundamental theorem of calculus once again, shows that
\begin{align*} 
\operatorname{curl} _ \eta  D\eta(t) & = t \operatorname{curl} Du_0 +  \varepsilon_{ \cdot ji}\int_0^t [{A_t}^s_j D\eta^i,_s -DA^s_j v^i,_s] dt' \\
& \qquad\qquad \qquad + \int_0^t  \int_0^{t'} [D B(A(t''),Dv(t'')) +  DF(t'')] dt'' dt' \,,
\end{align*} 
and finally that
\begin{align} 
D\operatorname{curl} \eta(t) &= t D\operatorname{curl} u_0 - \varepsilon_{ \cdot ji}\int_0^t {A_t}^s_j(t') dt' \,  D\eta^i,_s \nonumber  \\
& +  \varepsilon_{ \cdot ji}\int_0^t [{A_t}^s_j D\eta^i,_s -DA^s_j v^i,_s] dt'
+ \int_0^t  \int_0^{t'}[ DB(A(t''),Dv(t'')) +  DF(t'')] dt'' dt'  \,. \label{ssscurl1}
\end{align}

\noindent
{\bf Step 1. Estimate for $\boldsymbol{ \operatorname{curl} \eta }$.}
To obtain an estimate for $\| \operatorname{curl} \eta(t)\|^2_3$, we let
$D^2$ act on (\ref{ssscurl1}).
With $ \p_t A^s_j = - A^s_l v^l,_p A^p_j$ and $ D A^s_j = - A^s_l
D\eta^l,_p A^p_j$, we see that the first three terms on the
right-hand side of (\ref{ssscurl1}) are bounded by $\t M_0 + C\, T\, P(
\sup_{t \in [0,T]} \t E(t))$, where we remind the reader that
$\t M_0 = P(\t E(0))$ is a polynomial function of  $\t E$ at time $t=0$.
Since
$$
DB_k(A, Dv) = - \varepsilon_{kji} [Dv^i,_s A^s_l v^l,_p A^p_j + v^i,_s
A^s_l Dv^l,_p A^p_j+ v^i,_sv^l,_pD( A^s_l  A^p_j)],
$$
the highest-order term arising from the action of $D^2$ on $DB(A,Dv)$ is
written as
$$
-\varepsilon_{kji}\int_0^t \int_0^{t'} [D^3v^i,_s A^s_l v^l,_p A^p_j +
v^i,_s A^s_l D^3v^l,_p A^p_j] dt''dt'  \,.
$$
Both summands in the integrand scale like $D^4v\, Dv\, A\, A$.  The
precise structure of this summand is not very important; rather,
the derivative count is the focus.
Integrating by parts in time,
\begin{align*}
\int_0^t\int_0^{t'} D^4v\, Dv\, A\, A \,dt'' dt' = - \int_0^t\int_0^{t'}
D^4\eta\, (Dv\, A\, A )_t dt'' dt'  + \int_0^t D^4\eta\, Dv\, A\, A\,
dt'
\end{align*}
from which it follows that
$$
\bigl\|  \int_0^t  \int_0^{t'} D^3 B(A(t''),Dv(t'')) dt'' dt'\bigr\|^2_0
\le C\, T\, P( \sup_{t \in [0,T]} \t E(t)) \, .
$$

We next estimate the term associated to $ F$. Since
\begin{align*}
DF^k = - \kappa\ \varepsilon_{kji} \bigl[ & Dv^r,_s   A^s_j \bigl[
(\rho_0 J^{-1} ),_l A^l_r\bigr],_m A^m_i +v^r,_s   A^s_j D\bigl[ (\rho_0
J^{-1} ),_l A^l_r\bigr],_m A^m_i\\
& + v^r,_s  \bigl[ (\rho_0 J^{-1} ),_l A^l_r\bigr],_m D\bigl( A^s_j
A^m_i\bigr)\bigr] \,,
\end{align*}
the highest-order term arising from the action of $D^2$ on $DF$ is
written as
$$
\kappa \varepsilon_{kji} \int_0^t \int_0^{t'} [ D^3 v^r,_s   A^s_j
\bigl[ (\rho_0 J^{-1} ),_l A^l_r\bigr],_m A^m_i +v^r,_s   A^s_j
D^3\bigl[ (\rho_0 J^{-1} ),_l A^l_r\bigr],_m A^m_i] dt''dt'  \,.
$$
The first summand in the integrand scales like $D^4v\, D^2(\rho_0 J^{-1}
)\, A\, A$, and can be estimated by integrating by parts in time in a
similar way as for the terms associated to $D^3B(A,Dv)$.

The {\it a priori} more problematic term is the second one as it seems
to call for five space derivatives on $\int_0^t\eta$ that we do not have
at our disposal. We first notice that since 
$$(\rho_0 J^{-1} ),_l
A^l_r=\rho,_r\circ\eta=[D\rho]_r\circ\eta,$$
 this integral is under the
form 
$$\kappa
\int_0^t\int_0^{t'} D^4(D\rho\circ\eta)\, Dv\, A\, A \,dt'' dt'\ .$$
Integrating by parts in time (in the integral from $0$ to $t'$),
\begin{align*}
\kappa\int_0^t\int_0^{t'} D^4 (D\rho\circ\eta)\, Dv\, A\, A \,dt'' dt' =
&-\kappa \int_0^t\int_0^{t'} (Dv\, A\, A )_t
D^4\int_0^{t''}D\rho(\eta)\,  dt'''dt'' dt' \\
& + \kappa\int_0^t Dv\, A\, A\, D^4\int_0^{t'}D\rho(\eta)\, dt''\, dt'.
\end{align*}
We now explain why we have control of four space derivatives of the
antiderivative (with respect to time) of $D\rho(\eta)$. By definition of
the $\kappa$-problem (\ref{approx.a}'), we have
 \begin{equation}\label{rst5}
 v_t+2D\rho\circ\eta+2\kappa
[D\rho\circ\eta]_t=0,
\end{equation}  
which implies by integrating (\ref{rst5})  in time twice that
\begin{equation}
\label{5rho}
2\int_0^t\int_0^{t'} D^4(D\rho\circ\eta) dt''dt' +2\kappa \int_0^t
D^4(D\rho\circ\eta)dt'=-D^4\eta(t)+tD^4u_0 \,,
\end{equation}
where we have used the fact that $D^4 \eta(0) =0$ since $\eta(0) =e$.
We can now use our Lemma \ref{kelliptic} which first yields,
independently of $\kappa$,
$$
\|\int_0^t\int_0^{t'} D^4(D\rho\circ\eta)dt''dt'\|^2_0\le \t M_0 + C \t E(t) \,,
$$
and then by using (\ref{5rho}),
\begin{equation}
\label{5rhobis}
\|\int_0^t  \kappa D^4(D\rho\circ\eta)dt'\|^2_0\le \t M_0 + C \t E(t)\ .
\end{equation}
Thanks to (\ref{5rhobis}), we get the estimate
$$
\bigl\|  \int_0^t  \int_0^{t'} D^3 F dt'' dt'\bigr\|^2_0 \le C\, T\, P(
\sup_{t \in [0,T]} \t E(t)) \,,
$$
and hence
$$
\sup_{t \in [0,T]} \| \operatorname{curl} \eta(t)\|^2_3 \le \t M_0  + C\,
T\, P( \sup_{t \in [0,T]} \t E(t)) \,.
$$
\vspace{.1 in}
\noindent
{\bf Step 2. Estimate for $\boldsymbol{ \operatorname{curl} v_t }$.}  From (\ref{rst1}),
\begin{equation}\label{rst2}
\operatorname{curl} v_t = - \varepsilon_{ \cdot ji}\int_0^t {A_t}^s_j(t') dt' \,  v_t^i,_s + F \,.
\end{equation} 
Since 
$$
2\kappa \p_t [ D\rho \circ \eta]  + 2 D\rho \circ \eta  = - v_t \,,
$$
by Lemma \ref{kelliptic}, we see that 
\begin{equation}\label{need7}
\| [D\rho \circ \eta](t)\|_3^2 \le M_0 +  \|v_t(t)\|_3^2 \,,
\end{equation} 
from which it immediately follows that
$\| F \|^2_2 \le \t M_0 + C\, T\, P({\sup_{t\in[0,T]}} \t E(t))$.
For later use, we note from equation (\ref{approx.a}) that together with (\ref{need7}), we have that
\begin{equation}\label{need8}
\| \kappa \p_t [D\rho \circ \eta](t)\|_3^2 \le M_0 +  \|v_t(t)\|_3^2 \,.
\end{equation}

Since the highest-order term in $D^2 B(A,Dv)$ is $D^3v$, we see that
$ \| B(A,Dv) \|^2_2 \le \t M_0 +
 C\, T\,  P({\sup_{t\in[0,T]}} \t E(t))$
 so that 
\begin{equation} \label{sscurl2_2d}
\| \operatorname{curl} v_t(t) \|^2_2 \le  \t M_0 + C\, T\, P({\sup_{t\in[0,T]}} \t E(t)) \,.
\end{equation} 

\vspace{.1 in}
\noindent
{\bf Step 3. Estimates for $\boldsymbol{ \operatorname{curl} v_{ttt} }$ and $ \boldsymbol{\operatorname{curl} \p_t^5 v}$.}
By time-differentiating (\ref{rst2}), estimating in the same way as {\bf Step 2}, we find that
$$
\| \operatorname{curl}  v_{ttt}(t) \|^2_1 \le \t M_0 + C\, T\, P({\sup_{t\in[0,T]}} \t E(t)) \,,
$$
and
$$
\|\operatorname{curl} \p_t^5 v(t) \|^2_1 \le \t M_0 + C\, T\, P({\sup_{t\in[0,T]}} \t E(t)) \,.
$$

\vspace{.1 in}
\noindent
{\bf Step 4. Estimate for $\boldsymbol{ \rho_0 \bar \p^4 \operatorname{curl} \eta } $.}
To prove this weighted estimate, we write (\ref{curlv_3d}) as
\begin{equation}\nonumber
\operatorname{curl} v(t) =  \varepsilon_{jki} v^i,_s \int_0^t {A_t}^s_j(t')dt' + \operatorname{curl} u_0 + \int_0^t [B(A,Dv)+F] (t')dt'\,,
\end{equation} 
and integrate in time to find that
\begin{align} 
 \operatorname{curl} \eta(t)  & = t \operatorname{curl} u_0+ \underbrace{\int_0^t \varepsilon_{jki} v^i,_s \int_0^{t'} {A_t}^s_j(t'')dt'' dt'}_{ \mathcal{I} _1}  +\underbrace{ \int_0^t\int_0^{t'}  B(A,Dv)(t'') dt''dt' }_{ \mathcal{I} _2} \nonumber \\
& \qquad\qquad \qquad +   \underbrace{\int_0^t\int_0^{t'}  F(t'') dt''dt' }_{ \mathcal{I} _3} \,. \label{rst3}
\end{align} 

It follows that
\begin{equation}\label{rst9}
\rho_0 \bar \p^4 \operatorname{curl} \eta(t) =  t\rho_0 \bar \p^4 \operatorname{curl} u_0 + \rho_0 \bar \p^4 \mathcal{I} _1
+ \rho_0 \bar \p^4 \mathcal{I} _2 + \rho_0 \bar \p^4 \mathcal{I} _3 \,.
\end{equation} 
Notice that by definition,
$$
\| t\rho_0 \bar \p^4 \operatorname{curl} u_0\|_0^2 \le  \t M_0 \,,
$$
so we must estimate the $L^2( \Omega )$-norm of  $ \rho_0 \bar \p^4 \mathcal{I} _1
+ \rho_0 \bar \p^4 \mathcal{I} _2 + \rho_0 \bar \p^4 \mathcal{I} _3$.  We first estimate $ \rho_0 \bar \p^4 \mathcal{I} _2$.
We write
\begin{align} 
& \rho_0 \bar \p^4 \mathcal{I} _2(t)= 
 \underbrace{\int_0^t\int_0^{t'} \varepsilon_{kji}{A_t}^s_j \rho_0 \bar \p^4 v^i,_s dt''dt'}_{ \mathcal{K} _1}
+ \underbrace{ \int_0^t\int_0^{t'} \varepsilon_{kji} \rho_0 \bar \p^4{A_t}^s_j  v^i,_s dt''dt'}_{ \mathcal{K} _2}
+ \mathcal{R}  \,, \nonumber
\end{align} 
where $ \mathcal{R} $ denotes {\it remainder terms} which are  lower-order  in the derivative count; in particular  the terms with the highest
derivative count in $ \mathcal{R} $ scale like $\rho \bar \p^3 Dv$ or $\rho \bar\p^4\eta$,
and hence satisfy  the inequality $ \| \mathcal{R} (t)\|^2_0 \le \t M_0 + C\, T\, P( \sup_{t \in [0,T]} \t E(t))$.  We focus on 
the integral $ \mathcal{K} _1$; integrating by parts in time, we find that
$$
\mathcal{K} _1 (t)
= 
- \int_0^t\int_0^{t'} \varepsilon_{kji} \p_t^2{A}^s_j \rho_0 \bar \p^4 \eta^i,_s dt''dt'
+ \int_0^t  \varepsilon_{kji}{A_{t}}^s_j \rho_0 \bar \p^4 \eta^i,_s   dt'  
$$
and hence
$$
\left\| \mathcal{K} _1(t) \right\|^2_0 \le \t M_0+ C\,T\, P( \sup_{t \in [0,T]} \t E(t)) \,.
$$
Using the identity $\p_t A^s_j = - A^s_ p v^p_b A^b_j$, we see that $ \mathcal{K} _2(t)$ can be estimated in the same fashion to
yield the inequality
\begin{equation}\label{rst7}
\left\|\rho_0 \bar \p^4 \mathcal{I} _2(t) \right\|^2_0 \le \t M_0+ C\,T\, P( \sup_{t \in [0,T]}  \t E(t)) \,.
\end{equation} 
Using the same integration-by-parts argument, we have similarly that
\begin{equation}\label{rst8}
\left\|\rho_0 \bar \p^4 \mathcal{I} _1(t) \right\|^2_0 \le \t M_0+ C\,T\, P( \sup_{t \in [0,T]} \t E(t)) \,.
\end{equation} 
It thus remains to estimate $ \rho_0 \bar \p^4 \mathcal{I} _3$ in (\ref{rst3}).   Now
$$
\rho_0 \bar \p^4 \mathcal{I} _3 = \int_0^t \int_0^{t'} \rho_0 \bar \p^4 F(t'') dt'' dt' \,,
$$
which can be written under the form
$$
\underbrace{\kappa  \int_0^t \int_0^{t'} \rho_0 \bar \p^4 Dv \ D [D\rho(\eta) ] \ A \ dt'' dt' }_{ \mathcal{T} _1} 
+
\underbrace{ \kappa  \int_0^t \int_0^{t'} \rho_0 \bar \p^4 D [D\rho(\eta) ] \ A \  Dv \ dt'' dt' }_{ \mathcal{T} _2}+ \mathcal{R} \,,
$$
where $ \mathcal{R} $ once again denotes a lower-order remainder term which satisfies $ \| \mathcal{R} (t)\|_0^2 \le \t M_0 +
C\, T\, P( \sup_{t \in [0,T]} \t E(t))$.

Since 
$$
2\kappa \p_t D [ D\rho \circ \eta]  + 2D[ D\rho \circ \eta]  = - Dv_t \,,
$$
by Lemma \ref{kelliptic}, we see that independently of $ \kappa $,
$$
\| D[D\rho \circ \eta](t)\|_2^2 \le \t M_0 +  \|v_t(t)\|_3^2 \le\t M_0+ C \t E(t) \  \,,
$$
and that
$$
\| \kappa \p_t  D[D\rho \circ \eta](t)\|_2^2 \le\t M_0+ C \t E(t) \  \,.
$$
Thus, the Sobolev embedding theorem shows that
$$
\| \kappa \p_t  D[D\rho \circ \eta](t)\|_{L^ \infty ( \Omega )}^2 \le\t M_0+ C \t E(t) \  \,.
$$
Hence, by using the same integration-by-parts in-time argument that we used to estimate the integral $ \mathcal{K} _1$ above,
we immediately obtain the inequality
$$
 \| \mathcal{T} _1 (t)\|_0^2 \le \t M_0 +
C\, T\, P( \sup_{t \in [0,T]} \t E(t)) \,.
$$

In order to estimate the integral $ \mathcal{T} _2$, we must rely on the structure of the Euler equations
(\ref{rst5})  once again.   Integrating in time twice, we see that
\begin{equation}
\label{6rho}
2\int_0^t\int_0^{t'} \rho_0 \bar \p^4 D(D\rho\circ\eta) +2\kappa \int_0^t
\rho_0 \bar \p^4D(D\rho\circ\eta)=-\rho_0 \bar \p^4D \eta(t) + t \rho_0 \bar \p^4 Du_0 \,.
\end{equation}
According to Lemma \ref{kelliptic},
independently of $\kappa$,
$$
\|\int_0^t\int_0^{t'} \rho_0 \bar \p^4D(D\rho\circ\eta)\|^2_{L^\infty(0,T;L^2)}\le \t M_0 + C \| \rho_0 \bar \p^4 D\eta(t)\|_0^2  \le \t M_0 + C \t E(t) \,,
$$
and then by using (\ref{6rho}),
\begin{equation}
\label{6rhobis}
\bigl\| \kappa \int_0^t  \rho_0  \bar \p^4D (D\rho\circ\eta)\bigr\|^2_{L^\infty(0,T;L^2)}\le \t M_0 + C \t E(t)\ .
\end{equation}
Returning to the estimate of $ \mathcal{T} _2$, we integrate-by-parts in time (with respect to the integral from $0$ to $t'$) to find that
\begin{align*} 
\mathcal{T} _2  & = 
 \int_0^t \int_0^{t'}  \kappa  \int_0^{t''} \rho_0 \bar \p^4 D [D\rho(\eta) ](s) ds \ [A \  Dv]_t(t'') \ dt'' dt'   \\
 & \qquad\qquad \qquad +
  \int_0^t  \kappa  \int_0^{t'} \rho_0 \bar \p^4 D [D\rho(\eta) ](t'') dt'' \ A \  Dv(t') \ dt'  \,.
\end{align*} 
The inequality (\ref{6rhobis}) then shows that 
$$
 \| \mathcal{T} _2 (t)\|_0^2 \le \t M_0 +
C\, T\, P( \sup_{t \in [0,T]} \t E(t)) \,,
$$
so that
$$
 \| \rho_0 \bar \p^4 \mathcal{I} _3  (t)\|_0^2 \le \t M_0 +
C\, T\, P( \sup_{t \in [0,T]} \t E(t)) \,,
$$
and with (\ref{rst7}) and (\ref{rst8}),  we see that (\ref{rst9}) shows that
\begin{equation}\label{sscurl13_2d}
 \| \rho_0 \bar \p^4  \operatorname{curl} \eta (t)\|_0^2 \le \t M_0 +
C\, T\, P( \sup_{t \in [0,T]} \t E(t)) \,.
\end{equation} 

\vspace{.1 in}
\noindent
{\bf Step 5. Estimate for $\boldsymbol{ \rho_0 \bar \p^3 \operatorname{curl} v_t } $.}  From (\ref{rst2}),
\begin{align*} 
\| \rho_0 \bar \p^3 \operatorname{curl} v_t (t)\|_0^2 & \le   \bigl\| \varepsilon_{ \cdot ji}  \rho_0 \bar \p^3\left(\int_0^t {A_t}^s_j(t') dt' \,  v_t^i,_s(t)\right) \bigr\|_0^2 + \| \rho_0 \bar \p^3 F (t)\|_0^2  \\
& \le C\, T\, P( \sup_{t \in [0,T]} \t E(t))  + \| \rho_0 \bar \p^3 F (t)\|_0^2 \,,
\end{align*}  
and using (\ref{assump2}), 
\begin{align*} 
\| \rho_0 \bar \p^3 F (t)\|_0^2 \le C \| \rho_0 \bar \p^3 Dv(t)\|_0^2 \| \kappa  D[D\rho(\eta)] \|_{L^ \infty (\Omega )}^2
+ C \| \kappa  \rho_0 \bar \p^3 D[D\rho(\eta)] \|_0^2 + C\, T\, P( \sup_{t \in [0,T]} \t E(t)) \,.
\end{align*} 
Since 
$$\| \kappa   D[D\rho(\eta(t))] \|_{L^ \infty (\Omega )}^2 \le C\| \kappa D\rho(\eta(t))\|_3^2 \le  C  \t M_0 + C\, T\, \sup_{t \in [0,T]} \t E(t) \,,$$
where we have used (\ref{need8}) and the fundamental theorem of calculus.  Once again
employing the fundamental theorem of calculus, 
$$
 \| \rho_0 \bar \p^3 Dv(t)\|_0^2 \le \t M_0 + C\, T\, \sup_{t \in [0,T]} \t E(t) \,,
$$
and hence
$$
 \| \rho_0 \bar \p^3 Dv(t)\|_0^2 \|  D[D\rho(\eta)] \|_{L^ \infty (\Omega )}^2 \le \t M_0 + C\, T\, P( \sup_{t \in [0,T]} \t E(t)) \,.
$$
On the other hand, since
$$
2\kappa \rho_0 \p_t  \bar \p^3 D [ D\rho \circ \eta]  + 2\rho_0 \bar \p^3 D[ D\rho \circ \eta]  = -\rho_0 \bar \p^3 Dv_t \,,
$$
by Lemma \ref{kelliptic}, we see that independently of $ \kappa $,
$$
\| \rho_0 \bar \p^3 D[ D\rho \circ \eta](t)\|_0^2 \le \t M_0 + \| \rho_0 \bar \p^3 Dv_t(t)\|_0^2 \le \t M_0 + C \t E(t) \,,
$$
and, in turn,
$$
\| \kappa \rho_0 \bar \p^3 D\p_t [ D\rho \circ \eta](t)\|_0^2 \le \t M_0 + \| \rho_0 \bar \p^3 Dv_t(t)\|_0^2 \le \t M_0 + C \t E(t) \,.
$$
By the fundamental theorem of calculus, we thus see that
$$
\| \kappa  \rho_0 \bar \p^3 D[D\rho(\eta)] \|_0^2 \le \t M_0 + C\,T\, P( \sup_{t \in [0,T]} \t E(t)) 
$$
which shows that
$$
\| \rho_0 \bar \p^3 \operatorname{curl} v_t (t)\|_0^2 \le \t M_0 + C\,T\, P( \sup_{t \in [0,T]} \t E(t)) \,.
$$

\vspace{.1 in}
\noindent
{\bf Step 6. Estimates for $\boldsymbol{ \rho_0 \bar \p^2 \operatorname{curl} v_{ttt}, \rho_0 \bar \p \operatorname{curl}  \p_t^5v,  \text{ and }\rho_0 \operatorname{curl} \p_t ^7 v} $.}  By time-differentiating (\ref{rst2}) and estimating as in {\bf Step 5}, we immediately obtain the
inequality
\begin{align*} 
\| \rho_0 \bar \p^2 \operatorname{curl} v_{ttt} (t)\|_0^2
+\| \rho_0 \bar \p \operatorname{curl} \p_t^5v (t)\|_0^2
+\| \rho_0  \operatorname{curl} \p_t^7v (t)\|_0^2
 \le \t M_0 + C\,T\, P( \sup_{t \in [0,T]} \t E(t)) \,.
\end{align*} 

\vspace{.1 in}
\noindent
{\bf Step 7. Estimate for $\boldsymbol{ \sqrt{ \kappa } \rho_0 \operatorname{curl} _\eta \bar \p^4 v} $.}  From
(\ref{curlv_3d})

\begin{align*} 
\sqrt{\kappa} \rho_0\operatorname{curl} _\eta \bar \p^4 v(t) & =  \underbrace{\sqrt{ \kappa }\rho_0 \bar \p^4\operatorname{curl} u_0}_{ \mathcal{S} _1}  +
\underbrace{\sqrt{ \kappa }\rho_0\varepsilon_{ \cdot ij} v^i,_r \bar \p^4 A^r_j (t) }_{ \mathcal{S} _2}\\
& +\underbrace{ \int_0^t  \sqrt{ \kappa }\rho_0 \bar \p^4 B(A,Dv)dt' }_{ \mathcal{S} _3} + \underbrace{\int_0^t  \sqrt{ \kappa }\rho_0 \bar \p^4 F dt' }_{ \mathcal{S} _4}+ \mathcal{R} (t) \,,
\end{align*} 
where $ \mathcal{R} (t)$ is a lower-order remainder term satisfying $\int_0^T | \mathcal{R} (t)|^2 dt \le C\,T\, P( \sup_{t \in [0,T]} \t E(t))$.  We
see that
$$
\int_0^t \| \mathcal{S} _1\|_0^2 dt' \le t \t M_0 \,,
$$
and since $\|\rho_0 \bar \p^4 D \eta(t)\|_0^2$ is  contained in the energy function $\t E(t)$,
$$
\int_0^t \| \mathcal{S} _2\|_0^2 dt' \le C\,T\, P( \sup_{t \in [0,T]} \t E(t)) \,.
$$
Jensen's inequality shows that 
$$
\int_0^t \| \mathcal{S} _3\|_0^2 dt' \le C\,T\, P( \sup_{t \in [0,T]} \t E(t)) \,.
$$
The highest-order terms in $ \mathcal{S} _4$ can be written under the form
$$
\underbrace{\int_0^t \ \kappa ^ {\frac{3}{2}}  \rho_0 \bar \p^4 Dv \ A \ D[D\rho(\eta)] \ A dt'}_{ { \mathcal{S} _4}_a} + 
\underbrace{\int_0^t  \kappa ^ {\frac{3}{2}} \rho_0 \bar \p^4 D[D\rho(\eta)] \ A \ Dv \ A dt'}_{ { \mathcal{S} _4}_b }  \,,
$$
with all other term being lower-order and easily estimated.   By Jensen's inequality and using (\ref{assump2}),
$$
\int_0^t \| {\mathcal{S} _4}_a(t')\|_0^2 dt' \le  C \kappa  \int_0^t t' \int_0^{t'} \|\sqrt{ \kappa } \rho_0 \bar \p^4 Dv(t'') \|_0^2 dt'' dt'
\le C\, T\,  \sup_{t \in [0,T]} \t E(t) \,.
$$
In order to estimate the term ${ \mathcal{S} _4}_b$, we use the identity

\begin{align*} 
&
2 \kappa^ {\frac{3}{2}} \rho_0 \bar \p^4 D[ D\rho(\eta)](t)
+ 2 \sqrt{ \kappa } \int_0^t  \rho_0 \bar \p^4 D[ D\rho(\eta)] dt'   = - \sqrt{ \kappa } \rho_0 \bar \p^4 D v(t)  \\
& \qquad \qquad\qquad \qquad + \sqrt{ \kappa } \rho_0 \bar \p^4 Du_0 + 2 \kappa^ {\frac{3}{2}} \rho_0 \bar \p^4 D^2\rho_0 \,,
\end{align*} 
which follows from differentiating the Euler equations.   Taking the $L^2( \Omega )$-inner-product of this equation with
$ \kappa^ {\frac{3}{2}} \rho_0 \bar \p^4 D[ D\rho(\eta)](t)$ and integrating in time, we deduce that
$$
\int_0^t \| \kappa^ {\frac{3}{2}} \rho_0 \bar \p^4 D[ D\rho(\eta)](t')\|_0^2 dt' \le \t M_0 + \sup_{t \in [0,T]} \t E(t) \,,
$$
from which it follows, using Jensen's inequality, that
$$
\int_0^t \| {\mathcal{S} _4}_b(t)\|_0^2 dt' \le C\,T\, P( \sup_{t \in [0,T]} \t E(t)) \,,
$$
and thus
$$
\int_0^t \| \sqrt{ \kappa } \rho_0 \operatorname{curl} _ \eta \bar \p^4 v(t') \|_0^2 dt' \le C\,T\, P( \sup_{t \in [0,T]} \t E(t)) \,.
$$

\vspace{.1 in}
\noindent
{\bf Step 8. Estimates for $\boldsymbol{ \sqrt{ \kappa } \rho_0 \operatorname{curl} _\eta \bar \p^{4-l} \p_t^{2l} v } $ for $\boldsymbol{l=1,2,3,4}$.} 
Following the identical methodology as we used for {\bf Step 7}, we obtain the desired inequality
$$
\sum_{l=1}^4 \int_0^t \|\sqrt{ \kappa } \rho_0 \operatorname{curl} _\eta \bar \p^{4-l} \p_t^{2l} v(t') \|_0^2 dt'
 \le C\,T\, P( \sup_{t \in [0,T]} \t E(t)) \,.
$$
\end{proof}

\subsection{$ \kappa $-independent energy estimates for horizontal and time derivatives}
We take $T \in (0,T_ \kappa )$.

\subsubsection{The $\bar \p^4$-problem}
\begin{proposition} \label{prop1_2d} For $ \delta >0$ and letting the constant $\t M_0$ depend on $1/ \delta $, 
\begin{align}
&\sup_{t \in [0,T]} \left(\| \sqrt{ \rho_0} \bar \p^4 v(t)\|_0^2 +\|\rho_0 \bar \p^4 D\eta(t)\|_0^2  +
\int_0^t \| \sqrt{ \kappa } \rho_0 \bar \p^4 D v(s) \|_0^2 ds
\right)   \nonumber \\
& \qquad\qquad\qquad\qquad\qquad\qquad\qquad
 \le \t M_0
+ \delta  \sup_{t \in [0,T]}\t E(t) + C\, T\, P( \sup_{t \in [0,T]} \t E(t)) \,. \label{energy1_2d}
\end{align} 
\end{proposition}

\begin{proof}
  Letting $\bar \p^4$ act on (\ref{approx.a}), and taking the 
$L^2(\Omega)$-inner product of this with $\bar \p^4 v^i$ yields
\begin{align} 
&\underbrace{ {\frac{1}{2}} {\frac{d}{dt}}  \int_\Omega \rho_0 |\bar \p^4 v|^2 dx}_ { \mathcal{I} _0} 
+\underbrace{  \int_\Omega \bar \p^4 a^k_i (\rho_0^2 J^{-2} ),_k \bar \p^4 v^i dx}_{ \mathcal{I} _1} 
+ \underbrace{ \int_\Omega a^k_i (\rho_0^2 \bar \p^4 J^{-2} ),_k \bar \p^4 v^i dx}_{ \mathcal{I} _2}  \nonumber \\
&
+ \underbrace{ \kappa  \int_\Omega \bar \p^4 \p_t a^k_i (\rho_0^2 J^{-2} ),_k \bar \p^4 v^i dx }_{ \mathcal{I} _3}
+ \underbrace{\kappa  \int_\Omega a^k_i (\rho_0^2 \bar \p^4 \p_t J^{-2} ),_k \bar \p^4 v^i dx }_{ \mathcal{I} _4} \nonumber \\
&
+\underbrace{ \kappa  \int_\Omega \bar \p^4  a^k_i (\rho_0^2 \p_t J^{-2} ),_k \bar \p^4 v^i dx }_{ \mathcal{I} _5}
+ \underbrace{ \kappa  \int_\Omega \p_t a^k_i (\rho_0^2 \bar \p^4 J^{-2} ),_k \bar \p^4 v^i dx }_{ \mathcal{I} _6} \nonumber \\
&= \underbrace{ \sum_{l=1}^3 c_l \int_\Omega \bar \p^{4-l} a^k_i \, (\rho_0^2 \bar \p^l J^{-2} ),_k \, \bar \p^4 v^i \, dx }_{ \mathcal{R} _1}
+ \underbrace{ \kappa  \sum_{l=1}^3 c_l \int_\Omega \bar \p^{4-l} \p_t a^k_i \, (\rho_0^2 \bar \p^l J^{-2} ),_k \, \bar \p^4 v^i \, dx }_{ \mathcal{R} _2} \nonumber \\
& \qquad 
+ \underbrace{ \kappa \sum_{l=1}^3 c_l \int_\Omega \bar \p^{4-l} a^k_i \, (\rho_0^2 \bar \p^l \p_t J^{-2} ),_k \, \bar \p^4 v^i \, dx}_{ \mathcal{R} _3} 
+ \underbrace{ \int_\Omega a^k_i ( \bar \p^4\rho_0^2 \ J^{-2} ),_k \bar \p^4 v^i dx}_{ \mathcal{R} _4} + \mathcal{R} _5
 \,. \label{cs1}
\end{align} 

The integrals $ \mathcal{I} _a$, $a=1,...,6$ denote the highest-order terms, while the integrals $ \mathcal{R} _a$, $a=1,...,5$ denote
lower-order remainder terms,
 which throughout the paper will consist of integrals 
which can be shown, via elementary inequalities together with our basic assumptions (\ref{assump2}),   to satisfy the following estimate:
\begin{equation}\label{remainder}
\int_0^T  \mathcal{R}_a (t)dt \le \t M_0  + \delta \sup_{t \in [0,T]} \t E(t) + C\, T\, P ( \sup_{t \in [0,T]}\t  E(t)) \,.
\end{equation} 
The remainder integral $ \mathcal{R} _5$ is comprised of the lower-order terms that are obtained when at most three horizontal
derivatives are distributed onto $\rho_0^2$, and although we do not explicitly write this term, we will explain its bound directly after
the analysis of the remainder term $ \mathcal{R} _4$ below.

We proceed to systematically estimate each of these integrals, and we begin with the lower-order remainder terms.

\vspace{.1 in}
\noindent
{\bf Analysis of  $\int_0^T \mathcal{R}_1(t) dt $.}  
We integrate by parts with respect to $x_k$ and then with respect to the time derivative $ \p_t$, and use (\ref{a3}) to obtain that
\begin{align*} 
\mathcal{R}_1 &  = - \sum_{l=1}^3 c_l 
 \int_0^T \int_\Omega  \bar \p^{4-l} a^k_i \, \rho_0^2 \bar \p^l J^{-2}   \ \bar \p^4 v^i,_k \ dxdt \\
& =  \sum_{l=1}^3 c_l  \int_0^T \int_\Omega \rho_0\left( \bar \p^{4-l} a^k_i  \bar \p^l J^{-2}\right)_t  \rho_0\bar \p^4 \eta^i,_k dxdt
-        \sum_{l=1}^3 c_l  \int_\Omega \rho_0  \bar \p^{4-l} {a}^k_i  \bar \p^l  J^{-2}  \rho_0  \bar \p^4 \eta^i,_k dx \Bigr|_0^T \,.
\end{align*} 

Notice that when  $l=3$,  the integrand in the spacetime integral on the right-hand side scales like
$ \ell \  [  \bar \p D\eta \, \rho_0  \bar \p^3 \p_t J^{-2}  + \bar \p Dv \,  \rho_0 \bar \p^3 J^{-2}  ] \ \rho_0 \bar \p^4 D \eta$ where $\ell$ denotes an $L^ \infty (\Omega)$
function. Since $\| \rho_0 \p_t^2 J^{-2} (t)\|^2_3$ is contained in the energy function $E(t)$ and since $\bar \p D \eta(t) \in L^ \infty (\Omega)$,
the first summand is estimated using an $L^ \infty$-$L^2$-$L^2$ H\"{o}lder's inequality, while for  the second summand, we use that
$\| \rho_0  J^{-2} (t)\|^2_4$ is contained in $E(t)$ together with an $L^4$-$L^4$-$L^2$ H\"{o}lder's inequality.

When $l=1$, the integrand in the spacetime integral on the right-hand side scales like
$ \ell \  [  \bar \p D\eta \, \rho_0  \bar \p^3  {a_t}^k_i  + \bar \p Dv \,  \rho_0 \bar \p^3 a^k_i ] \ \rho_0 \bar \p^4 \eta^i,_k$.
 Since $\| \rho_0 \bar \p^3 Dv_t (t)\|^2_0$ is contained in the energy function $E(t)$ and since $\bar \p D \eta \in L^ \infty (\Omega)$,
the first summand is estimated using an $L^ \infty$-$L^2$-$L^2$ H\"{o}lder's inequality.   We write the second summand as
$$ \bar \p Dv \,  \rho_0 \bar \p^3 a^\beta_i  \ \rho_0 \bar \p^4 \eta^i,_\beta+ \bar \p Dv \,  \rho_0 \bar \p^3 a^3_i  \ \rho_0 \bar \p^4 \eta^i,_3.$$  
We estimate
\begin{align}
&
\int_0^T\int_\Omega   \bar \p Dv \,  \rho_0 \bar \p^3 a^\beta_i  \ \rho_0 \bar \p^4 \eta^i,_\beta \ dxdt \nonumber \\
& \qquad\qquad  
=-\int_0^T\int_\Omega  [ \bar \p Dv \,  \rho_0 \bar \p^3 a^\beta_i,_\beta  \ \rho_0 \bar \p^4 \eta^i + \bar\p Dv,_\beta\,  \rho_0 \bar \p^3 a^\beta_i  \ \rho_0 \bar \p^4 \eta^i]
dxdt \nonumber \\
& \qquad\qquad
\le C \int_0^T \bigl( \| \bar \p  Dv(t)\|_{L^3(\Omega)} \| \rho_0 \bar \p^4 a (t) \|_0 \ \|\rho_0 \bar \p^4 \eta(t) \|_{L^6(\Omega)}\nonumber \\
& \qquad\qquad\qquad\qquad\qquad\qquad\qquad
+ \| \bar \p^2  Dv(t)\|_{L^3(\Omega)} \| \rho_0 \bar \p^4\eta(t) \|_{L^6(\Omega)} \|\bar \p^3 a \|_0 \bigr) dt \nonumber \\
& \qquad\qquad
\le C \int_0^T \bigl( \| \bar \p  Dv(t)\|_{H^{0.5}(\Omega)} \| \rho_0 \bar \p^4 a(t) \|_0 \ \|\rho_0 \bar \p^4 \eta(t) \|_1 \nonumber\\
& \qquad\qquad\qquad\qquad\qquad\qquad\qquad
+ \| \bar \p^2  Dv(t)\|_{H^{0.5}(\Omega)} \| \rho_0 \bar \p^4\eta(t) \|_1 \  \|\bar \p^3 a\|_0 \bigr) dt \nonumber \\
& \qquad\qquad
\le C \int_0^T \bigl( \|v(t)\|_{H^{3.5}(\Omega)} \| \rho_0 \bar \p^4 D\eta (t) \|_0^2 +   \|v(t)\|_{H^{2.5}(\Omega)} \| \rho_0 \bar \p^4 D\eta (t) \|_0 \|\eta(t) \|_4 \nonumber \\
& \qquad\qquad\qquad\qquad\qquad\qquad\qquad
+ \| v(t)\|_{H^{3.5}(\Omega)} \|  \eta(t) \|_4^2 \bigr) dt \,, \label{Rest_2d}
\end{align} 
where we have used H\"{o}lder's inequality, followed by the Sobolev embeddings 
$$
H^{0.5}(\Omega) \hookrightarrow L^3(\Omega) \ \ \text{ and } \ \ H^{1}(\Omega) \hookrightarrow L^6(\Omega) \,.
$$
We also rely on the interpolation estimate 
\begin{align} 
\|v\|^2_{L^2(0,T; H^{3.5}(\Omega))} & \le C\bigl(  \|v(t)\|_3 \|\eta\|_4 \bigr)\Bigr|^0_T +C \|v_t\|_{L^2(0,T; H^3(\Omega))}   \|\eta\|_{L^2(0,T; H^4(\Omega))} \nonumber \\
&\le M_0 + \delta \sup_{t \in [0,T]} \|\eta(t)\|_4^2+  C \, T\, \sup_{t \in [0,T]}  \Bigl(\|\eta(t)\|^2_4+ \|v_t(t)\|^2_3\Bigr) \,, \label{interp1}
\end{align} 
where the last inequality follows from Young's and  Jensen's inequalities.   Using this together with the Cauchy-Schwarz inequality,  (\ref{Rest_2d}) is
bounded by $C \, T\, P( \sup_{t \in [0,T]} E(t))$.   Next, since  (\ref{a3i}) shows that each component of $a^3_i$ is quadratic in $\bar \p \eta$, we see that the same analysis shows the
spacetime integral of $\bar \p Dv \,  \rho_0 \bar \p^3 a^3_i  \ \rho_0 \bar \p^4 \eta^i,_3$ has the same bound, and so we have estimated
the case $l=1$.  

 For the case that $l=2$, the integrand in the spacetime integral on the right-hand side  of the expression for $ \mathcal{R}_1 $
scales like $ \ell \   \bar \p^2 D\eta \,  \bar \p^2 D v   \ \rho_0 \bar \p^4D\eta$, so that an $L^6-L^3-L^2$ H\"{o}lder's inequality, followed
by the same analysis as for the case $l=1$ provides the same bound as for the case $l=1$.

To deal with the space integral on the right-hand side of the expression for $ \mathcal{R} _1$, the integral at time $t=0$ is equal to zero
since $\eta(x,0)=x$, whereas the integral evaluated at $t=T$ is written, using the fundamental theorem of calculus, as
$$
-        \sum_{l=1}^3 c_l  \int_\Omega \rho_0  \bar \p^{4-l} {a}^k_i  \bar \p^l  J^{-2}  \rho_0  \bar \p^4 \eta^i,_k dx\Bigr|_{t=T}=
-        \sum_{l=1}^3 c_l  \int_\Omega \rho_0  \int_0^T (\bar \p^{4-l} {a}^k_i  \bar \p^l  J^{-2})_t  \rho_0  \bar \p^4 \eta^i,_k(T) dx
$$
which can be estimated in the identical fashion as the corresponding spacetime integral.   As such, we have shown that $ \mathcal{R}_1 $
has the claimed bound (\ref{remainder}).

\vspace{.1 in}
\noindent
{\bf Analysis of  $\int_0^T \mathcal{R}_2(t) dt $.}  Using (\ref{a3}), we integrate by parts, to find that
\begin{align*} 
\int_0^T \mathcal{R} _2(t) dt &= - \sum_{l=1}^3 c_l \int_0^T  \int_\Omega \kappa  \bar \p^{4-l} a^k_i \, \rho_0^2 \bar \p^l J^{-2}  \, \bar\p^4 v^i,_k \, dx dt \\
&=  \sum_{l=1}^3 c_l  \sqrt{\kappa}\|  \bar \p^{4-l} a^k_i \, \rho_0 \bar \p^l J^{-2}\|_{L^2(0,T; L^2( \Omega ))}  \, \|\sqrt{ \kappa } \rho_0 \bar\p^4 v^i,_k \|_{L^2(0,T; L^2( \Omega ))}  \\
& \le C\,T\, P( \sup_{t \in [0,T]} \t E(t)) \,,
\end{align*} 
the last inequality following from the fact that $ \|\sqrt{ \kappa } \rho_0 \bar\p^4 Dv \|^2_{L^2(0,T; L^2( \Omega ))} $ is contained in
the energy function and that  for $l=1,2,3$, $\bar \p^{4-l} a^k_i \, \rho_0 \bar \p^l J^{-2}$ contains at most four space derivatives of $\eta(t)$, and is
controlled $L^ \infty$-in-time.

\vspace{.1 in}
\noindent
{\bf Analysis of  $\int_0^T \mathcal{R}_3(t) dt $.}  This remainder integral is estimated in the same way as $\int_0^T \mathcal{R}_2(t) dt $.

\vspace{.1 in}
\noindent
{\bf Analysis of  $\int_0^T \mathcal{R}_4(t) dt $.} Integration by parts using (\ref{a3}) shows that
\begin{align*} 
\int_0^T \mathcal{R}_4(t) dt  & = \int_0^T \int_\Omega a^k_i ( \bar \p^4\rho_0^2 \ J^{-2} ),_k \bar \p^4 v^i dxdt \\
&=  \int_0^T \int_\Omega \bar \p^4 \rho_0^2 \ ( J^{-2} a^k_i)_t \bar \p^4 \eta^i,_k  dx dt
-  \left. \int_\Omega \bar \p^4 \rho_0^2 \  J^{-2} a^k_i \bar \p^4 \eta^i,_k  dx\right|_{t=T}  \\
& =  \int_0^T \int_\Omega \bar \p^4 \rho_0^2 \ ( J^{-2} a^k_i)_t \bar \p^4 \eta^i,_k  dx dt
-  \int_\Omega \bar \p^4 \rho_0^2 \  \bar \p^4 \operatorname{div} \eta(T)  dx  \\
& \qquad\qquad \qquad  -   \int_\Omega \bar \p^4 \rho_0^2 \ \int_0^T  \p_t(J^{-2} a^k_i )dt \  \bar \p^4 \eta^i,_k (T) dx \,,
\end{align*} 
so that by the Cauchy-Schwarz inequality and Young's inequality, 
\begin{align} 
\int_0^T \mathcal{R}_4(t)dt 
 \le M_0 +  \delta \sup_{t \in [0,T]} E(t) + C\,T\, P( \sup_{t \in [0,T]} E(t)) \,. \nonumber
\end{align} 

\vspace{.1 in}
\noindent
{\bf Analysis of  $\int_0^T \mathcal{R}_5(t) dt $.} The highest-order term  is 
\begin{align*} 
 \int_0^T \int_\Omega a^k_i ( \bar \p^3\rho_0^2 \ \bar \p J^{-2} ),_k \bar \p^4 v^i dxdt 
\end{align*} 
which can be estimated directly using the Cauchy-Schwarz inequality to yield
\begin{align} 
\int_0^T \mathcal{R}_5(t)dt 
 \le M_0 +  \delta \sup_{t \in [0,T]} E(t) + C\,T\, P( \sup_{t \in [0,T]} E(t)) \,. \nonumber
\end{align} 

\vspace{.1 in}
\noindent
{\bf Analysis of the integral $\int_0^T\mathcal{I} _0(t)dt$.}  Integrating $ \mathcal{I}_0$ from $0$ to $T$, we see that
$$
\int_0^T \mathcal{I} _0(t)dt = {\frac{1}{2}} \int_ \Omega \rho_0 | \bar \p^4 v(T)|^2 dx - \t M_0 \,.
$$

\vspace{.1 in}
\noindent
{\bf Analysis of the integral $\int_0^T \mathcal{I} _1(t)dt$.}  
To estimate ${{\mathcal I} _1}$, we first integrate by parts using (\ref{a3}), to obtain
$$
\mathcal{I} _1 =
  - \int_\Omega \bar \p^4 a^k_i \rho_0^2 J^{-2} \,  \bar \p^4 v^i,_k dx \,.
$$
We then use the formula (\ref{a1}) for horizontally differentiating the cofactor matrix:
$$
{\mathcal{I} _1} =
\int_\Omega { \rho_0}^2  J  ^{-3} \,
\bar\p ^4 \eta^r,_s \, [a^s_i a^k_r - a^s_r a^k_i] 
\  \bar\p  ^4  v^i,_k \, dx   + \mathcal{R}\,,
$$
where the remainder $\mathcal{R}$ satisfies (\ref{remainder}).
We decompose the highest-order term in ${ \mathcal{I} _1}$ as the sum of the following two integrals:
\begin{align*}
{{\mathcal I}_1}_a & =  \int_\Omega  { \rho_0}^2 J  ^{-3} \
(\bar\p ^4 \eta^r,_s a^s_i)(\bar\p  ^4   v^i,_k  a^k_r) dx,  \\
{{\mathcal I}_1}_b & = -\int_\Omega { \rho_0}^2 J  ^{-3} \
(\bar\p ^4 \eta^r,_s a^s_r) (\bar\p  ^4  v^i,_k  a^k_i) dx  \,.  
\end{align*}
Since $v= \eta _t$, ${\mathcal{I} _1}_a$ is an exact derivative modulo an antisymmetric
commutation with respect to the free indices $i$ and $r$; namely,
\begin{equation}\label{newss0}
\bar\p ^4 \eta ^r,_s a^s_i \bar\p ^4 v^i,_k a^k_r
=\bar\p ^4 \eta ^i,_s a^s_r \bar\p ^4 v^i,_k a^k_r
+(\bar\p ^4 \eta ^r,_s a^s_i - \bar\p ^4 \eta ^i,_s a^s_r)\bar\p ^4 v^i,_k a^k_r \,.
\end{equation} 
and
\begin{equation}\label{newss1}
\bar\p ^4 \eta ^i,_s a^s_r \bar\p ^4 v^i,_k a^k_r = {\frac{1}{2}} \frac{d}{dt} \bigl( \bar\p ^4 \eta^i,_r a^r_k \,  \bar\p ^4 \eta^i,_s a^s_k \bigr)
- {\frac{1}{2}} \bar\p ^4 \eta^r,_s\, \bar\p ^4\eta^i,_k  \ ( a^s_r a^k_i)_t \,,
\end{equation} 
so the first term on the right-hand side of (\ref{newss0}) produces an  exact  time derivative of a positive energy contribution.

For the second term on the right-hand side of (\ref{newss0}), note the identity
\begin{equation}\label{newss2}
(\bar\p ^4 \eta ^r,_s a^s_i - \bar\p ^4 \eta ^i,_s a^s_r)\bar\p ^4 v^i,_k a^k_r
= -J^2 \varepsilon _{ijk} \bar \p ^4 \eta ^k,_r A^r_j \, \varepsilon _{imn}\bar \p ^4 v^n,_s A^s_m\,.
\end{equation} 
We have used the permutation symbol $\varepsilon$ to encode the anti-symmetry in this relation, and the basic
fact that the trace of the product of  symmetric and antisymmetric matrices is equal to zero.

Recalling our notation
$[{\operatorname{curl}} _\eta F]^i = \varepsilon_{ijk} F^k,_r A^r_j$ for a vector-field $F$,
(\ref{newss2}) can be written as
\begin{equation}\label{newss3}
(\bar\p ^4 \eta ^r,_s a^s_i - \bar\p ^4 \eta ^i,_s a^s_r)\bar\p ^4 v^i,_k a^k_r
= -J^2 \operatorname{curl} _\eta \bar\p ^4 \eta  \cdot \operatorname{curl} _\eta \bar\p ^4 v\,,
\end{equation} 
which can also be written as an exact derivative in time:
\begin{equation}\label{newss4}
\operatorname{curl} _\eta \bar\p ^4 \eta  \cdot \operatorname{curl} _\eta \bar\p ^4 v
= {\frac{1}{2}} \frac{d}{dt} | \operatorname{curl} _\eta \bar\p ^4 \eta |^2 - 
\bar\p ^4 \eta ^k,_r \bar\p ^4 \eta ^k,_s (A^r_j A^s_j)_t + 
\bar\p ^4 \eta ^k,_r \bar \p ^4 \eta ^j,_s (A^r_j A^s_k)_t \,.
\end{equation} 
The terms in (\ref{newss1}) and (\ref{newss4}) which are not the exact time derivatives are quadratic
in $\rho_0\bar \p ^4 D \eta $ with coefficients in $L^\infty ( [0,T] \times \Omega )$ and can thus be absorbed into
 remainder integrals $ \mathcal{R} $ satisfying the inequality (\ref{remainder}).
Letting
$$
D_\eta \bar\p^4 \eta = \bar \p^4 D \eta \ A   \ \ \text{ (matrix multiplication of } \bar \p^4 D\eta \text{ with $A$)} \,,
$$
we have that
\begin{equation}\nonumber
{\mathcal{I} _1}_a = {\frac{1}{2}} \frac{d}{dt}  \int_\Omega { \rho_0}^2 J^{-1} | D_\eta \bar \p ^4 \eta|^2 dx
-
{\frac{1}{2}} \frac{d}{dt}  \int_\Omega { \rho_0}^2 \,   J^{-1} | \operatorname{curl} _\eta
\bar \p ^4\eta |^2 dx  + \mathcal{R} \,.
\end{equation} 
where, once again,  the remainder
 $ \mathcal{R} $ satisfies (\ref{remainder}).

With the notation  $\operatorname{div} _\eta F = A^j_i F^i,_j$,
the differentiation formula (\ref{J1}) shows that  $ {\mathcal{I} _1}_b$ can be written as
\begin{equation}\nonumber
{\mathcal{I} _1}_b = 
-{\frac{1}{2}} \frac{d}{dt}  \int_\Omega { \rho_0}^2\,  J^{-1} | \operatorname{div} _\eta
\bar \p ^4\eta |^2 dx + \mathcal{R} \,.
\end{equation} 
It follows that
\begin{align*} 
\mathcal{I} _1& = {\frac{1}{2}} \frac{d}{dt}  \int_\Omega { \rho_0}^2 \left( J^{-1} | D_\eta \bar \p ^4 \eta |^2 -  J^{-1} | \operatorname{curl} _\eta\bar \p ^4\eta |^2 
-J^{-1} | \operatorname{div} _\eta \bar\p ^4\eta |^2 \right) dx + \mathcal{R} \\
&= {\frac{1}{2}} \frac{d}{dt}  \int_\Omega { \rho_0}^2 \left(  | D \bar \p ^4 \eta |^2 -  J^{-1} | \operatorname{curl} _\eta\bar \p ^4\eta |^2 
-J^{-1} | \operatorname{div} _\eta \bar\p ^4\eta |^2 \right) dx + \mathcal{R}\,,
\end{align*} 
where we have used the fundamental theorem of calculus for the second equality on the term $D_\eta \bar\p^4 \eta$ as well
as the fact that $T_ \kappa $ was chosen sufficiently small so that ${\frac{1}{2}} < J(t) < {\frac{3}{2}} $; in particular, we 
write 
$$D_ \eta  \bar\p^4 \eta =  \bar\p^4 D \eta \ A = \bar\p^4 D \eta \ \text{Id} + \bar\p^4 D \eta \ \int_0^t A_t (s) ds \,.
$$
It is thus clear that $\int_\Omega \rho_0^2 J^{-1} |D_ \eta\bar\p^4 \eta|^2 dx $ differs from $\int_ \Omega \rho_0^2 |D \bar \p^4 \eta|^2 dx$ by  $\mathcal{R} $.  Hence,
\begin{align*} 
\int_0^T \mathcal{I} _1(t) ds 
&\ge {\frac{1}{2}} \int_\Omega { \rho_0}^2 \left( {\frac{2}{3}} | D \bar \p ^4 \eta(T) |^2 -  J^{-1} | \operatorname{curl} _\eta\bar \p ^4\eta(T) |^2 
-J^{-1} | \operatorname{div} _\eta \bar\p ^4\eta(T) |^2 \right) dx \\
& \qquad\qquad - \t M_0  + \int_0^T \mathcal{R}(t)dt\,.
\end{align*}

\vspace{.1 in}
\noindent
{\bf Analysis of the integral $\int_0^T \mathcal{I} _2(t) dt$.}
 $\bar \p^4 J^{-2} = - 2 J^{-3} \bar \p^4 J$ plus lower-order terms, which have  at most three horizontal derivatives acting on
$J$.  For such lower-order terms, we integrate by parts with respect to $\p_t$, and estimate the resulting integrals in
the same manner as we estimated the remainder term $ \mathcal{R}_1 $, and obtain the same bound.

Thus, 
\begin{align*} 
\mathcal{I} _2 &= 2 \int_\Omega \rho_0^2 J^{-3} a^r_s \bar \p^4 \eta^s,_r\ a^k_i   \bar \p^4 v^i,_k dx  + \mathcal{R}  \\
&=  {\frac{d}{dt}} \int_\Omega \rho_0^2 J^{-3} a^r_s \bar \p^4 \eta^s,_r\ a^k_i   \bar \p^4 \eta^i,_k dx  
-  \int_\Omega \rho_0^2 (J^{-3} a^r_sa^k_i )_t  \ \bar \p^4 \eta^s,_r \bar \p^4 \eta^i,_k dx + \mathcal{R}  \\
&=  {\frac{d}{dt}} \int_\Omega \rho_0^2 J^{-1} | \operatorname{div} _\eta \bar \p^4 \eta|^2 dx  
 + \mathcal{R} 
\end{align*} 
so that
\begin{align*} 
\int_0^T\mathcal{I} _2(t)dt 
&=   \int_\Omega \rho_0^2 J^{-1} | \operatorname{div} _\eta \bar \p^4 \eta(T)|^2 dx   - \t M_0
 + \int_0^T \mathcal{R} (t)dt \,.
\end{align*} 

\vspace{.1 in}
\noindent
{\bf Analysis of the integral $\int_0^T \mathcal{I} _3(t)dt$.}    This follows closely our analysis of the integral $ \mathcal{I} _1$.
We first integrate by parts, using (\ref{a3}), to obtain
$$
\mathcal{I} _3 =
  -  \kappa \int_\Omega  \bar \p^4 \p_t a^k_i \rho_0^2 J^{-2} \,  \bar \p^4 v^i,_k dx \,.
$$
We then use the formula (\ref{a2}) for horizontally differentiating the cofactor matrix:
$$
{\mathcal{I} _3} = \kappa 
\int_\Omega { \rho_0}^2  J  ^{-3} \,
\bar\p ^4 v^r,_s \, [a^s_i a^k_r - a^s_r a^k_i] 
\  \bar\p  ^4  v^i,_k \, dx   + \mathcal{R}\,,
$$
where the remainder $\mathcal{R}$ satisfies (\ref{remainder}).
We decompose the highest-order term in ${ \mathcal{I} _3}$ as the sum of the following two integrals:
\begin{align*}
{{\mathcal I}_3}_a & =   \kappa \int_\Omega  \rho_0^2 J  ^{-3} \
(\bar\p ^4 v^r,_s a^s_i)(\bar\p  ^4   v^i,_k  a^k_r) dx,  \\
{{\mathcal I}_3}_b & = - \kappa \int_\Omega { \rho_0}^2 J  ^{-1} \
| \operatorname{div} _ \eta \bar\p ^4 v|^2 dx  \,.  
\end{align*}
Since 
\begin{align*} 
{ \mathcal{I} _3}_a 
& = \kappa \int_\Omega  \rho_0^2  J^{-3}\left[ \bar\p ^4 v ^i,_s a^s_r \bar\p ^4 v^i,_k a^k_r
+(\bar\p ^4 v ^r,_s a^s_i - \bar\p ^4 v^i,_s a^s_r)\bar\p ^4 v^i,_k a^k_r \right] dx \\
& = \kappa \int_ \Omega \rho_0 J^{-1} \bar\p ^4 v ^i,_s A^s_r \bar\p ^4 v^i,_k A^k_r dx - \kappa \int_ \Omega 
\rho_0^2 J^{-1} | \operatorname{curl} _\eta \bar\p ^4 v |^2 dx \,,
\end{align*} 
and letting
$$
D_\eta \bar\p^4 v = \bar \p^4 Dv \ A   \ \ \text{ (matrix multiplication of } \bar \p^4 Dv \text{ with $A$)} \,,
$$
we thus have that
\begin{align*} 
\int_0^T\mathcal{I} _3(t) dt & =  \kappa \int_0^T \int_ \Omega \rho_0^2 J^{-1}  | D_ \eta  \bar\p^4 v  |^2 \, dx dt - \kappa \int_0^T\int_ \Omega 
\rho_0^2 J^{-1} | \operatorname{curl} _\eta \bar\p ^4 v |^2 dx dt
 \\
& \qquad \qquad
 - \kappa\int_0^T \int_\Omega { \rho_0}^2 J  ^{-1} \
| \operatorname{div} _ \eta \bar\p ^4 v|^2 dxdt  + \int_0^T\mathcal{R}dt  \,. 
\end{align*}

\vspace{.1 in}
\noindent
{\bf Analysis of the integral $\int_0^T \mathcal{I} _4(t)dt$.}  Integrating by parts, and 
using (\ref{J2}), 
\begin{align*} 
\int_0^T\mathcal{I} _4(t)dt &= 2  \kappa\int_0^T \int_\Omega \rho_0^2 J^{-3} a^r_s \bar \p^4 v^s,_r\ a^k_i   \bar \p^4 v^i,_k dxdt  + \int_0^T\mathcal{R}dt  \\
&=  2 \kappa \int_0^T \int_\Omega { \rho_0}^2 J  ^{-1} \
| \operatorname{div} _ \eta \bar\p ^4 v|^2 dx dt +\int_0^T \mathcal{R}dt \,.
\end{align*}

\vspace{.1 in}
\noindent
{\bf Analysis of the integral $\int_0^T \mathcal{I} _5(t)dt$.} Integrating by parts, and 
using the Cauchy-Schwarz inequality, we see that
\begin{align*} 
\int_0^T \mathcal{I} _5(t)dt  & =  - \kappa \int_0^T  \int_\Omega\rho_0^2 \p_t J^{-2}\,  \bar \p^4  a^k_i \,  \bar \p^4 v^i,_k dx dt \\
& \le C \sqrt{ \kappa } \sup_{t \in [0,T]} \| \p_t J^{-2} \|_{L^ \infty ( \Omega )} \| \rho_0 \bar \p^4 a^k_i \|_{L^2(0,T; L^2( \Omega ))}
\|\sqrt\kappa\rho_0 \bar \p^4 v^i,_k \|_{L^2(0,T; L^2( \Omega ))} \\
& \le C\, T\, P( \sup_{t \in [0,T]} \t E(t)) \,,
\end{align*}  
the last inequality following from the Sobolev embedding theorem and the $L^ \infty(0,T)$ control of  $ \| \rho_0 \bar \p^4 a^k_i (t)\|_0$.

\vspace{.1 in}
\noindent
{\bf Analysis of the integral $\int_0^T \mathcal{I} _6(t) dt$.} Estimating in the same fashion as for $ \mathcal{I} _5$ shows that
$$
\int_0^T \mathcal{I} _6(t)dt \le C\, T\, P( \sup_{t \in [0,T]} \t E(t)) \,.
$$

\vspace{.1 in}
\noindent
{\bf The sum $\sum_{a=0}^6 \int_0^T \mathcal{I} _a(t)dt$.}  
By considering the sum of all the integrals $\int_0^T \mathcal{I} _a(t)dt$ for $a=0,...,6$, we obtain the inequality
\begin{align*} 
&\sup_{t \in [0,T]} {\frac{1}{2}} \left(
 \int_\Omega \rho_0 | \bar \p^4 v(t)|^2dx +
 \int_\Omega \rho_0^2J^{-1}  | \bar \p^4 D\eta(t)|^2dx 
\right)
 + \kappa \int_0^T \int_\Omega \rho_0^2J^{-1}  | D_\eta\bar \p^4 v(t)|^2dx dt
 \\
& \qquad\qquad  \le \t M_0 +  \delta  \sup_{t \in [0,T]} \t E(t) + C\,T\, P( \sup_{t \in [0,T]}\t E(t))  \\
& \qquad\qquad\qquad +   \sup_{t \in [0,T]} \int_\Omega \rho _0^2J^{-1}  | \bar \p^4 \operatorname{curl} \eta(t)|^2  dx
+ \kappa   \int_0^T \int_\Omega \rho _0^2J^{-1}  | \operatorname{curl} _\eta \bar \p^4v(t)|^2  dxdt \,.
\end{align*} 
Using the fundamental theorem of calculus, 
\begin{align*} 
&
\kappa \int_0^T \int_\Omega \rho_0^2 J^{-1}  | D_\eta\bar \p^4 v(t)|^2dx dt 
=\kappa \int_0^T \int_\Omega \rho_0^2J^{-1}    \bar \p^4 v^i,_r A^r_k  \bar \p^4 v^i,_s A^s_k dx dt \\
& \qquad = \kappa \int_0^T \int_\Omega \rho_0^2    \bar \p^4 v^i,_k  \bar \p^4 v^i,_k \, dx dt  + 
\kappa \int_0^T \int_\Omega \rho_0^2 [J^{-1}A^r_k A^s_k - \delta ^r_k \delta ^s_k]    \bar \p^4 v^i,_r  \bar \p^4 v^i,_s dx dt \,,
\end{align*} 
where we are using the Kronecker Delta symbol $\delta ^r_k$ to denote the components of the identity matrix.   It thus follows
from (\ref{assump1}) that

\begin{align*} 
&\sup_{t \in [0,T]} {\frac{1}{2}}  \left(
 \int_\Omega \rho_0 | \bar \p^4 v(t)|^2dx +
 \int_\Omega \rho_0^2J^{-1}  | \bar \p^4 D\eta(t)|^2dx 
\right) 
+ \frac{\kappa}{2} \int_0^T \int_\Omega \rho_0^2 | D\bar \p^4 v(t)|^2dx dt
 \\
& \qquad\qquad  \le \t M_0 +  \delta  \sup_{t \in [0,T]} \t E(t) + C\,T\, P( \sup_{t \in [0,T]}\t E(t))  \\
& \qquad\qquad\qquad +   \sup_{t \in [0,T]} \int_\Omega \rho _0^2J^{-1}  | \bar \p^4 \operatorname{curl} \eta(t)|^2  dx
+ \kappa   \int_0^T \int_\Omega \rho _0^2J^{-1}  | \operatorname{curl} _\eta \bar \p^4v(t)|^2  dxdt \,.
\end{align*} 

The curl estimates (\ref{curl_estimate}) provide the bound for the last two integrals from which the desired result is obtained  and
 the proof of the proposition is completed.
\end{proof}

\begin{corollary}[Estimates for the trace of the tangential components of $\eta(t)$] \label{cor1_2d} For $ \alpha =1,2$, and
$\delta >0$,
$$ \sup_{t \in [0,T]}  |\eta^\alpha(t) |^2_{3.5}   \le\t  M_0 
+\delta  \sup_{t \in [0,T]} \t E+ C\, T\, P
( \sup_{t \in [0,T]} \t E(t)) \,.$$
\end{corollary}
\begin{proof} 
The weighted embedding estimate (\ref{w-embed}) shows that 
$$
\| \bar \p^4 \eta(t)\|^2_0 \le C  \int_\Omega \rho_0^2 \bigl( |\bar \p^4 \eta|^2 + |\bar\p^4D\eta|^2 \bigr) dx \,.
$$

Now 
$$\sup_{t \in [0,T]} \int_\Omega \rho_0^2 |\bar \p^4 \eta|^2 dx = \sup_{t \in [0,T]} \int_\Omega \rho_0^2 \left| \int_0^t \bar \p^4 v dt'\right|^2 dx \le T^2 \sup_{t \in [0,T]} \| \sqrt{\rho_0} \bar \p^4 v \|^2_0 \,.$$
It follows from Proposition \ref{prop1_2d} that
$$
\sup_{t \in [0,T]} \| \bar \p^4 \eta(t)\|^2_0 \le \t M_0 + C\,T\, P( \sup_{t \in [0,T]} \t E(t)) \,.
$$
According to our curl estimates (\ref{curl_estimate}), $ \sup_{t \in [0,T]} \| \operatorname{curl} \eta\|^2_3 \le\t  M_0 + C\,T\, P( \sup_{t \in [0,T]} 
\t E(t))$, from which it follows that
$$
\sup_{t \in [0,T]} \| \bar \p^4 \operatorname{curl} \eta(t)\|^2_{H^1(\Omega)'} \le \t M_0 + C\,T\, P( \sup_{t \in [0,T]} \t E(t)) \,,
$$
since $\bar \p$ is a horizontal derivative, and integration by parts with respect to $\bar \p$ does not produce any boundary contributions.
From the tangential trace inequality (\ref{tangentialtrace}), we find that
$$
\sup_{t \in [0,T]} | \bar \p^4 \eta^\alpha(t) |^2_{-0.5} \le \t M_0 + C\,T\,P( \sup_{t \in [0,T]} \t E(t)) \,,
$$
from which it follows that
$$
\sup_{t \in [0,T]} |  \eta^\alpha(t) |^2_{3.5} \le \t M_0 + C\,T\,P( \sup_{t \in [0,T]} \t E(t)) \,,
$$
\end{proof}

\subsection{The $\p_t^8$-problem}
\begin{proposition} \label{prop5_2d} For $ \delta >0$ and letting the constant $\t M_0$ depend on $1/ \delta $, 
\begin{align}
&\sup_{t \in [0,T]} \left(\| \sqrt{\rho_0} \p_t^8v(t)\|_0^2+ \| \rho_0 \p_t^7 Dv(t)\|_0^2 
+ \int_0^t \| \sqrt{ \kappa } \rho_0  D \p_t^8v(s) \|_0^2 ds\right) \nonumber \\
& \qquad\qquad\qquad\qquad\qquad\qquad\qquad
 \le \t M_0
+\delta  \sup_{t \in [0,T]} \t E(t) + C\, T\, P( \sup_{t \in [0,T]} \t E(t)) \,. \label{energy5_2d}
\end{align} 
\end{proposition}

\begin{proof}
   Letting $\p_t^8 $ act on (\ref{approx.a}), and taking the 
$L^2(\Omega)$-inner product of this with $\p_t^8 v^i$ yields
\begin{align} 
&\underbrace{ {\frac{1}{2}} {\frac{d}{dt}}  \int_\Omega \rho_0 |\p_t^8 v|^2 dx}_ { \mathcal{I} _0} 
+\underbrace{  \int_\Omega \p_t^8a^k_i (\rho_0^2 J^{-2} ),_k \p_t^8 v^i dx}_{ \mathcal{I} _1} 
+ \underbrace{ \int_\Omega a^k_i (\rho_0^2 \p_t^8 J^{-2} ),_k \p_t^8 v^i dx}_{ \mathcal{I} _2}  \nonumber \\
&
+ \underbrace{ \kappa  \int_\Omega \p_t^8 \p_t a^k_i (\rho_0^2 J^{-2} ),_k \p_t^8 v^i dx }_{ \mathcal{I} _3}
+ \underbrace{\kappa  \int_\Omega a^k_i (\rho_0^2 \p_t^8 \p_t J^{-2} ),_k \p_t^8 v^i dx }_{ \mathcal{I} _4} \nonumber \\
&
+\underbrace{ \kappa  \int_\Omega \p_t^8  a^k_i (\rho_0^2 \p_t J^{-2} ),_k \p_t^8 v^i dx }_{ \mathcal{I} _5}
+ \underbrace{ \kappa  \int_\Omega \p_t a^k_i (\rho_0^2 \p_t^8 J^{-2} ),_k \p_t^8 v^i dx }_{ \mathcal{I} _6} \nonumber \\
&= \underbrace{ \sum_{l=1}^7 c_l \int_\Omega  \p_t^{8-l} a^k_i \, (\rho_0^2 \p_t^l J^{-2} ),_k \, \p_t^8 v^i \, dx }_{ \mathcal{R} _1}
+ \underbrace{ \kappa  \sum_{l=1}^7 c_l \int_\Omega  \p_t^{8-l} \p_t a^k_i \, (\rho_0^2  \p_t^l J^{-2} ),_k \, \p_t^8 v^i \, dx }_{ \mathcal{R} _2} \nonumber \\
& \qquad 
+ \underbrace{ \kappa \sum_{l=1}^7 c_l \int_\Omega \p_t^{8-l} a^k_i \, (\rho_0^2  \p_t^l \p_t J^{-2} ),_k \, \p_t^8 v^i \, dx}_{ \mathcal{R} _3}
 \,. \label{cs101}
\end{align} 
The integrals $ \mathcal{I} _a$, $a=1,...,6$ denote the highest-order terms, while the integrals $ \mathcal{R} _a$, $a=1,2,3$ denote
lower-order remainder terms, which we will once again prove satisfy 
 \begin{equation}\label{remainder2}
\int_0^T  \mathcal{R}_a (t)dt \le \t M_0  + \delta \sup_{t \in [0,T]} \t E(t) + C\, T\, P ( \sup_{t \in [0,T]}\t  E(t)) \,.
\end{equation} 

The analysis of the lower-order remainder term $\mathcal{R}_1(t)$ differs slightly from the corresponding remainder term in
the $\bar \p^4$ energy estimates, so we proceed with the details of this analysis.

\vspace{.1 in}
\noindent
{\bf Analysis of  $\int_0^T \mathcal{R}_1(t) dt $.}  Using (\ref{a3}),
we integrate by parts with respect to $x_k$ and then with respect to the time derivative $ \p_t$ to obtain that
\begin{align*} 
\mathcal{R} _1&  = - \sum_{l=1}^7 c_l 
 \int_0^T \int_\Omega  \p_t^{8-l} a^k_i \, \rho_0^2 \p_t^l J^{-2}   \ \p_t^8 v^i,_k \ dxdt \\
& =  \sum_{l=1}^7 c_l  \int_0^T \int_\Omega \rho_0\left( \p_t^{8-l} a^k_i  \p_t^l J^{-2}\right)_t  \rho_0\p_t^7 v^i,_k dxdt
-        \sum_{l=1}^7 c_l  \int_\Omega \rho_0  \p_t^{8-l} {a}^k_i  \bar \p^l  J^{-2}  \rho_0  \p_t^7 v^i,_k dx \Bigr|_0^T \,.
\end{align*} 

Notice that when  $l=7$,  the integrand in the spacetime integral on the right-hand side scales like
$ \ell \   [ Dv_t \, \rho_0 \p_t^6Dv + Dv \, \rho_0 \p_t^7Dv ]  \, \rho_0 \p_t^7D v$ where $\ell$ denotes an $L^ \infty (\Omega)$
function. 
Since $\| \rho_0 \p_t^7 Dv (t)\|^2_0$ is contained in the energy function $\t E(t)$, $ D v_t(t)$ is bounded in $L^ \infty (\Omega)$,
and since we can write
$ \rho_0 \p_t^6 Dv(t) = \rho_0 \p_t^6 Dv(0) + \int_0^t \rho_0 \p_t^7 Dv(t')dt'$,
the first  and second summands are both estimated using an $L^ \infty$-$L^2$-$L^2$ H\"{o}lder's inequality.

The case $l=6$ is estimated exactly the same way as the case $l=3$ in the proof of Proposition \ref{prop1_2d}.   For the case
$l=5$, the integrand in the spacetime integral scales like
$\ell [ Dv_{tt} \rho_0 \p_t^6 J^{-2}  + Dv_{ttt}  \rho_0 Dv_{tttt}] \rho_0 \p_t^7 Dv$.  Both summands can be estimated using 
an $L^ 3$-$L^6$-$L^2$ H\"{o}lder's inequality.  The case $l=4$ is treated as the case $l=5$.  The case $l=3$ is also treated in
the same way as $l=5$.  The case $l=2$ is estimated exactly the same way as the case $l=1$ in the proof of Proposition \ref{prop1_2d}.
The case $l=1$ is treated in the same way as the case $l=7$.

To deal with the space integral on the right-hand side of the expression for $ \mathcal{R}_1 $, the integral at time $t=0$ is bounded by $\t M_0$, whereas the integral evaluated at $t=T$ is written, using the fundamental theorem of calculus, as
\begin{align*} 
\sum_{l=1}^7 c_l  \int_\Omega \rho_0  \p_t^{8-l} {a}^k_i  \p_t^l  J^{-2}  \rho_0 \p_t^7 v^i,_k dx\Bigr|_{t=T} & =
\sum_{l=1}^7 c_l \int_\Omega \rho_0  \p_t^{8-l} {a}^k_i(0)  \p_t^l  J^{-2} (0) \rho_0 \p_t^7 v^i,_k(T) dx\\
& +  \sum_{l=1}^7 c_l  \int_\Omega \rho_0  \int_0^T (\p_t^{8-l} {a}^k_i  \p_t^l  J^{-2})_t dt'\,  \rho_0  \p_t^7v^i,_k(T) dx.
\end{align*} 
The first integral on the right-hand side is estimated using Young's inequality, and is bounded by $\t M_0 + \delta \sup_{t \in [0,T]} \t E(t)$,
while the second integral 
can be estimated in the identical fashion as the corresponding spacetime integral.   As such, we have shown that $ \mathcal{R}_1 $
has the claimed bound (\ref{remainder2}).

\vspace{.1 in}
\noindent
{\bf Analysis of  $\int_0^T \mathcal{R}_2(t) dt $.}  Using (\ref{a3}), we integrate by parts, to find that
\begin{align*} 
\int_0^T \mathcal{R} _2(t) dt &= - \sum_{l=1}^7 c_l \int_0^T  \int_\Omega \kappa   \p_t^{8-l} a^k_i \, \rho_0^2  \p_t^l J^{-2}  \, \p_t^8 v^i,_k \, dx dt \\
&=  \sum_{l=1}^7 c_l  \sqrt{\kappa}\|  \p_t^{8-l} a^k_i \, \rho_0  \p_t^l J^{-2}\|_{L^2(0,T; L^2( \Omega ))}  \, \| \sqrt{ \kappa } \rho_0 \p_t^8 v^i,_k \|_{L^2(0,T; L^2( \Omega ))}  \\
& \le C\,T\, P( \sup_{t \in [0,T]} \t E(t)) \,,
\end{align*} 
the last inequality following from the fact that $\| \sqrt{ \kappa } \rho_0 \p_t^8 Dv \|^2_{L^2(0,T; L^2( \Omega ))}$ and
$\| \rho_0 \p_t^7 Dv \|^2_{L^\infty (0,T; L^2( \Omega ))}$ are both contained in the energy function.

\vspace{.1 in}
\noindent
{\bf Analysis of  $\int_0^T \mathcal{R}_3(t) dt $.}  This remainder integral is estimated in the same way as $\int_0^T \mathcal{R}_2(t) dt $.

\vspace{.1 in}
\noindent
{\bf Analysis of the integral $\int_0^T\mathcal{I} _0(t)dt$.}  Integrating $ \mathcal{I}_0$ from $0$ to $T$, we see that
$$
\int_0^T \mathcal{I} _0(t)dt = {\frac{1}{2}} \int_ \Omega \rho_0 | \p_t^8 v(T)|^2 dx - \t M_0 \,.
$$

\vspace{.1 in}
\noindent
{\bf Analysis of the integral $\int_0^T\mathcal{I} _1(t)dt$.} To estimate ${{\mathcal I} _1}$, we again integrate by parts using (\ref{a3}), to obtain
$$
\mathcal{I} _1 =
  - \int_\Omega \p_t^8 a^k_i \rho_0^2 J^{-2} \,  \p_t^8 v^i,_k dx \,.
$$
Using  the differentiation identity (\ref{a2}),
the same anti-symmetric commutation that we used for the $\bar \p^4$-differentiated problem can be employed once again to yield
\begin{align*}
\rho_0^{2 }( \p_t^7 v^r,_s A^s_i)
\  (\p_t^8  v^i,_k A^k_r)
& = 
\frac{1}{2}\frac{d}{dt} |{\rho_0 } D_\eta \p_t^7 v(t)|^2
-
\frac{1}{2}\frac{d}{dt} |{\rho_0 }  {\operatorname{curl}} _\eta \p_t^7 v(t)|^2 \\
&  \qquad \qquad + {\frac{1}{2}}  
{ \rho_0}^2\p_t^7 v^k,_r \p_t^7 v^b,_s
(A^r_j A^s_m)_t [ \delta^j_m\delta^k_b - \delta^j_b \delta ^k_m] \,,
\end{align*}
and
\begin{align*}
-\rho_0^2(\p_t^7  v^r,_s\, A^s_r) \ (\p_t^8  v^i,_k  A^k_i) 
& = -{\frac{1}{2}} {\frac{d}{dt}} 
|{\rho_0} {\operatorname{div}} _\eta \p_t ^7 v|^2  
 + {\frac{1}{2}} 
 \rho_0^{2 } \p_t^7  v^r,_s \p_t^7  v^i,_k
(A^s_r A^k_i)_t  \,.
\end{align*}
Hence,
\begin{align*} 
\mathcal{I} _1 = {\frac{1}{2}} \frac{d}{dt}  \int_\Omega { \rho_0}^2 \left(  |D \p_t^7
 v  |^2 -  J^{-1} | \operatorname{curl} _\eta \p_t^7 v|^2 
-J^{-1} | \operatorname{div} _\eta \p_t^7 v|^2 \right) dx + \mathcal{R} \,,
\end{align*} 
and
\begin{align} 
\int_0^T \mathcal{I} _1(t)dt & = {\frac{1}{2}}  \int_\Omega { \rho_0}^2 \left(  |D \p_t^7
 v(T)  |^2 -  J^{-1} | \operatorname{curl} _\eta \p_t^7 v(T)|^2 
-J^{-1} | \operatorname{div} _\eta \p_t^7 v(T)|^2 \right) dx  \nonumber \\
& \qquad\qquad - \t M_0 + \int_0^T \mathcal{R}(t) dt \,,   \label{ssadd101}
\end{align} 
where the remainder integral $ \mathcal{R} $ satisfies (\ref{remainder2}).

\vspace{.1 in}
\noindent
{\bf Analysis of the integral $\int_0^T \mathcal{I} _2(t) dt$.}
 $\p_t^8 J^{-2} = - 2 J^{-3} \p_t^8 J$ plus lower-order terms, which have  at most  seven time derivatives acting on
$J$.  For such lower-order terms, we integrate by parts with respect to $\p_t$, and estimate the resulting integrals in
the same manner as we estimated the remainder term $ \mathcal{R}_1 $, and obtain the same bound.

Thus, 
\begin{align*} 
\mathcal{I} _2 &= 2 \int_\Omega \rho_0^2 J^{-3} a^r_s \p_t^8  \eta^s,_r\ a^k_i  \p_t^8  v^i,_k dx  + \mathcal{R}  \\
&=  {\frac{d}{dt}} \int_\Omega \rho_0^2 J^{-3} a^r_s \p_t^7 v^s,_r\ a^k_i  \p_t^7 v^i,_k dx  
-  \int_\Omega \rho_0^2 (J^{-3} a^r_sa^k_i )_t  \ \p_t^7v^s,_r \p_t^7v^i,_k dx + \mathcal{R}  \\
&=  {\frac{d}{dt}} \int_\Omega \rho_0^2 J^{-1} | \operatorname{div} _\eta \p_t^7 v|^2 dx  
 + \mathcal{R} 
\end{align*} 
so that
\begin{align*} 
\int_0^T\mathcal{I} _2(t)dt 
&=   \int_\Omega \rho_0^2 J^{-1} | \operatorname{div} _\eta \p_t^7v(T)|^2 dx   - \t M_0
 + \int_0^T \mathcal{R} (t)dt \,.
\end{align*} 

\vspace{.1 in}
\noindent
{\bf Analysis of the integral $\int_0^T \mathcal{I} _3(t)dt$.}    This follows closely our analysis of the integral $ \mathcal{I} _1$.
We first integrate by parts, using (\ref{a3}), to obtain
$$
\mathcal{I} _3 =
  -  \kappa \int_\Omega \p_t^9  a^k_i \rho_0^2 J^{-2} \,  \p_t^8  v^i,_k dx \,.
$$
We then use the formula (\ref{a2}) for horizontally differentiating the cofactor matrix:
$$
{\mathcal{I} _3} = \kappa 
\int_\Omega { \rho_0}^2  J  ^{-3} \,
\p_t^8  v^r,_s \, [a^s_i a^k_r - a^s_r a^k_i] 
\ \p_t^8   v^i,_k \, dx   + \mathcal{R}\,,
$$
where the remainder $\mathcal{R}$ satisfies (\ref{remainder}).
We decompose the highest-order term in ${ \mathcal{I} _3}$ as the sum of the following two integrals:
\begin{align*}
{{\mathcal I}_3}_a & =   \kappa \int_\Omega  \rho_0^2 J  ^{-3} \
(\p_t^8  v^r,_s a^s_i)(\p_t^8   v^i,_k  a^k_r) dx,  \\
{{\mathcal I}_3}_b & = - \kappa \int_\Omega { \rho_0}^2 J  ^{-1} \
| \operatorname{div} _ \eta \p_t^8  v|^2 dx  \,.  
\end{align*}
Since 
\begin{align*} 
{ \mathcal{I} _3}_a 
& = \kappa \int_\Omega  \rho_0^2  J^{-3}\left[\p_t^8  v ^i,_s a^s_r \p_t^8  v^i,_k a^k_r
+(\p_t^8  v ^r,_s a^s_i - \p_t^8  v^i,_s a^s_r)\p_t^8  v^i,_k a^k_r \right] dx \\
& = \kappa \int_ \Omega \rho_0 J^{-1} \p_t^8  v ^i,_s A^s_r \p_t^8  v^i,_k A^k_r dx - \kappa \int_ \Omega 
\rho_0^2 J^{-1} | \operatorname{curl} _\eta \p_t^8  v |^2 dx \,,
\end{align*} 
and letting
$$
D_\eta \p_t^8  v = \p_t^8  Dv \ A   \ \ \text{ (matrix multiplication of }\p_t^8  Dv \text{ with $A$)} \,,
$$
we thus have that
\begin{align*} 
\int_0^T\mathcal{I} _3(t) dt & =  \kappa \int_0^T \int_ \Omega \rho_0 J^{-1}  | D_ \eta \p_t^8  v  |^2 \, dx dt - \kappa \int_0^T\int_ \Omega 
\rho_0^2 J^{-1} | \operatorname{curl} _\eta \p_t^8  v |^2 dx dt
 \\
& \qquad \qquad
 - \kappa\int_0^T \int_\Omega { \rho_0}^2 J  ^{-1} \
| \operatorname{div} _ \eta \p_t^8  v|^2 dxdt  + \int_0^T\mathcal{R}(t)dt  \,. 
\end{align*}

\vspace{.1 in}
\noindent
{\bf Analysis of the integral $\int_0^T \mathcal{I} _4(t)dt$.}  Integrating by parts, and 
using (\ref{J2}), 
\begin{align*} 
\int_0^T\mathcal{I} _4(t)dt &= 2  \kappa\int_0^T \int_\Omega \rho_0^2 J^{-3} a^r_s \p_t^8  v^s,_r\ a^k_i   \p_t^8  v^i,_k dxdt  + \int_0^T\mathcal{R}(t)dt  \\
&=  2 \kappa \int_0^T \int_\Omega { \rho_0}^2 J  ^{-1} \
| \operatorname{div} _ \eta \p_t^8  v|^2 dx dt +\int_0^T \mathcal{R}(t)dt \,.
\end{align*}

\vspace{.1 in}
\noindent
{\bf Analysis of the integral $\int_0^T \mathcal{I} _5(t)dt$.} Integrating by parts, and 
using the Cauchy-Schwarz inequality, we see that
\begin{align*} 
\int_0^T \mathcal{I} _5(t)dt  & =  - \kappa \int_0^T  \int_\Omega\rho_0^2 \p_t J^{-2}\,  \p_t^8   a^k_i \,  \p_t^8  v^i,_k dx dt \\
& \le C \sqrt{ \kappa } \sup_{t \in [0,T]} \| \p_t J^{-2} \|_{L^ \infty ( \Omega )} \| \rho_0 \p_t^8  a^k_i \|_{L^2(0,T; L^2( \Omega ))}
\|\sqrt\kappa\rho_0 \p_t^8  v^i,_k \|_{L^2(0,T; L^2( \Omega ))} \\
& \le C\, T\, P( \sup_{t \in [0,T]} \t E(t)) \,,
\end{align*}  
the last inequality following from the Sobolev embedding theorem and the $L^ \infty(0,T)$ control of  $ \| \rho_0\p_t^8 a^k_i (t)\|_0$.

\vspace{.1 in}
\noindent
{\bf Analysis of the integral $\int_0^T \mathcal{I} _6(t) dt$.} Estimating in the same fashion as for $ \mathcal{I} _5$ shows that
$$
\int_0^T \mathcal{I} _6(t)dt \le C\, T\, P( \sup_{t \in [0,T]} \t E(t)) \,.
$$

\vspace{.1 in}
\noindent
{\bf The sum $\sum_{a=0}^6 \int_0^T \mathcal{I} _a(t)dt$.}  
By considering the sum of all the integrals $\int_0^T \mathcal{I} _a(t)dt$ for $a=0,...,6$, we obtain the inequality
\begin{align*} 
&\sup_{t \in [0,T]} {\frac{1}{2}}  \left(
\int_\Omega \rho_0 | \p_t^8 v(t)|^2dx +
 \int_\Omega \rho_0^2J^{-1}  | \p_t^7 Dv(t)|^2dx\right)  
 + \kappa \int_0^T \int_\Omega \rho_0^2J^{-1}  | D_\eta \p_t^8 v(t)|^2dx dt
 \\
& \qquad\qquad  \le \t M_0 +  \delta  \sup_{t \in [0,T]} \t E(t) + C\,T\, P( \sup_{t \in [0,T]}\t E(t))  \\
& \qquad\qquad\qquad +   \sup_{t \in [0,T]} \int_\Omega \rho _0^2J^{-1}  | \p_t^7 \operatorname{curl} v(t)|^2  dx
+ \kappa   \int_0^T \int_\Omega \rho _0^2J^{-1}  | \operatorname{curl} _\eta \p_t^8 v(t)|^2  dxdt  \,.
\end{align*} 
Using the fundamental theorem of calculus, 
\begin{align*} 
&
\kappa \int_0^T \int_\Omega \rho_0^2 J^{-1}  | D_\eta\p_t^8 v(t)|^2dx dt 
=\kappa \int_0^T \int_\Omega \rho_0^2J^{-1}    \p_t^8 v^i,_r A^r_k  \p_t^8 v^i,_s A^s_k dx dt \\
& \qquad = \kappa \int_0^T \int_\Omega \rho_0^2    \p_t^8 v^i,_k  \p_t^8 v^i,_k \, dx dt  + 
\kappa \int_0^T \int_\Omega \rho_0^2 [J^{-1}A^r_k A^s_k - \delta ^r_k \delta ^s_k]    \p_t^8 v^i,_r  \p_t^8  v^i,_s dx dt \,,
\end{align*} 
where we are using the Kronecker Delta symbol $\delta ^r_k$ to denote the components of the identity matrix.   It thus follows
from (\ref{assump1}) that

\begin{align*} 
&\sup_{t \in [0,T]} {\frac{1}{2}}  \left(
\int_\Omega \rho_0 | \p_t^8 v(t)|^2dx +
 \int_\Omega \rho_0^2J^{-1}  | \p_t^7 Dv(t)|^2dx\right)  + {\frac{\kappa }{2}} \int_0^T \int_\Omega \rho_0^2 | D\p_t^8 v(t)|^2dx dt
 \\
& \qquad\qquad  \le \t M_0 +  \delta  \sup_{t \in [0,T]} \t E(t) + C\,T\, P( \sup_{t \in [0,T]}\t E(t))  \\
& \qquad\qquad\qquad +   \sup_{t \in [0,T]} \int_\Omega \rho _0^2J^{-1}  |\p_t^7 \operatorname{curl} v(t)|^2  dx
+  \kappa  \int_0^T \int_\Omega \rho _0^2J^{-1}  | \operatorname{curl} _\eta \p_t^8 v(t)|^2  dxdt \,.
\end{align*} 

The curl estimates (\ref{curl_estimate}) provide the bound for the last two integrals from which the desired result is obtained  and
 the proof of the proposition is completed.
\end{proof}

\begin{corollary}[Estimates for $\p_t^7 v(t)$] \label{cor2} 
$$ \sup_{t \in [0,T]}   \|\p_t^7 v(t) \|^2_0 
   \le\t  M_0 +  \delta  \sup_{t \in [0,T]} \t E +C\, T\, P
( \sup_{t \in [0,T]} \t E(t)) \,.$$
\end{corollary}
\begin{proof} 
The weighted embedding estimate (\ref{w-embed}) shows that 
$$
\| \p_t^7 v(t)\|^2_0 \le C  \int_\Omega \rho_0^2 \bigl( | \p^7 v|^2 + |D \p_t^7 v|^2 \bigr) dx \,.
$$

Now 
$$\int_\Omega \rho_0^2 | \p_t^7 v|^2 dx \le M_0 + \int_\Omega \rho_0^2 \left| \int_0^t \p_t^8 v dt'\right|^2 dx \le M_0 +  T^2 \sup_{t \in [0,T]} \| \sqrt{\rho_0} \p_t^8 v \|^2_0 \,.$$
Thus,   Proposition \ref{prop5_2d} shows that 
$$ \sup_{t \in [0,T]}  \|\p_t^7 v(t) \|^2_0
  \le\t  M_0 +  \delta  \sup_{t \in [0,T]} \t E +C\, T\, P
( \sup_{t \in [0,T]} \t E(t)) \,.$$
\end{proof}

\subsection{The $\p_t^2 \bar \p^3$, $\p_t^4 \bar \p^2$, and $\p_t^6\bar \p$ problems}  Since we have provided detailed proofs
of the energy estimates for the two end-point cases of all space derivatives, the $\bar \p^4$ problem, and all time derivatives, the
$\p_t^8$ problem, we have covered all of the estimation strategies for all possible error terms in the three remaining
intermediate problems; meanwhile, the energy contributions for the three intermediate are found in the identical fashion as
for the $\bar \p^4$ and $\p_t^8$ problems.  As such we have the additional estimate
\begin{proposition} \label{additional_2d} For $ \delta >0$ and letting the constant $M_0$ depend on $1/ \delta $,  for $ \alpha =1,2$,
\begin{align*} 
& \sup_{t \in [0,T]} \sum_{a=1}^3 \Bigl[ | \p_t^{2a} \eta ^ \alpha (t)|_{3.5-a}^2 + \| \sqrt{\rho_0} \bar \p^{4-a} \ \p_t^{2a} v(t)\|_0^2
 + \| \rho_0 \bar \p^{4-a} \ \p_t^{2a} D\eta(t)\|^2_0  
  \\
 &\qquad
 + \kappa \int_0^t  \| \rho_0 \bar \p^{4-a} \ \p_t^{2a} Dv(t)\|^2_0 ds
 \Bigr] 
 \le \t M_0 + \delta  \sup_{t \in [0,T]} \t E(t) + C\,T\, P( \sup_{t \in [0,T]} \t E(t))  \,.
\end{align*} 
\end{proposition}

\subsection{Additional elliptic-type estimates for normal derivatives}
Our energy estimates provide  a priori control of horizontal and time derivatives of $\eta$; it remains to gain a priori control of the
normal (or vertical) derivatives of $\eta$.   This is accomplished via a bootstrapping procedure relying on having $\p_t^7 v(t)$ bounded in
$L^2(\Omega)$.

\begin{proposition} \label{pt5v_2d} For $t \in [0,T]$, $\p_t^5 v(t) \in H^1(\Omega)$, $ \rho_0 \p_t^6 J^{-2} (t) \in H^1(\Omega)$ and
$$
\sup_{t \in [0,T]} \left( \| \p_t^5 v(t)\|^2_1 +  \| \rho_0 \p_t^6 J^{-2} (t)\|^2_1 \right) \le 
\t M_0 + \delta   {\sup_{t\in[0,T]}} \t E(t) + C \, T\, P({\sup_{t\in[0,T]}} \t E(t)) \,.
$$
\end{proposition}
\begin{proof}
We begin by taking six time-derivatives of (\ref{approx.a}') to obtain
\begin{align} 
2 \kappa  \p_t^7 \bigl[A^k_i (\rho_0 J^{-1} ),_k\bigr] + 2 \p_t^6 \bigl[A^k_i (\rho_0 J^{-1} ),_k \bigr]= - \p_t^7 v^i \,. \nonumber
\end{align} 
According to Lemma \ref{kelliptic},  and the bound on $\| \p_t^7v(t)\|_0^2 $ given by Corollary \ref{cor2}, 
\begin{equation}\label{tt1}
\sup_{t \in [0,T]} \left\| \p_t^6 \bigl[2A^k_i (\rho_0 J^{-1} ),_k \bigr]\right\|_0^2 \le \t M_0
+ \delta   {\sup_{t\in[0,T]}} \t E(t) + C \, T\, P({\sup_{t\in[0,T]}} \t E(t)) \,.
\end{equation}

For $ \beta =1,2$,
\begin{align} 
2A^k_i (\rho_0 J^{-1} ),_k &= \rho_0 a^k_i J ^{-2},_k + 2 \rho_0,_k a^k_i J^{-2} \nonumber \\
&= \rho_0 a^3_i J ^{-2},_3 + 2 \rho_0,_3 a^3_i J^{-2}  + \rho_0 a^\beta _i J ^{-2},_\beta + 2 \rho_0,_\beta a^\beta _i J^{-2}
\label{eq400_2d}
\end{align} 
  Letting $\p_t^6$ act on equation (\ref{eq400_2d}), we have that
\begin{align*} 
\rho_0a^3_i \p_t^6 J^{-2} ,_3 & + 2 \rho_0,_3 a^3_i \p_t^6 J^{-2}  
= \underbrace{2 \p_t^6 \bigl[A^k_i (\rho_0 J^{-1} ),_k \bigr]}_{ \mathcal{J} _1}- \underbrace{\rho_0\p_t^6 (a^\beta_iJ^{-2} ,_\beta  )}_{ \mathcal{J} _2} -\underbrace{2 \rho_0,_\beta\p_t^6( a^\beta _i J^{-2})}_{ \mathcal{J} _3} \\
& - \underbrace{(\p_t^6 a^3_i)[ \rho_0 J^{-2} ,_3 + 2\rho_0,_3J^{-2} ] }_{ \mathcal{J} _4}
  + \underbrace{ \sum_{a=1}^5 c_a \p_t^a a^3_i \p_t^{6-a} [ \rho_0 J^{-2} ,_3 + 2\rho_0,_3J^{-2} ]  }_{ \mathcal{J} _5} \,.
\end{align*} 

\vspace{.05 in}
\noindent
{\it Bounds for $ \mathcal{J} _1(t)$.}
The inequality (\ref{tt1}) establishes the $L^2( \Omega )$ bound for $ \mathcal{J} _1(t)$. 

\vspace{.05 in}
\noindent
{\it Bounds for $ \mathcal{J} _2(t)$.}
According to Proposition \ref{additional_2d}, 
\begin{equation}\label{tt2}
 \sup_{t \in [0,T]} \left( \| \sqrt{\rho_0} \p_t^6 v(t)\|_0^2 +  \|\rho_0 \bar \p D \p_t^5 v(t)\|^2_0\right) \le \t M_0 + \delta  \sup_{t \in [0,T]} \t E(t) + C\,T\, P( \sup_{t \in [0,T]} 
\t E(t)) \,,
\end{equation} 
so that with (\ref{assump2}), we see that the differentiation identity (\ref{J2}) shows that
$$
 \sup_{t \in [0,T]}  \| \mathcal{J} _2(t)\|^2_0 \le \t M_0 + \delta  \sup_{t \in [0,T]} \t E(t) + C\,T\, P( \sup_{t \in [0,T]} 
\t E(t)) \,.
$$
 
\vspace{.05 in}
\noindent
{\it Bounds for $ \mathcal{J} _3(t)$.}
The  estimate for $\mathcal{J} _3(t)$ follows from the inequality for $\beta=1,2$
\begin{align*} 
\left\| \frac{ \rho_0,_ \beta }{\rho_0}\right\|_{L^\infty ( \Omega )} \le C\left\| \frac{ \rho_0,_ \beta }{\rho_0}\right\|_2 \le C \|\rho_0,_ \beta\|_3 \,,
\end{align*} 
where the first inequality follows from the Sobolev embedding theorem, and the second from our higher-order Hardy
inequality Lemma \ref{Hardy} since $\rho_0,_ \beta \in H^3( \Omega ) \cap \h$ for $\beta=1,2$.

Thus, 
\begin{align} 
\| 2 \rho_0,_\beta\p_t^6( a^\beta _i J^{-2})\|_0^2 & = \| 2 \rho_0\p_t^6( a^\beta _i J^{-2}) \   \frac{ \rho_0,_ \beta }{\rho_0} \|_0^2 \nonumber\\
 & \le \| 2 \rho_0\p_t^6( a^\beta _i J^{-2})\|_0^2  \  \left\| \frac{ \rho_0,_ \beta }{\rho_0}\right\|_{L^\infty ( \Omega )}^2 \nonumber \\
 & \le   \t M_0 + \delta  \sup_{t \in [0,T]} \t E(t) + C\,T\, P( \sup_{t \in [0,T]} 
\t E(t)) \label{ttt1000}
\end{align} 
thanks to (\ref{tt2}) and (\ref{assump2}), and the fact that $\|\rho_0\|_4$ is bounded by assumption,
from which it follows that
$$
 \sup_{t \in [0,T]}  \| \mathcal{J} _3(t)\|^2_0 \le \t M_0 + \delta  \sup_{t \in [0,T]} \t E(t) + C\,T\, P( \sup_{t \in [0,T]} 
\t E(t)) \,.
$$

\vspace{.05 in}
\noindent
{\it Bounds for $ \mathcal{J} _4(t)$.}
The identity (\ref{a3i}) shows that $a^3_i $ is quadratic in $\bar \p\eta$, and in particular, only depends on horizontal
derivatives.    From the estimate (\ref{tt2}) and the weighted embedding (\ref{w-embed}), we may infer that 
$$
 \sup_{t \in [0,T]}  \| \bar \p \p_t^5 v(t) \|^2_0 \le \t M_0 + \delta  \sup_{t \in [0,T]} \t E(t) + C\,T\, P( \sup_{t \in [0,T]} 
\t E(t)) \,.
$$
Thus
$$
 \sup_{t \in [0,T]}  \| \mathcal{J} _4(t)\|^2_0 \le \t M_0 + \delta  \sup_{t \in [0,T]} \t E(t) + C\,T\, P( \sup_{t \in [0,T]} 
\t E(t)) \,.
$$

\vspace{.05 in}
\noindent
{\it Bounds for $ \mathcal{J} _5(t)$.}  Each summand in $ \mathcal{J} _5(t)$ is a lower-order term, such that the time-derivative
of each summand is controlled by the energy function $\t E(t)$; as such, the fundamental theorem of calculus shows that
$$
 \sup_{t \in [0,T]}  \| \mathcal{J} _5(t)\|^2_0 \le \t M_0 + \delta  \sup_{t \in [0,T]} \t E(t) + C\,T\, P( \sup_{t \in [0,T]} 
\t E(t)) \,.
$$

We have therefore shown that
for all $t \in [0,T]$,
$$
 \left\|  \rho_0 a^3_i \p_t^6 J^{-2} ,_3 +2\rho_0,_3 a^3_i \p_t^6J^{-2}  
\right\|^2_0 \le  \t M_0 + \delta  \sup_{t \in [0,T]} \t E(t) + C\,T\, P( \sup_{t \in [0,T]} 
\t E(t)) \,,
$$
and our objective is to infer that  the $L^2( \Omega )$-norm of each summand on the right-hand side is uniformly bounded on $[0,T]$.

To this end, we expand the $L^2( \Omega )$-norm to obtain the inequality
\begin{align} 
& \|  \rho_0 |a^3_ \cdot| \p_t^6 J^{-2} ,_3 (t) \|^2_0
 + 4\| |a^3_ \cdot\rho_0,_3| \, \p_t^6 J^{-2}  (t)\|^2_0  + 4 \int_ \Omega  \rho_0 \rho_0,_3 | a^3_ \cdot |^2 \p_t^6 J^{-2}  \p_t^6J^{-2} ,_3 \, dx 
\nonumber \\
& \qquad \qquad \qquad
\le
\t M_0 + \delta  \sup_{t \in [0,T]} \t E(t) + C\,T\, P( \sup_{t \in [0,T]} \t E(t)) \,. \label{tt3}
\end{align} 
For each $ \kappa >0$,  solutions to our degenerate parabolic approximation (\ref{approx}) have sufficient regularity to ensure that  $\rho_0 |\p_t^6 J^{-2}|^2,_3$ is  integrable.  As such,
we integrate-by-parts with respect to  $x_3$ to find that
\begin{align} 
&4 \int_\Omega \rho_0 \rho_0,_3 | a^3_ \cdot  |^2 \p_t^6J^{-2}  \p_t^6 J^{-2} ,_3 \, dx \nonumber \\
& \qquad\qquad\qquad
=
-2  \left\|  |a^3_ \cdot \rho_0,_3 |\, \p_t^6 J^{-2} (t)\right\|^2_0
-2 \int_\Omega \rho_0 (\rho_0,_3 | a^3_ \cdot |^2 ),_3 (\p_t^6 J^{-2}  )^2 \, dx  \,. \label{tt4}
\end{align} 
Substitution of (\ref{tt4}) into (\ref{tt3}) yields
\begin{align} 
& \| \rho_0|a^3_ \cdot | \p_t^6 J^{-2},_3 (t)\|^2_0 + 2 \| \, |a^3_ \cdot \rho_0,_3 |\,\p_t^6 J^{-2} (t)\|^2_0 \nonumber
\\
& \qquad \qquad \qquad
\le
\t M_0 + \delta   {\sup_{t\in[0,T]}}\t E(t) + C \, T\, P({\sup_{t\in[0,T]}} \t E(t)) 
+C \int_{\Omega} \rho_0 |\p_t^6 J^{-2} |^2 \, dx \,. \label{tt5}
\end{align} 
Using (\ref{assump1}), we see that  $|a^3_ \cdot |^2 $ has a strictly positive lower-bound.  By the physical vacuum condition (\ref{degen}), for $ \epsilon >0$ taken sufficiently small, there are  constants $\theta_1,\theta_2 >0$ such that  $|\rho_0,_3(x)| \ge \theta_1$ whenever $1- \epsilon  \le x_3 \le 1$, and $\rho_0 (x) > \theta_2$ whenever $0\le x\le 1- \epsilon $; hence, by
readjusting the constants  
on the right-hand side of (\ref{tt5}), we
find that
\begin{align} 
& \| \rho_0 \p_t^6 J^{-2},_3 (t)\|^2_0 + 2 \| \p_t^6 J^{-2} (t)\|^2_0 \nonumber
\\
& \qquad \qquad \qquad
\le
\t M_0 + \delta   {\sup_{t\in[0,T]}} \t E(t) + C \, T\, P({\sup_{t\in[0,T]}} \t E(t)) 
+C \int_{\Omega} \rho_0 |\p_t^6 J^{-2} |^2 \, dx \,. \label{tt6}
\end{align} 

By Proposition \ref{additional_2d}, for $ \beta =1,2,$ 
$$ \sup_{t \in [0,T]} \|\rho_0  \p_t^6 J^{-2},_ \beta  (t)\| \le
\t M_0 + \delta   {\sup_{t\in[0,T]}} \t E(t) + C \, T\, P({\sup_{t\in[0,T]}} \t E(t)) \,,
$$
and by the fundamental theorem of calculus and Proposition \ref{prop5_2d},
$$ \sup_{t \in [0,T]} \|\rho_0  \p_t^6 J^{-2} (t)\| \le
\t M_0 + \delta   {\sup_{t\in[0,T]}} \t E(t) + C \, T\, P({\sup_{t\in[0,T]}} \t E(t)) \,.
$$
These two inequality, combined with 
(\ref{tt6}), show that 
\begin{align*} 
& \| \rho_0 \p_t^6 J^{-2}(t)\|^2_1 + \| \p_t^6 J^{-2} (t)\|^2_0
\\
& \qquad \qquad \qquad
\le
\t M_0 + \delta   {\sup_{t\in[0,T]}} \t E(t) + C \, T\, P({\sup_{t\in[0,T]}} \t E(t)) 
+C \int_{\Omega} \rho_0 |\p_t^6 J^{-2} |^2 \, dx \,.
\end{align*} 

We use Young's inequality and the fundamental theorem of calculus (with respect to $t$) for the last integral to find that for $ \theta >0$
\begin{align*} 
C\int_{\Omega^+} \rho_0 \p_t^6 J^{-2} \, \p_t^6 J^{-2} \, dx 
& \le \theta \left\| \p_t^6 J^{-2} (t)\right\|^2_0 + C_ \theta  \left\|  \rho_0 \p_t^6  J^{-2}  (t)\right\|^2_0  \\
& \le \theta \left\| \p_t^6 J^{-2} (t)\right\|^2_0 + C_ \theta  \left\|  \rho_0 \p_t^5 D v  (t)\right\|^2_0  \\
& \le \theta \left\| \p_t^6 J^{-2} (t)\right\|^2_0 + \t M_0 + \delta  \sup_{t \in [0,T]} \t E(t)+  C\,T\, P( \sup_{t \in [0,T]} \t E(t)) \,,
\end{align*} 
where we have used the fact that $ \| \rho_0 \p_t^7 Dv(t)\|^2_0$ is contained in the energy function $\t E(t)$.  We choose
$\theta \ll 1$ and once again
readjust the constants; as a result, we see that on $[0,T]$
\begin{align} 
& \| \rho_0 \p_t^6 J^{-2} (t)\|^2_1
 + \left\| \p_t^6 J^{-2} (t)\right\|^2_0
\le
\t M_0 + \delta   {\sup_{t\in[0,T]}} \t E(t) + C \, T\, P({\sup_{t\in[0,T]}} \t E(t))  \,. \label{eq402_2d}
\end{align} 
With $J_t= a^j_i v^i,_j$, we see that
\begin{equation}\label{tt7}
a^j_i \p_t^5 v^i,_j= \p_t^6 J - v^i,_j \p_t^5 a^j_i - \sum_{a=1}^4 c_a \p_t^a a^j_i \p_t^{5-a} v^i,_j
\end{equation} 
so that using (\ref{eq402_2d}) together with  the fundamental theorem of calculus  for the last two terms on the right-hand side
of (\ref{tt7}), 
we see that
\begin{align} 
\left\| a^j_i \p_t^5v^i,_j (t)\right\|^2_0
\le
\t M_0+ \delta   {\sup_{t\in[0,T]}} \t E(t) + C \, T\, P({\sup_{t\in[0,T]}} \t E(t))  \,, \nonumber
\end{align}
from which it follows that 
\begin{align} 
\left\| \operatorname{div} \p_t^5v (t)\right\|^2_0
\le
\t M_0 + \delta  {\sup_{t\in[0,T]}} \t E(t) + C \, T\, P({\sup_{t\in[0,T]}} \t E(t))  \,. \nonumber
\end{align}
Proposition \ref{curl_est} provides the estimate
$$\| \operatorname{curl} \p_t^5v (t)\|^2_0\le \t M_0   + \delta  {\sup_{t\in[0,T]}} \t E(t) +  C \, T\, P({\sup_{t\in[0,T]}} \t E(t))$$
and Proposition \ref{additional_2d} shows that for $ \alpha =1,2$,
$$|\p_t^5v^ \alpha (t)|^2_{0.5}\le \t M_0   + \delta  {\sup_{t\in[0,T]}} \t E(t) +  C \, T\, P({\sup_{t\in[0,T]}} \t E(t)) \,.$$

We thus conclude from  Proposition \ref{prop1} that 
\begin{align*} 
 \sup_{t \in [0,T]} \left\|  \p_t^5v (t)\right\|^2_1
\le  \t M_0   + \delta  {\sup_{t\in[0,T]}} \t E(t) +  C \, T\, P({\sup_{t\in[0,T]}} \t E(t))  \,.
\end{align*}
\end{proof}

Having a good bound for $\p_t^5v(t)$ in $H^1(\Omega)$ we proceed with our bootstrapping.   
\begin{proposition} \label{pt3v_2d} For $t \in [0,T]$, $ v_{ttt}(t) \in H^2(\Omega)$, $ \rho_0 \p_t^4 J^{-2} (t) \in H^2(\Omega)$ and
$$
\sup_{t \in [0,T]} \left( \|  v_{ttt}(t)\|^2_2 +  \| \rho_0 \p_t^4 J^{-2} (t)\|^2_2 \right) \le 
M_0 + \delta  {\sup_{t\in[0,T]}} E(t) + C \, T\, P({\sup_{t\in[0,T]}} E(t)) \,.
$$
\end{proposition}
\begin{proof} 
We take four time-derivatives of (\ref{approx.a}') to obtain
\begin{align} 
 \kappa  \p_t^5 \bigl[A^k_i (\rho_0 J^{-1} ),_k\bigr] + 2 \p_t^4 \bigl[A^k_i (\rho_0 J^{-1} ),_k \bigr]= - \p_t^5 v^i \,. \nonumber
\end{align} 
According to Lemma \ref{kelliptic},  and the bound on $\| \p_t^5v(t)\|_1^2 $ given by Proposition \ref{pt5v_2d}, 
\begin{equation}\label{ttt1}
\sup_{t \in [0,T]} \left\| \p_t^4 \bigl[2A^k_i (\rho_0 J^{-1} ),_k \bigr]\right\|_1^2 \le \t M_0
+ \delta   {\sup_{t\in[0,T]}} \t E(t) + C \, T\, P({\sup_{t\in[0,T]}} \t E(t)) \,.
\end{equation}

For $ \beta =1,2$,
\begin{align} 
2A^k_i (\rho_0 J^{-1} ),_k &= \rho_0 a^k_i J ^{-2},_k + 2 \rho_0,_k a^k_i J^{-2} \nonumber \\
&= \rho_0 a^3_i J ^{-2},_3 + 2 \rho_0,_3 a^3_i J^{-2}  + \rho_0 a^\beta _i J ^{-2},_\beta + 2 \rho_0,_\beta a^\beta _i J^{-2}
\label{eq600_2d}
\end{align} 
 Letting $\p_t^4$ act on equation (\ref{eq600_2d}), we have that
\begin{align} 
\rho_0a^3_i \p_t^4 J^{-2} ,_3 & + 2 \rho_0,_3 a^3_i \p_t^4 J^{-2}  
= \underbrace{2 \p_t^4 \bigl[A^k_i (\rho_0 J^{-1} ),_k \bigr]}_{ \mathcal{J} _1}- \underbrace{\rho_0\p_t^4 (a^\beta_iJ^{-2} ,_\beta  )}_{ \mathcal{J} _2} -\underbrace{2 \rho_0,_\beta\p_t^4( a^\beta _i J^{-2})}_{ \mathcal{J} _3}\nonumber \\
& - \underbrace{(\p_t^4 a^3_i)[ \rho_0 J^{-2} ,_3 + 2\rho_0,_3J^{-2} ] }_{ \mathcal{J} _4}
  + \underbrace{ \sum_{a=1}^3 c_a \p_t^a a^3_i \p_t^{4-a} [ \rho_0 J^{-2} ,_3 + 2\rho_0,_3J^{-2} ]  }_{ \mathcal{J} _5} \,. \label{ttadd12}
\end{align}

In order to estimate $ \p_t^4 J^{-2} (t)$ in $H^1(\Omega)$, we first estimate horizontal derivatives of  $\p_t^4 J^{-2} (t)$ in $L^2( \Omega )$.   As
such, we consider for $ \alpha =1,2$,
\begin{align} 
\rho_0a^3_i \p_t^4 J^{-2} ,_{3\alpha } +2 \rho_0,_3 a^3_i \p_t^4 J^{-2}  ,_ \alpha 
&= \sum_{l=1}^5 \mathcal{J} _l,_ \alpha 
- (\rho_0 a^3_i),_ \alpha \p_t^4 J^{-2} ,_3 - 2 (\rho_0,_3 a^3_i),_ \alpha \p_t^4 J^{-2} 
\,. \label{ttadd13}
\end{align} 

\vspace{.1 in}
\noindent
{\it Bounds for $ \mathcal{J} _1,_ \alpha$.}  
The estimate (\ref{ttt1}) shows that
$$
\| \mathcal{J} _1,_ \alpha (t)\|_0^2 \le \t M_0 + \delta  {\sup_{t\in[0,T]}} \t E(t) + C \, T\, P({\sup_{t\in[0,T]}}\t E(t)) \,.
$$

\vspace{.1 in}
\noindent
{\it Bounds for $ \mathcal{J} _2,_ \alpha$.} 
 Proposition \ref{additional_2d} provides the estimate
\begin{equation}\label{stz1}
 \sup_{t \in [0,T]}  \left( \| \bar \p^2 v_{ttt}(t)\|_0^2+
 \| \rho_0 \bar \p^2 Dv_{ttt}(t)\|_0^2\right) \le \t M_0 + \delta  {\sup_{t\in[0,T]}} \t E(t) + C \, T\, P({\sup_{t\in[0,T]}}\t E(t)) \,.
\end{equation} 
We write
\begin{align*} 
\mathcal{J} _ 2,_ \alpha & = \rho_0 \p_t^4( a^ \beta _i,_ \alpha J^{-2} ,_ \beta + a^ \beta _i J^{-2} ,_{ \beta \alpha })  + 
\rho_0 , _ \alpha    \p_t^4( a^ \beta _i J^{-2} ,_{ \beta })   \,.
\end{align*} 
Using (\ref{assump2}), for $ \alpha =1,2$, the highest-order term  in $\rho_0 \p_t^4(  a^ \beta _i J^{-2} ,_{ \beta \alpha }) $ satisfies the inequality
\begin{align*} 
\| \rho_0  a^ \beta _i \p_t^4 J^{-2} ,_{ \beta \alpha } \|_0^2  \le C \| \rho_0 \bar \p^2 D v_{ttt}\|_0^2 
\end{align*} 
which has the bound (\ref{stz1}), and the lower-order terms have the same bound using the fundamental theorem of calculus; for
example 
\begin{align*} 
\| \rho_0 J^{-2} ,_{ \beta \alpha } \p_t^4 a^ \beta _i J^{-2} ,_{ \beta \alpha }\|_0^2 & \le  \| \rho_0   J^{-2},_{ \alpha \beta } \|_{L^6(\Omega)} \| \p_t^4 a^ \beta _ i\|_{L^3( \Omega )}  \\
& \le C \| \rho_0   J^{-2},_{ \alpha \beta } \|_1 \| \p_t^4 a^ \beta _ i\|_{0.5} \le \t M_0 \,,
\end{align*} 
where we have used H\"{o}lder's inequality, the Sobolev embedding theorem, and (\ref{assump2}) for the final inequality.   On the
other hand, the term $\rho_0 , _ \alpha    \p_t^4( a^ \beta _i J^{-2} ,_{ \beta })  $ is estimated in the same manner as (\ref{ttt1000}), which shows
that 
$$
\| \mathcal{J} _2,_ \alpha (t)\|_0^2 \le \t M_0 + \delta  {\sup_{t\in[0,T]}} \t E(t) + C \, T\, P({\sup_{t\in[0,T]}}\t E(t)) \,.
$$

 \vspace{.1 in}
\noindent
{\it Bounds for $ \mathcal{J} _3,_ \alpha$.}  Using the fact that $\|\p_t \mathcal{J} _3, \alpha\|_0^2 $ can be bounded by the
energy function, the fundamental theorem of calculus shows that
 $$
\| \mathcal{J} _3,_ \alpha (t)\|_0^2 \le \t M_0 + \delta  {\sup_{t\in[0,T]}} \t E(t) + C \, T\, P({\sup_{t\in[0,T]}}\t E(t)) \,.
$$
 
  \vspace{.1 in}
\noindent
{\it Bounds for $ \mathcal{J} _4,_ \alpha$.}   Again, using the fact that  the vector $a^3_i$ only contains horizontal derivatives of $\eta^i$,
(\ref{assump2}) shows that for $ \alpha =1,2$,
\begin{align*} 
\|\bigl(\p_t^4 a^3_i \ [ \rho_0 J^{-2} ,_3 + 2\rho_0,_3J^{-2} ] \bigr), _\alpha \|_0^2 & \le C \| \bar \p^2 v_{ttt} \|_0^2 + \t M_0 \\
&  \le \t M_0 + \delta  {\sup_{t\in[0,T]}} \t E(t) + C \, T\, P({\sup_{t\in[0,T]}}\t E(t)) \,,
\end{align*} 
 the last inequality following from (\ref{stz1}), and thus
  $$
\| \mathcal{J} _4,_ \alpha (t)\|_0^2 \le \t M_0 + \delta  {\sup_{t\in[0,T]}} \t E(t) + C \, T\, P({\sup_{t\in[0,T]}}\t E(t)) \,.
$$
 
   \vspace{.1 in}
\noindent
{\it Bounds for $ \mathcal{J} _5,_ \alpha$.}  These are lower-order terms, estimated with  the fundamental theorem of calculus and
(\ref{assump2}), yielding
 $$
\| \mathcal{J} _5,_ \alpha (t)\|_0^2 \le \t M_0 + \delta  {\sup_{t\in[0,T]}} \t E(t) + C \, T\, P({\sup_{t\in[0,T]}}\t E(t)) \,.
$$

  \vspace{.1 in}
\noindent
{\it Bounds for $- (\rho_0 a^3_i),_ \alpha \p_t^4 J^{-2} ,_3 - 2 (\rho_0,_3 a^3_i),_ \alpha \p_t^4 J^{-2} $.}   The bounds for
these terms follows in the same fashion as for $ \mathcal{J} _2, _ \alpha $ and show that
 $$
\| (\rho_0 a^3_i),_ \alpha \p_t^4 J^{-2} ,_3 + 2 (\rho_0,_3 a^3_i),_ \alpha \p_t^4 J^{-2}\|_0^2 \le \t M_0 + \delta  {\sup_{t\in[0,T]}} \t E(t) + C \, T\, P({\sup_{t\in[0,T]}}\t E(t)) \,.
$$
We have hence bounded the $L^2( \Omega )$-norm of the
right-hand side of (\ref{ttadd13}) by
 $\t M_0 + \delta  {\sup_{t\in[0,T]}} \t E(t) + C \, T\, P({\sup_{t\in[0,T]}} \t E(t))$.
Using the same integration-by-parts argument   just given above in the proof of Proposition \ref{pt5v_2d}, we  conclude that for $ \alpha =1,2$,
\begin{equation}\label{ssttadd14}
\sup_{t \in [0,T]} \bigl( \| \p_t^4 J^{-2} ,_ \alpha (t) \|^2_0 +  \| \rho_0 \p_t^4 J^{-2},_ \alpha  (t)\|^2_1\bigr) \le 
\t M_0 + \delta  {\sup_{t\in[0,T]}} \t E(t) + C \, T\, P({\sup_{t\in[0,T]}}\t E(t)) \,.
\end{equation}
From the inequality (\ref{ssttadd14}), we may infer that for $ \alpha =1,2$,
\begin{equation}\label{need700}
\sup_{t \in [0,T]}  \| \operatorname{div} v_{ttt} ,_ \alpha (t) \|^2_0  \le 
\t M_0 + \delta  {\sup_{t\in[0,T]}} \t E(t) + C \, T\, P({\sup_{t\in[0,T]}}\t E(t)) \,,
\end{equation} 
and according to Proposition \ref{curl_est}, for $ \alpha =1,2$,
\begin{equation}\label{need701}
\sup_{t \in [0,T]}  \| \operatorname{curl} v_{ttt} ,_ \alpha (t) \|^2_0  \le 
\t M_0 + \delta  {\sup_{t\in[0,T]}} \t E(t) + C \, T\, P({\sup_{t\in[0,T]}}\t E(t)) \,.
\end{equation} 
The boundary regularity of $v_{ttt},_ \alpha $, $ \alpha =1,2$,  follows from Proposition \ref{additional_2d}:
\begin{equation}\label{need702}
\sup_{t \in [0,T]}  | v_{ttt} ,_ \alpha (t) |^2_{0.5}  \le 
\t M_0 + \delta  {\sup_{t\in[0,T]}} \t E(t) + C \, T\, P({\sup_{t\in[0,T]}}\t E(t)) \,.
\end{equation} 
Thus, the inequalities (\ref{need700}), (\ref{need701}), and (\ref{need702}) together with (\ref{hodge}) and (\ref{ssttadd14}) show that
\begin{equation}\label{ttadd14}
\sup_{t \in [0,T]} \bigl( \| v_{ttt},_ \alpha (t) \|^2_1 +  \| \rho_0 \p_t^4 J^{-2},_ \alpha  (t)\|^2_1\bigr) \le 
\t M_0 + \delta  {\sup_{t\in[0,T]}} \t E(t) + C \, T\, P({\sup_{t\in[0,T]}}\t E(t)) \,.
\end{equation}

In order to estimate $\| \p_t ^4 J^{-2} ,_3(t)\|_0^2$, we
next differentiate (\ref{ttadd12}) in the vertical direction $x_3$ to obtain
\begin{align} 
\rho_0a^3_i \p_t^4 J^{-2} ,_{33 } +3 \rho_0,_3 a^3_i \p_t^4 J^{-2}  ,_ 3
&= \sum_{l=1}^5 \mathcal{J} _l,_ 3
- \rho_0 a^3_i,_ 3 \p_t^4 J^{-2} ,_3 - 2 (\rho_0,_3 a^3_i),_ 3 \p_t^4 J^{-2} 
\,. \label{ttadd15}
\end{align} 
Following our estimates for the horizontal derivatives, 
 the newly acquired  inequality (\ref{ttadd14}) together with Propositions \ref{additional_2d} and \ref{pt5v_2d} show that the 
 right-hand side of (\ref{ttadd15}) is 
 bounded  in $L^2(\Omega)$ by $\t M_0 + \delta  {\sup_{t\in[0,T]}} \t E(t) + C \, T\, P({\sup_{t\in[0,T]}} \t E(t))$.
 
 It follows that for $k=1,2,3$,
$$ \| \rho_0a^3_i \p_t^4 J^{-2} ,_{k3} +3 \rho_0,_3 a^3_i \p_t^4 J^{-2}  ,_k\|_0^2 \le \t M_0 + \delta  {\sup_{t\in[0,T]}} \t E(t) + C \, T\, P({\sup_{t\in[0,T]}} \t E(t))\,.$$
Note that the coefficient in front of $\rho_0,_3 a^3_i \p_t^4 J^{-2},_k $ has changed from $2$ to $3$, but the identical integration-by-parts argument 
 that we used in the proof of Proposition \ref{pt5v_2d} is once again employed and  shows that
$$
\| \rho_0 \p_t^4 J^{-2} (t)\|_2^2 + \|\p_t^4 J^{-2} (t)\|_1^2
\le \t M_0 + \delta  {\sup_{t\in[0,T]}} \t E(t) + C \, T\, P({\sup_{t\in[0,T]}} \t E(t))\,.
$$
We can thus infer that
$$
 \|\operatorname{div} v_{ttt} (t)\|_1^2
\le \t M_0 + \delta  {\sup_{t\in[0,T]}} \t E(t) + C \, T\, P({\sup_{t\in[0,T]}} \t E(t))\,.
$$
According to Proposition \ref{curl_est}, 
$\| \operatorname{curl} v_{ttt} (t)\|^2_1\le \t M_0 + \delta  {\sup_{t\in[0,T]}} \t E(t) + C \, T\, P({\sup_{t\in[0,T]}} \t E(t))$ and with the bound on $v_{ttt} ^\alpha  $ given
by  Proposition \ref{additional_2d}, Proposition \ref{prop1}  provides the estimate
\begin{align*} 
\left\|  v_{ttt} (t)\right\|^2_2
\le\t M_0 + \delta  {\sup_{t\in[0,T]}} \t E(t) + C \, T\, P({\sup_{t\in[0,T]}} \t E(t)) \,.
\end{align*}
\end{proof} 

\begin{proposition} \label{pt1v_2d} For $t \in [0,T]$, $ v_{t}(t) \in H^3(\Omega)$, $ \rho_0 \p_t^2 J^{-2} (t) \in H^3(\Omega)$ and
$$
\sup_{t \in [0,T]} \left( \|  v_{t}(t)\|^2_3 +  \| \rho_0 \p_t^2 J^{-2} (t)\|^2_3 \right) \le 
\t M_0 + \delta  {\sup_{t\in[0,T]}} \t E(t) + C \, T\, P({\sup_{t\in[0,T]}} \t E(t)) \,.
$$
\end{proposition}
\begin{proof} 
We take two time-derivatives of (\ref{approx.a}') to obtain
\begin{align} 
 \kappa  \p_t^3 \bigl[A^k_i (\rho_0 J^{-1} ),_k\bigr] + 2 \p_t^2 \bigl[A^k_i (\rho_0 J^{-1} ),_k \bigr]= - \p_t^3 v^i \,. \nonumber
\end{align} 
According to Lemma \ref{kelliptic},  and the bound on $\| \p_t^3v(t)\|_2^2 $ given by Proposition \ref{pt3v_2d}, 
\begin{equation}\label{ssttt1}
\sup_{t \in [0,T]} \left\| \p_t^2 \bigl[2A^k_i (\rho_0 J^{-1} ),_k \bigr]\right\|_2^2 \le \t M_0
+ \delta   {\sup_{t\in[0,T]}} \t E(t) + C \, T\, P({\sup_{t\in[0,T]}} \t E(t)) \,.
\end{equation} 
 Letting $\p_t^2$ act on equation (\ref{eq600_2d}), we have that
\begin{align} 
\rho_0a^3_i \p_t^2 J^{-2} ,_3 & + 2 \rho_0,_3 a^3_i \p_t^2 J^{-2}  
= 2 \p_t^2 \bigl[A^k_i (\rho_0 J^{-1} ),_k \bigr] - \rho_0\p_t^2 (a^\beta_iJ^{-2} ,_\beta  ) - 2 \rho_0,_\beta\p_t^2( a^\beta _i J^{-2}) \nonumber \\
& - (\p_t^2 a^3_i)[ \rho_0 J^{-2} ,_3 + 2\rho_0,_3J^{-2} ] 
  +  c_a \p_t a^3_i \p_t^ [ \rho_0 J^{-2} ,_3 + 2\rho_0,_3J^{-2} ]   \,. \n
\end{align} 
The bound (\ref{ssttt1}) allows us to proceed by using the same
argument that we used in the proof of Proposition \ref{pt3v_2d}, and this leads to   the desired inequality.
\end{proof} 

\begin{proposition} \label{eta_2d} For $t \in [0,T]$, $ \eta (t) \in H^4(\Omega)$, $ \rho_0  J^{-2} (t) \in H^4(\Omega)$ and
$$
\sup_{t \in [0,T]} \left( \|  \eta(t)\|^2_4 +  \| \rho_0  J^{-2} (t)\|^2_4 \right) \le 
M_0 + \delta  {\sup_{t\in[0,T]}} E(t) + C \, T\, P({\sup_{t\in[0,T]}} E(t)) \,.
$$
\end{proposition}
\begin{proof} 
From  (\ref{approx.a}'), we see that
\begin{align} 
 \kappa  \p_t \bigl[A^k_i (\rho_0 J^{-1} ),_k\bigr] + 2  \bigl[A^k_i (\rho_0 J^{-1} ),_k \bigr]= -  v_t^i \,. \nonumber
\end{align} 
According to Lemma \ref{kelliptic},  and the bound on $\| v_t(t)\|_3^2 $ given by Proposition \ref{pt1v_2d}, 
\begin{equation}\label{sssttt1}
\sup_{t \in [0,T]} \left\| 2A^k_i (\rho_0 J^{-1} ),_k \right\|_2^2 \le \t M_0
+ \delta   {\sup_{t\in[0,T]}} \t E(t) + C \, T\, P({\sup_{t\in[0,T]}} \t E(t)) \,.
\end{equation} 
Since
\begin{align} 
\rho_0a^3_i  J^{-2} ,_3 & + 2 \rho_0,_3 a^3_i  J^{-2}  
= 2  \bigl[A^k_i (\rho_0 J^{-1} ),_k \bigr]   \,, \n
\end{align} 
we can use the bound
 (\ref{sssttt1}) and proceed by using the same
argument that we used in the proof of Proposition \ref{pt3v_2d} to conclude the proof.
\end{proof} 

\section{Proof of  Theorem \ref{theorem_main} (The Main Result)}\label{sec10}
 
\subsection{Time of existence and bounds independent of $ \kappa $ and existence of solutions to (\ref{ce0})}\label{subsec_finish1}
 Combining the estimates from Propositions \ref{curl_est}, \ref{prop1_2d}, \ref{prop5_2d}, \ref{additional_2d}, 
 \ref{pt5v_2d}, \ref{pt3v_2d}, \ref{pt1v_2d},  \ref{eta_2d}, and Corollary \ref{cor2} we obtain the following inequality on $(0, T_ \kappa )$:
 $$
 \sup_{t \in [0,T]} \t E(t) \le
 \t M_0+ \delta  {\sup_{t\in[0,T]}} \t E(t) + C \, T\, P({\sup_{t\in[0,T]}} \t E(t)) \,.
 $$
By choosing $ \delta $  sufficiently small, we have that
$$
 \sup_{t \in [0,T]} \t E(t) \le
\t M_0 +  C \, T\, P({\sup_{t\in[0,T]}} \t E(t)) \,.
 $$

Using our continuation argument, presented in Section 9 of \cite{CoSh2006}, this provides us with
a time of existence $T_1$ independent of $\kappa $ and an
estimate on $(0,T_1)$ independent of $\kappa $ of the type:
\begin{equation*}
\sup_{t \in [0,T_1]}  \t E(t) \le 2 \t M_0 \,,
\end{equation*}
as long as the conditions (\ref{assump1}) and (\ref{assump2}) of Subsection \ref{subsec_assumptions} hold. 
These conditions can now be verified by using the fundamental theorem of calculus and further shrinking the
time-interval if necessary.    For example,
 since
$$\|\eta(t)\|_{3.5}\le2\| e\|_{3.5}+ 2\int_0^t \|v(t')\|_{3.5}dt',$$ we see that for
$t$ taken sufficiently small, $\| \eta(t)\|^2_{3.5}\le 2| \Omega| + 1$.
 The other conditions in Subsection \ref{subsec_assumptions} are
satisfied with similar arguments. This leads us to a time of existence $T_2>0$
independent of $\kappa$ for which we have the estimate on $(0,T_2)$
\begin{equation}
\sup_{t \in [0,T_2]}  \t E(t)\le  2\t M_0 \,. \label{eq420}
\end{equation}
In particular, our sequence of solutions $\{\eta^ \kappa \}_{ \kappa >0}$ to our approximate $ \kappa $-problem (\ref{approx}) satisfy the $ \kappa $-independent bound (\ref{eq420}) on
the $ \kappa $-independent time-interval $(0,T_2)$.   

\subsection{The limit as $ \kappa \rightarrow 0$} 
By the $ \kappa $-independent estimate (\ref{eq420}), standard compactness arguments provide the existence of a strongly convergent subsequences for $\epsilon  >0$
\begin{align*} 
\eta^{ \kappa '}  & \to \eta  \text{ in } L^{ 2 }((0,T_2); H^{3 }(\Omega)) \\
v_t^{ \kappa '}  & \to v_t  \text{ in } L^{ 2 }((0,T_2); H^{2 }(\Omega)) \,.
\end{align*}

Consider the variational form of  (\ref{approx.a}): for all
$ \varphi \in  L^2(0,T_2; H^1(\Omega)) $, 
\begin{align*} 
\int_0^{T_2} \Bigl[ \int_\Omega \rho_0 ( v^ {\kappa'} )_t^i \, \varphi^i \, dx  - \int_\Omega \rho_0^2 (J^ {\kappa '})^{-2} (a^{ \kappa'} )^k_i \, \varphi^i,_k \, dx -
\kappa \int_\Omega \rho_0^2 \p_t[(J^{ \kappa'} )^{-2} (a^{ \kappa'} )^k_i] \, \varphi^i,_k \, dx\Bigr] dt  =0 \,.
\end{align*} 
The strong convergence of the sequences $(\eta^{ \kappa '}, v_t^{ \kappa '})$ show that the limit $(\eta, v_t)$  satisfies
 \begin{align*} 
\int_0^{T_2} \Bigl[ \int_\Omega \rho_0  v_t^i \, \varphi^i \, dx  - \int_\Omega \rho_0^2 J^{-2} a^k_i \, \varphi^i,_k \Bigr] \, dx dt  =0 
\end{align*} 
which shows that $ \eta$ is a solution to (\ref{ce0}) on the $ \kappa $-independent time interval $(0,T_2)$.  A standard argument
shows that $v(0) = u_0$ and $\eta(0)=e$.

\subsection{Uniqueness of solutions to the compressible Euler equations (\ref{ce0})}  
Suppose that $(\eta,v)$ and $(\bar \eta, \bar v)$ are both   solutions of (\ref{ce0}) with the same initial
data that satisfies the estimate (\ref{uniquedata}).

Let 
$$
\delta v = v - \bar v \,, \ \ \ \delta \eta = \eta - \bar \eta \, \ \ \delta a = a - \bar a \,, \ \ \delta J^{-2}  = J^{-2}  - \bar J^{-2} \,, \ \text{ etc.}
$$
Then $ \delta v$ satisfies
\begin{alignat*}{2}
\rho_0 \delta v_t^i + \delta a^k_i (\rho _0 ^2 J^{-2} ),_k + \bar a^k_i ( \rho^2 \delta J^{-2} ),_k &=0  &&\text{ in } (0,T] \times \Omega \,, \\
\delta v &=0  &&\text{ on } \{t=0\} \times \Omega \,.
\end{alignat*} 
Consider  the energy function
\begin{align*} 
\mathcal{E} (t) & = \sum_{a=0}^4 \| \p_t^{2a} \delta  \eta(t)\|^2_{4-a} + \sum_{a=0}^3 \Bigl[\| \rho_0 \bar\p^{4-a}\p_t^{2a} D\delta \eta(t)\|^2_{L^2(\omega)}
+\| \sqrt{\rho_0} \bar\p^{4-a}\p_t^{2a}  \delta v(t)\|^2_{L^2(\omega)} \Bigr] \\
& \qquad\qquad+\sum_{a=0}^3 \| \rho_0 \p_t^{2a} \delta J^{-2} (t)\|^2_{4-a}  + \| \rho_0 \p_t^{7} D \delta v(t)\|^2_0 +  \| \rho_0 \p_t^{8}  \delta v(t)\|^2_0 \,.
\end{align*} 

Given the transport-type structure for the curl of $\delta \eta$ and its space and time derivatives, together with the assumed smoothness
of $\eta$ and $\bar \eta$, we can proceed in the same fashion as our estimates in Section \ref{section_mainestimates}, and using that $ \delta v(0)=0$, we
 obtain
$$
\sup_{t \in [0,T]} \mathcal{E} (t) \le C\,T\, P( \sup_{t \in [0,T]} \mathcal{E} (t)) \,,
$$
which shows that $ \delta v (t)=0$ on $[0,T]$.

\subsection{Additional estimates for  $\boldsymbol{ \operatorname{curl} _\eta v}$}
We have constructed solutions for which
$$
\sum_{a=0}^4\Bigl[ \| \p_t^{2a}  \eta(t)\|^2_{4-a} + \|\rho_0 \p_t^{2a} D \eta(t)\|^2_{4-a}
+\| \sqrt{\rho_0} \bar\p^{4-a}\p_t^{2a}  v(t)\|^2_0
$$
remains bounded on $(0,T)$.

 The curl estimates required us to assume that the
initial velocity field $u_0$ satisfies 
$$\| \operatorname{curl} u_0\|_3 < \infty \text{ and } \| \rho_0 \bar \p^4 \operatorname{curl} u_0\|_0 < \infty \,.
$$
Well-posedness requires that the dynamics maintain this regularity for $t \in (0,T)$, and we show that the Lagrangian curl does indeed
maintain the continuity-in-time.

 \begin{corollary} \label{cor_curlv_3d}
$$
\sup_{t \in [0,T]} \bigl( \| \operatorname{curl} _ \eta v(t)\|^2_3 + \| \rho_0 \bar \p^4 \operatorname{curl} _\eta v(t)\|^2_0 \bigr) \le M_0 + \delta \sup_{t \in [0,T]} E(t) + 
C\,T\ P( \sup_{t \in [0,T]} E(t)) \,.
$$\end{corollary}
\begin{proof}
Letting $D^3$ act on the identity (\ref{curlv_3d}) for $ \operatorname{curl} _\eta v$, we see that the highest-order term scales like
$$
D^3 \operatorname{curl} u_0 + \int_0^t D^4 v \, Dv\, A \, A dt' \,.
$$
We integrate by parts to see that the highest-order contribution to $ D^3 \operatorname{curl} _\eta v(t)$ can be written as
$$
D^3 \operatorname{curl} u_0 - \int_0^t D^4 \eta \, [Dv\, A \, A]_t dt'  + D^4\eta(t)\, Dv(t)\, A(t)\, A(t) \,,
$$
which, according to Proposition \ref{eta_2d}, has $L^2(\Omega)$-norm bounded by $$ M_0 + \delta  {\sup_{t\in[0,T]}} E(t) +
C \, T\, P({\sup_{t\in[0,T]}} E(t))\,,$$ after readjusting the constants; thus, the inequality for the $H^3(\Omega)$-norm of $ \operatorname{curl} _\eta
v(t)$ is proved

The same type of analysis works for the weighted estimate.
After integration by parts in time,  the highest-order term in the expression for $\rho_0 \bar \p^4 \operatorname{curl} _\eta v(t)$ scales like
$$
\rho_0 \bar \p^4 \operatorname{curl} u_0 - \int_0^t \rho_0 \bar \p^4 D \eta \, [Dv\, A \, A]_t dt'  + \rho_0 \bar \p^4D\eta(t)\, Dv(t)\, A(t)\, A(t) \,.
$$
Hence, the inequality (\ref{energy1_2d}) shows that the weighted estimate holds as well.
\end{proof}
By taking $ \delta >0$ and $T>0$ small enough, we see that $\sup_{t \in [0,T]} E(t) < M_0$, where $E(t)$ is defined in (\ref{formal}).

\subsection{Optimal regularity for initial data}\label{subsec_optimal}
For the purposes of constructing solutions to our degenerate parabolic $ \kappa $-problem (\ref{approx}), in Section \ref{subsec::initdata},
we smoothed our initial data so that both
our
initial velocity field $u_0^\vartheta$ is smooth, and and our initial density function $\rho_0^\vartheta$  is smooth, positive in the interior, and vanishing
on the boundary $\Gamma$ with the physical vacuum condition (\ref{degen}).   

Our a priori estimates then allows us to pass to the limit  $\lim_{\vartheta \to 0} u_0^\vartheta= u_0$ and $\lim_{\vartheta \to 0} \rho_0^\vartheta
=\rho_0$.  By construction,
 $\rho_0 \in H^4( \Omega )$, satisfies $\rho_0 >0$ in $\Omega$, and the physical vacuum condition (\ref{degen}) near
the boundary $\Gamma$.  Similarly, the initial velocity field need only satisfy $E(0)< \infty $.

 \section{The case of general $\gamma > 1$}   \label{sec::generalgamma}
 
We denote by $a_0$ the integer satisfying the inequality
$$
1 < 1+ {\frac{1}{\gamma -1}}  -a_0 \le 2 \,.
$$
The general higher-order energy function is given by
\begin{align*} 
E_\gamma(t) & =
\sum_{a=0}^4 \| \p_t^{2a}  \eta(t)\|^2_{4-a} + \sum_{a=0}^4 \Bigl[\| \rho_0 \bar\p^{4-a}\p_t^{2a} D\eta(t)\|^2_0
+\| \sqrt{\rho_0} \bar\p^{4-a}\p_t^{2a}  v(t)\|^2_0 \Bigr] \nonumber\\
& \qquad\qquad  
+ \sum_{a=0}^3 \| \rho_0 \p_t^{2a}  J^{-2} (t)\|^2_{4-a} + \| \operatorname{curl} _\eta v(t)\|^2_ 3 
+ \| \rho_0 \bar \p^4 \operatorname{curl} _\eta v(t)\|^2_0 \\
& \qquad\qquad + \sum_{a=0}^{a_0} \| \sqrt{\rho_0}^{1+ {\frac{1}{\gamma-1}} - a} \partial_t^{7+a_0-a} Dv( t )\|_0^2 \,,
\end{align*} 
and we set $M_0^\gamma = P( E_ \gamma (0))$.

Notice the last sum in $E_ \gamma $ appears whenever $ \gamma < 2$, and the number of time-differentiated problems increases
as $\gamma $ approaches $1$.   
We explain this last summation of norms in $E_\gamma$ with a particular example; namely, consider the case that $\gamma= {\frac{3}{2}} $.
Then, $\rho_0 \sim d^2$ near $\Gamma$, $a_0=1$, and the last summation is written as
$$
\sum_{a=0}^{a_0} \| \sqrt{d}^{1+ {\frac{1}{\gamma-1}} - a} \partial_t^{7+a_0-a} Dv( t, \cdot )\|_0^2 =
\| d^ {\frac{3}{2}}  \partial_t^8 Dv(t)\|_0^2 + \| d^ {\frac{1}{2}}  \partial_t^7 Dv(t)\|_0^2
$$
which is equivalent to
\begin{equation}\label{ssextrag}
\int_\Omega   \rho_0^ {\frac{3}{2}} | \partial_t^8 Dv(t)|^2 dx  + \int_\Omega | \rho_0^ {\frac{1}{2}}  |\partial_t^7 Dv(t)|^2 dx \,.
\end{equation} 
The Euler equations with $\gamma= {\frac{3}{2}} $ are written as 
\begin{equation}\label{ce32}
\rho_0 v_t ^i+  a^k_i (\rho_0^ {\frac{3}{2}}  J^ {-\frac{3}{2}} ),_k =0.
\end{equation} 
Energy estimates on the ninth time-differentiated problem produce the first integral in  (\ref{ssextrag}), while the second
integral is obtained using  our elliptic-type estimates on the seventh time-differentiated version of  (\ref{ce32}). (Notice that
the value of $\gamma$ does not play a role in our elliptic-type estimates.)  Having control
on the two integral in (\ref{ssextrag}) then shows that we are back in the situation for the case that $\gamma\ge 2$; namely,
we see that $\partial _t^7 v(t)$ is even better than $L^2(\Omega )$, which allows us to proceed as before.   In particular, for $\gamma < 2$
the power on $\rho_0$ in the first integral in (\ref{ssextrag}) is greater than one, and by weighted embedding estimates, this
means that the embedding occurs into a less regular Sobolev space; this accounts for the need to have more time-differentiated
problems when $\gamma < 2$.

Using this energy function, the same methodology as we used for the case $\gamma=2$, shows
that $\sup_{t \in [0,T]} E_\gamma(t)$ remains bounded for $T>0$ taken sufficiently small.

 \begin{theorem} [Existence and uniqueness for the case $\gamma> 1$] \label{thm_main2}
Suppose that $\rho_0 \in H^4(\Omega)$, $\rho_0(x) >0$ for $x \in \Omega$, and  $\rho_0$ satisfies (\ref{degen}).  Furthermore,
suppose that $u_0$ is given such that  $M_0< \infty $.  Then there exists a  solution to (\ref{ce0}) (and hence (\ref{ceuler})) on $[0,T]$ for $T>0$ taken
 sufficiently small, such that
 $$
 \sup_{t \in [0,T]} E_\gamma(t) \le 2M_0 \,.
 $$
 Moreover if the initial data satisfies
 \begin{align}
 &\sum_{a=0}^5\Bigl[ \| \p_t^{2a}  \eta(0)\|^2_{5-a} + \|\rho_0 \p_t^{2a}  D\eta(0)\|^2_{5-a}
+\| \sqrt{\rho_0} \bar\p^{5-a}\p_t^{2a}  v(0)\|^2_0 \Bigr] \nonumber\\
& \qquad\qquad  
 + \| \operatorname{curl} _\eta v(0)\|^2_ 4 
+ \| \rho_0 \bar \p^5 \operatorname{curl} _\eta v(0)\|^2_0 
+ \sum_{a=0}^{a_0} \| \sqrt{\rho_0}^{1+ {\frac{1}{\gamma-1}} - a} \partial_t^{9+a_0-a} Dv( t )\|_0^2 
  < \infty \,,  \nonumber
 \end{align} 
 then the solution is unique.
\end{theorem}

\vspace{.1 in}

\noindent
{\bf Acknowledgments.}
SS was supported by the National Science Foundation under
grant DMS-0701056, and by the United States Department of Energy through the Idaho National Laboratory's LDRD Project NE-156.

\end{document}